\documentclass[a4paper,11pt]{article}
\usepackage{mattsstyle}
\usepackage{xfrac}
\usepackage{graphicx}
\usepackage{caption}
\usepackage{subcaption}
\usepackage{soul}

\def\TLp#1{\mathrm{TL}^{#1}}
\def\dTLp#1{d_{\TLp{#1}}}
\def\Lp#1{\mathrm{L}^{#1}}

\def\Wp#1{\mathrm{W}^{#1}}
\def\dWp#1{d_{\Wp{#1}}}
\def\Ck#1{\mathrm{C}^{#1}}
\def\Ckc#1{\Ck{#1}_{\mathrm{c}}}
\def\Wkp#1#2{\mathrm{W}^{#1,#2}}

\def\spaceBar{\, | \,}
\def\EnergyS{\cE^{(s)}_\infty}
\def\EnergySneps{\cE_{n,\eps}^{(s)}}
\def\EnergySnepstrunc{\cE_{n,\eps,\text{trunc}}^{(s)}}
\def\EnergySnepsn{\cE_{n,\eps_n}^{(s)}}
\def\EnergySnepsntrunc{\cE_{n,\eps_n,\text{trunc}}^{(s)}}

\def\ceil#1{\left\lfloor{#1}\right\rfloor}

\def\dTorus{d_{\mathrm{man}}} 

\def\addmaths#1{\mathbf{\textcolor{darkgreen}{#1}}}
\def\removemaths#1{\mathbf{\textcolor{darkred}{\hbox{\st{$#1$}}}}}

\allowdisplaybreaks

\def\commentOut#1{}

\title{Consistency of Fractional Graph-Laplacian Regularization in Semi-Supervised Learning with Finite Labels}
\author[1]{Adrien Weihs}
\author[1,2]{Matthew Thorpe}
\date{July 2023}
\affil[1]{School of Mathematics,\protect\\ University of Manchester,\protect\\ Manchester, M13 9PL, UK. \vspace{\baselineskip}}
\affil[2]{The Alan Turing Institute\protect\\ London, NW1 2DB}

\begin{document}

\maketitle

\begin{abstract}
Laplace learning is a popular machine learning algorithm for finding missing labels from a small number of labelled feature vectors using the geometry of a graph.
More precisely, Laplace learning is based on minimising a graph-Dirichlet energy, equivalently a discrete Sobolev $\Wkp{2}{1}$ semi-norm, constrained to taking the values of known labels on a given subset.
The variational problem is asymptotically ill-posed as the number of unlabeled feature vectors goes to infinity for finite given labels due to a lack of regularity in minimisers of the continuum Dirichlet energy in any dimension higher than one.
In particular, continuum minimisers are not continuous.
One solution is to consider higher-order regularisation, which is the analogue of minimising Sobolev $\Wkp{s}{2}$ semi-norms.
In this paper we consider the asymptotics of minimising a graph variant of the Sobolev $\Wkp{s}{2}$ semi-norm with pointwise constraints.
We show that, as expected, one needs $s>d/2$ where $d$ is the dimension of the data manifold.
We also show that there must be an upper bound on the connectivity of the graph; that is, highly connected graphs lead to degenerate behaviour of the minimiser even when $s>d/2$.
\end{abstract}



\section{Introduction} \label{sec:Intro}

A typical question in machine learning is the one of asymptotic consistency.
Suppose that a machine learning algorithm is formulated as a variational problem depending on a parameter $\eps$, on a set of $n$ feature vectors $\Omega_n = \{x_i\}_{i=1}^n\subset \bbR^d$ where we assume that $x_i\iid \mu\in\cP(\Omega)$: the algorithm aims to minimize a discrete objective $\cE_{n,\eps}(u_n)$ over $u_n:\Omega_n \mapsto \bbR^m$.
It is a natural problem to consider the asymptotics of the minimizers of $\cE_{n,\eps}(\cdot)$: in what sense and to what do the latter converge as $n \to \infty$ and $\eps_n = \eps \to 0$?
Furthermore, by taking the appropriate limit of the objectives $\cE_{n,\eps}(\cdot)$, we obtain a limiting continuum objective $\cE_{\infty}(\cdot)$ which is defined for functions $u:\Omega \subseteq \bbR^d \mapsto \bbR^m$: it is also relevant to ask how the minimizers of $\cE_\infty(\cdot)$ relate to the limit of the minimizers of $\cE_{n,\eps}(\cdot)$. Answering these questions allows one to gain rigorous insights in the design of machine learning algorithms.

The specific problem we consider throughout the paper is the regression problem associated to a higher-order variant of Laplace learning. In order to develop intuition, we start by introducing Laplace learning \cite{LapRef}.
Given labels $\{\ell_i\}_{i=1}^N \subset \{0,1\}$ and feature vectors $\{x_i\}_{i=1}^n$, semi-supervised algorithms aim to find the missing labels for feature vectors $\{x_i\}_{i=N+1}^n$.
On an undirected graph $G_{n,\eps} = (\Omega_n,W_{n,\eps})$ with vertices $\Omega_n$ and edge-weight matrix $W_{n,\eps}=(w_{\eps,ij})_{i,j=1}^n$, the Laplace learning method computes the missing labels by solving the variational problem

\begin{align} \label{eq:Intro:LaplaceLearning}
\begin{split}
& u_n \in \argmin_{u_n:\Omega_n\to \bbR} \cE_{n,\eps}(u_n) \quad \text{such that } u_n(x_i) = \ell_i \text{ for } i \leq N \\
& \text{where } \cE_{n,\eps}(u_n) = \frac{1}{2}\sum_{i,j = 1}^n w_{\eps,ij} \l u_n(x_i) - u_n(x_j)\r^2.
\end{split}
\end{align}
The motivation behind this formulation is that vertices $x_i$ and $x_j$ that are close in the graph --- i.e. $w_{\eps,ij}$ is large --- should have similar labels: we impose a discrete regularity requirement on $u_n$.
From the latter equation we note that $u_n(x_i)$ will take values in $\bbR$ whereas, for our purposes of binary classification, we have labels $\ell_i\in\{0,1\}$.
Therefore, in order to classify our vertices, the proposed classification rule for $N < i \leq n$ is:
\begin{equation} \label{eq:Intro:LapThresh}
\hat{\ell}_i = \begin{cases} 0 \quad \text{if } u_n(x_i) < 0.5 \\ 1 \quad \text{else.} \end{cases}
\end{equation}

Laplace learning is defined, via~\eqref{eq:Intro:LaplaceLearning} and~\eqref{eq:Intro:LapThresh}, as a variational problem given the graph $G_{n,\eps}$.
In some applications, rather than being given a graph, one has to define the edge-weight matrix $W_{n,\eps}$.
Since our aim will be to study asymptotic consistency with $n \to \infty$, we need to define $W_{n,\eps}$ in a scalable way.
A natural construction is to use the feature vectors and define $w_{\eps,ij}$ as follows:
\begin{equation} \label{eq:Intro:Weights}
w_{\eps,ij} = \frac{1}{\eps^d}\eta\l\frac{\vert x_i - x_j \vert}{\eps}\r
\end{equation}
for a non-decreasing function $\eta:[0,\infty) \mapsto [0,\infty)$. A typical choice of $\eta$ would be $\eta = \one_{[0,1]}$, the indicator function of the set $[0,1]$: this ensures that whenever $x_i$ and $x_j$ are further apart than $\eps$, we have $w_{\eps,ij} = 0$.
Defining the edge-weights by~\eqref{eq:Intro:Weights} has two advantages. 

On the geometrical side, this allows us to link the intrinsic geometry defined in the graph with the extrinsic Euclidean geometry from the ambient space.
In particular, this will allow for closeness in the graph to be linked to closeness in Euclidean distance and thus, the weighted finite differences $w_{\eps,ij}\vert u(x_i) - u(x_j)\vert^2$ can be approximated by $\vert \nabla u(x_i)\vert^2$ when $u:\Omega\to\bbR$ is sufficiently smooth. 

On the asymptotic side, as we let $n$ tend to infinity, which means that we increase the number of vertices in the graph, it is then natural to let $\eps=\eps_n$ tend to 0 as there is increasingly more local information available at each point which allows one to resolve the geometry in the graph at finer scales. Since the numerical cost often correlates with the number of neighbours (or the density of the matrix $W_{n,\eps}$), scaling $\eps_n\to 0$ also has the advantage of decreasing computation time.
From an analysis point-of-view scaling $\eps_n\to 0$ allows us to replace the discrete objective $\cE_{n,\eps_n}$ based on finite differences with a continuum objective $\cE_{\infty}$ based on derivatives. For example, up to appropriate scaling which is detailed in \eqref{eq:Back:Rel:pLapDisProb}, the continuum energy associated to Laplace learning using weights defined in \eqref{eq:Intro:Weights} is a $\Wkp{1}{2}$ semi-norm. We recall that for $1 \leq k \in \bbN$, $p \geq 1$, $\alpha$ a multi-index and $u \in \Wkp{k}{p}(\Omega)$, the norm and semi-norm of $\Wkp{k}{p}(\Omega)$ can be defined as $\Vert u \Vert_{\Wkp{k}{p}} =  \Vert u \Vert_{\Lp{p}} + \sum_{1 \leq \vert \alpha \vert \leq k } \Vert \partial^\alpha u \Vert_{\Lp{p}}$ and $\sum_{\vert \alpha \vert = k } \Vert \partial^\alpha u \Vert_{\Lp{p}}$ respectively.

While we will rigorously fix notations in Section \ref{subsec:Main:Not}, we now introduce the objective functions that we will be considering. 
Given the graph Laplacian $\Delta_{n,\eps}$ defined in \eqref{eq:Main:Not:Setting:Deltaneps} and the discrete $\Lp{2}(\Omega_n)$ inner product $\langle\cdot,\cdot\rangle_{\Lp{2}(\Omega_n)}$ defined in \eqref{eq:Main:Not:Setting:InnerProduct}, we define the fractional Laplacian energy \cite{Stuart} for $u_n:\Omega_n \mapsto \bbR$ and $s > 0$:
\[
\EnergySneps(u_n) = \langle u_n, \Delta_{n,\eps}^s u_n \rangle_{\Lp{2}(\Omega_n)}.
\]
The associated continuum energy is an approximate $\Wkp{s}{2}$ semi-norm (as $s \in \bbR$, the fractional Sobolev semi-norms are defined differently to the above and we refer to Remarks \ref{rem:Main:Res:Sobolev} and \ref{rem:Main:Res:Approximating} for further discussions) for $u \in \cH^s(\Omega)$ where $\cH^s$ is defined in \eqref{eq:Main:Not:Setting:H}:
\[
\EnergyS(u) = \langle u, \Delta_\rho^s u \rangle_{\Lp{2}(\Omega)},
\]
and $\Delta_\rho$ is the continuum weighted Laplacian defined in \eqref{eq:Main:Not:WeightLap}. For $s = 1$, we recover the discrete and continuum energies of Laplace learning.

The underlying question when studying asymptotic consistency is the one of the appropriate notion of convergence.
For a function $u:\Omega \mapsto \bbR^m$, we define $u\lfloor_{\Omega_n}$ to be the restriction of $u$ onto the sample points $\Omega_n=\{x_1,\hdots,x_n\}$.
The first type of convergence that has been analysed in the literature is concerned with the notion of \emph{pointwise} convergence of energies: for $u:\Omega \mapsto \bbR^m$, one is interested if $\cE_{n,\eps}(u\lfloor_{\Omega_n}) \to \cE_{\infty}(u)$ as $n\to \infty$ and $\eps_n \to 0$ at an appropriate rate.
This has for example been studied in \cite{NIPS2006_5848ad95}, \cite{COIFMAN20065}, ~\cite{Gine}, \cite{10.1007/11503415_32}, \cite{10.1007/11776420_7}, \cite{Singer} and~\cite{10.5555/3104322.3104459}.
Another convergence type is \emph{spectral} convergence, where one analyses the convergence of the eigenpairs of the discrete operator to the continuum one. 
Convergence of eigenvalues has for example been studied in \cite{NIPS2006_5848ad95}, \cite{Trillos} and \cite{CALDER2022123}.
Convergence of eigenvectors is more subtle and we return to this below.

Usually one really wants to establish convergence of minimisers; that is, if $u_n$ is the constrained minimiser of $\cE_{n,\eps_n}^{(s)}$ and $u_\infty$ is the constrained minimiser of $\cE_{\infty}^{(s)}$ then does $u_n$ converge to $u_\infty$?
Since $u_n$ is a discrete function defined on $\Omega_n$ and $u_\infty$ is a continuum function defined on $\Omega$ we need to define a notion of convergence that allows one to compare functions on different domains.
There are two natural ways to do this.
In the first we restrict $u_\infty$ to the domain $\Omega_n$ and compare to $u_n$, i.e. we treat $u_\infty\lfloor_{\Omega_n} - u_n$ which is a function from $\Omega_n\to \bbR^m$.
However, note that we must have sufficient regularity of $u_\infty$ in order to define it's restriction to $\Omega_n$.
There are certain settings where one has enough regularity of the continuum minimiser, e.g. for $p$-Laplacian regularisation, in order to be able to define pointwise evaluation, see for example \cite{Slepcev}, \cite{calder2020rates}, \cite{CALDER2022123}, and~\cite{FLORES202277}. 

When one doesn't have enough regularity to define $u_\infty\lfloor_{\Omega_n}$, we need to consider another approach. 
The second notion of convergence, rather than restricting $u_\infty$ to $\Omega_n$, extends $u_n$ to $\Omega$.
One way to define this extension is via an optimal transport map $T_n:\Omega\to\Omega_n$ between the empirical measure $\mu_n=\frac{1}{n}\sum_{i=1}^n \delta_{x_i}$ and the continuum measure $\mu$.
This defines a partitioning of the state space $\Omega$. 
The extension of $u_n$ to $\Omega$ can then be defined by $\tilde{u}_n=u_n\circ T_n:\Omega\to\bbR^m$.
One can then compare $\tilde{u}_n-u_\infty$.
If one uses an $\Lp{p}$ norm on $\tilde{u}_n-u_\infty$ then this defines the $\TLp{p}$ distance, introduced by~\cite{GARCIATRILLOS2018239} and reviewed in Section~\ref{subsec:Back:TLp}, which can be viewed as a metric in an appropriate space.
In this work it is the latter notion of convergence, i.e. $\TLp{p}$, that we consider.

We point out that, since eigenvectors of $\cE_{n,\eps}^{(s)}$ are discrete functions and eigenvectors of $\cE_{\infty}^{(s)}$ are continuum functions, the convergence of eigenvectors uses one of the two notions of convergence described above.
In particular, uniform convergence (which falls into the first type of convergence) was studied in \cite{10.1214/009053607000000640}, \cite{JMLR:v12:pelletier11a}, \cite{10.1093/imaiai/iaw016}, \cite{https://doi.org/10.48550/arxiv.1510.08110} and \cite{CALDER2022123},  and $\TLp{p}$ (or related) convergence was studied in \cite{GARCIATRILLOS2018239} and \cite{Trillos}. 

Having described the numerous types of convergence at our disposal, we now come back to our question of interest: when do constrained minimizers of $\cE_{n,\eps_n}^{(s)}(\cdot)$ converge to the constrained minimizers of $\cE_{\infty}^{(s)}(\cdot)$?  
Our answer is given in Theorem \ref{thm:Main:Res:ConsFracLap} and is broken into two cases: the \emph{ill-posed} case where we provide conditions on the scaling $\eps_n\to 0$ (depending on $s$ and $\dim(\Omega)$) for constrained minimizers of $\cE_{n,\eps_n}^{(s)}$ to converge to constants, and the \emph{well-posed} case where we provide conditions on the scaling of $\eps_n\to 0$ for constrained minimizers of $\cE_{n,\eps_n}^{(s)}$ to converge to constrained minimizers of $\cE_{\infty}^{(s)}$.
Conditions for the ill-posed case in a closely related setting can be found in \cite{Stuart}.
Our main contribution is to add conditions for the well-posed case.

The paper is set out as follows.
In the next section we present some background material, including a description of the metric space $\TLp{p}$ in which we work, an overview of $\Gamma$-convergence, and a brief description of related works.
The main results are presented in Section~\ref{sec:Main} along with our assumptions and notation.
The proofs are given in Section~\ref{sec:Proofs}.
There is a small gap in the rate at which $\eps_n\to 0$ that is not covered by our theoretical results, and so we provide some numerical experiments in Section~\ref{sec:NumExp} to give insights into what can be expected in this regime.

\section{Background}

In this section we review some background material.
Namely, we include a description of the $\TLp{p}$ topology, a brief overview of $\Gamma$-convergence, and a short review of related works.

\subsection{The \texorpdfstring{$\TLp{p}$}{TLp} Space} \label{subsec:Back:TLp}

Let $\cP(\Omega)$ be the set of probability measures on $\Omega$ and $\cP_p(\Omega)$ be the set of probability measures on $\Omega$ with finite $p$th-moment.
The set of functions $u$ that are measurable with respect to $\mu$ and such that $\int_{\Omega} \vert u(x)\vert^p \, \dd \mu(x) < +\infty$ is denoted by $\Lp{p}(\mu)$ .
The pushforward of a measure $\mu\in\cP(\Omega)$ by a map $T:\Omega\to \cZ$ is the measure $\nu\in\cP(\cZ)$ defined by
\[ \nu(A) = T_{\#} \mu(A) := \mu(T^{-1}(A)) = \mu\l \lb x\spaceBar T(x)\in A \rb\r \qquad \text{for all measurable sets } A. \]
For $\mu,\nu\in\cP_p(\Omega)$ we denote by $\Pi(\mu,\nu)$ the set of all probability measures on $\Omega\times\Omega$ such that the first marginal is $\mu$ and the second marginal is $\nu$, i.e. $(P_X)_{\#}\pi=\mu$ and $(P_Y)_{\#}\pi=\nu$ where $P_X: \Omega\times\Omega \ni (x,y) \mapsto x \in \Omega$ and $P_Y: \Omega\times\Omega \ni (x,y) \mapsto y \in \Omega$.
We start by recalling the definition of the $\TLp{p}$ space and metric from~\cite{Trillos3}.

\begin{mydef}
For an underlying domain $\Omega$, define the set
\[ \TLp{p} = \lb (\mu,u) \spaceBar \mu \in \cP_p(\Omega), u \in \Lp{p}(\mu) \rb. \]
For $(\mu,u),(\nu,v) \in \TLp{p}$, we define the $\TLp{p}$ distance $\dTLp{p}$ as follows:
\[ \dTLp{p}((\mu,u),(\nu,v)) = \inf_{\pi \in \Pi(\mu,\nu)} \int_{\Omega \times \Omega} \vert x-y \vert^p + \vert u(x) - v(y) \vert^p \,\dd\pi(x,y). \]
\end{mydef} 

The key property of the $\TLp{p}$ distance is that $\dTLp{p}((\mu,u),(\nu,v))$ is equal to the $p$-Wasserstein distance between the measures $\mu$ and $\nu$ raised to the graphs of $u$ and $v$, i.e. $\dTLp{p}((\mu,u),(\nu,v))=\dWp{p}((\Id\times u)_{\#}\mu,(\Id\times v)_{\#}\nu)$ where $\dWp{p}$ is the $p$-Wasserstein distance, see~\cite{Trillos3}.
This is useful, since it allows one to leverage the properties of the well-studied Wasserstein distances.
We refer to \cite{villani2009,Santambrogio} for a treatment of optimal transport and Wasserstein distances.

We will say that a sequence $\{(\mu_n,u_n)\}_{n\in\bbN}\subset \TLp{p}$ converges in $\TLp{p}$ to some $(\mu,u)\in\TLp{p}$ if
\[ \dTLp{p}((\mu_n,u_n),(\mu,u))\to 0. \]
As shown in~\cite{Trillos3}, by leveraging the connection between the $\TLp{p}$ distance and the $p$-Wasserstein distance we have the following equivalent notions of convergence. A version of the following proposition can also be stated for non-absolutely continuous measures.

\begin{proposition}
\label{prop:Back:TLp}
Let $(\mu,u) \in \TLp{p}$ where $\mu$ is absolutely continuous with respect to Lebesgue measure and let $\{(\mu_n,u_n)\}_{n=1}^\infty$ be a sequence in $\TLp{p}$. The following are equivalent:
\begin{enumerate}
    \item $(\mu_n,u_n)$ converges to $(\mu,u)$ in $\TLp{p}$;
    \item $\mu_n$ converges weakly to $\mu$ and there exists a sequence of transport maps $\{T_n\}_{n=1}^\infty$ with $(T_{n})_\# \mu = \mu_n$ and $\int_\Omega \vert x - T_n(x) \vert \, \dd \mu(x) \to 0$ such that
    \[
    \int_\Omega \vert u(x) - u(T_n(x))\vert^p \, \dd \mu(x) \to 0;
    \]
    \item $\mu_n$ converges weakly to $\mu$ and for any sequence of transport maps $\{T_n\}_{n=1}^\infty$ with $(T_{n})_\#\mu = \mu_n$ and $\int_\Omega \vert x - T_n(x) \vert \, \dd \mu(x) \to 0$, we have
    \[
    \int_\Omega \vert u(x) - u(T_n(x))\vert^p \, \dd \mu(x) \to 0.
    \]
    
\end{enumerate}

\end{proposition}

In our work, $\mu_n$ will denote the empirical measures of our samples $\{x_i\}_{i=1}^n$ and $\mu$ will be the measure from which the points are sampled. 
Furthermore, $\{u_n\}_{n=1}^\infty$ and $u$  will respectively be the minimizers of our discrete objectives $\cE_{n,\eps_n}(\cdot)$ and continuum objective $\cE_{\infty}(\cdot)$.
Since (in our setting with probability one) $\mu_n\weakstarto \mu$ then with a small abuse of notation we say that $u_n$ converges to $u$ in $\TLp{p}$ if $(\mu_n,u_n)$ converges to $(\mu,u)$ in $\TLp{p}$.

We recall the following result, the proof of which is a simple consequence of~\cite[Theorem 2]{Trillos} (in the special case where $\Omega$ is a torus) with the Borel--Cantelli Lemma, for use in our analysis.

\begin{theorem}
\label{thm:Back:TLp:LinftyMapsRate}
Existence of transport maps.
Assume that $\Omega$ is a torus, $x_i\iid\mu\in\cP(\Omega)$ where $\mu$ has a density that is bounded above and below by positive constants.
Then, there exists a constant $C > 0$ such that $\bbP$-a.s., there exists a sequence of transport maps $\{T_n:\Omega \mapsto \Omega_n \}_{n=1}^\infty$ from $\mu$ to $\mu_n$ such that:
\begin{equation}
\label{eq:Back:TLp:LinftyMapsRate}
\begin{cases}
\limsup_{n \to \infty} \frac{n^{1/2} \Vert \Id - T_n \Vert_{\Lp{\infty}} }{\log(n)^{3/4}} \leq C & \text{if } d = 2; \\
\limsup_{n \to \infty} \frac{n^{1/d} \Vert \Id - T_n \Vert_{\Lp{\infty}} }{\log(n)^{1/d}} \leq C &\text{if } d \geq 3.
\end{cases}
\end{equation}
\end{theorem}

In terms of the assumptions we introduce later, the conditions in the above theorem are given by~\ref{ass:Main:Ass:S1}, \ref{ass:Main:Ass:M1}, \ref{ass:Main:Ass:M2} and~\ref{ass:Main:Ass:D1}.

\subsection{\texorpdfstring{$\Gamma$}{Gamma}-Convergence} \label{subsec:Back:Gamma}

The appropriate framework to describe the convergence of variational problems is $\Gamma$-convergence from the calculus of variations.
We refer to~\cite{gammaConvergence} for a detailed treatment of the subject and only recall the key properties we later need here.
We begin with the definition of $\Gamma$-convergence.

\begin{mydef}
Let $(Z,d_Z)$ be a metric space and $E_n:Z \to \bbR$ a sequence of functionals.
We say that $E_n$ $\Gamma$-converges to $E$ with respect to $d_Z$ if:
\begin{enumerate}
\item For every $z \in Z$ and every sequence $\{z_n\}$ with $d_Z(z_n,z) \to 0$:
\[ \liminf_{n \to \infty} E_n(z_n) \geq E(z); \]
\item For every $z \in Z$, there exists a sequence $\{z_n\}$ with $d_Z(z_n,z) \to 0$ and
\[ \limsup_{n\to \infty} E_n(z_n) \leq E(z). \]
\end{enumerate}
\end{mydef}

The notion of $\Gamma$-convergence allows one to derive the convergence of minimizers from compactness.

\begin{mydef}
We say that a sequence of functionals $E_n:Z \to \bbR$ has the compactness property if the following holds: if $\{n_k\}_{k \in \bbN}$ is an increasing sequence of integers and $\{z_k\}_{k \in \bbN}$ is a bounded sequence in $Z$ for which $\sup_{k\in \bbN} E_{n_k}(z_k) < \infty$, then the closure of $\{z_k\}$ has a convergent subsequence. 
\end{mydef}

\begin{proposition}
\label{prop:Back:Gamma:minimizers}
Convergence of minimizers.
Let $E_n:Z \mapsto [0,\infty]$ be a sequence of functionals which are not identically equal to $\infty$.
Suppose that the functionals satisfy the compactness property and that they $\Gamma$-converge to $E:Z \mapsto [0,\infty]$.
Then
\[ \lim_{n\to \infty} \inf_{z\in Z} E_n(z) = \min_{z \in Z} E(z). \]
Furthermore, the closure of every bounded sequence $\{z_n\}$ for which \begin{equation} \label{eq:Back:Gamma:MinConv}
\lim_{n \to \infty} \left(E_n(z_n) - \inf_{z \in Z} E_n(z) \right) = 0
\end{equation}
has a convergent subsequence and each of its cluster points is a minimizer of $E$.
In particular, if $E$ has a unique minimizer, then any sequence satisfying~\eqref{eq:Back:Gamma:MinConv} converges to the unique minimizer of $E$.
\end{proposition}

In our work, the setup will be as follows: we will show that our discrete objective $\cE_{n,\eps_n}$ $\Gamma$-converges, with respect to the $\TLp{2}$ topology, to the continuum objective $\cE_{\infty}$. 
Then, we will show that the sequence of minimizers of $\cE_{n,\eps_n}$, $\{u_n\}_{n=1}^\infty$, are precompact in $\TLp{2}$ --- its closure has a convergent subsequence.
Using \eqref{eq:Back:Gamma:MinConv}, we will obtain that the minimizers of $\cE_{n,\eps_n}$ converge to the minimizers of $\cE_{\infty}$.

\subsection{Related Works} \label{subsec:back:related}

The framework described in Sections~\ref{subsec:Back:TLp} and~\ref{subsec:Back:Gamma} has been used to show the convergence of minimizers for different objective functions.
Namely, in~\cite{Trillos3} it is shown that the minimizers of the graph total variation
\[ \cE_{n,\eps_n}(u_n) = \frac{1}{\eps_n n^2}\sum_{i=1,j=1}^{n} w_{\eps_n,ij} \vert u_n(x_i) - u_n(x_j) \vert \]
converge to the minimizers of the weighted total variation
\[ \cE_{\infty}(u) = \TV(u,\rho^2) = \sup_{\phi \in \Ckc{\infty}(\Omega,\bbR^d)} \lb \int_\Omega u \Div(\phi) \, \dd x \,:\, \phi\in\Ckc{\infty}(\Omega;\bbR^d), \vert \phi(x) \vert \leq \rho^2(x)\rb \]
where $\rho$ is the Lebesgue density of $\mu \in \cP(\Omega)$ (from which the points $x_i$ are sampled) and $\Ckc{\infty}(\Omega,\bbR^d)$ are smooth functions with compact support from $\Omega$ to $\bbR^d$.

Other examples include various graph cut problems and their continuum counterparts in 
\cite{JMLR:v17:14-490}, \cite{trillos2017estimating}, \cite{thorpeCheeger} and \cite{doi:10.1137/16M1098309}, the Mumford-Shah functional in \cite{Caroccia_2020} and an application in empirical risk minimization \cite{garcia_trillos_murray_2017}.

Not necessarily relying on the $\TLp{p}$-framework, other functionals have been used for clustering and semi-supervised learning. In particular, the Ginzburg-Landau functional which has been extended to graphs in \cite{doi:10.1137/16M1070426} and  $\Gamma$-convergence of (variations of) the latter have been considered in \cite{Gennip}, \cite{cristoferi_thorpe_2020} and \cite{thorpe2019asymptotic}. The eikonal equation on graphs is studied through a viscosity solutions approach in \cite{doi.org/10.48550/arxiv.2202.08789} and \cite{doi.org/10.48550/arxiv.2105.01977} and one finds the same technical tools applied in \cite{Calder_2018} too.

The authors in~\cite{GARCIATRILLOS2018239} extended the spectral convergence results cited in Section \ref{sec:Intro} to include convergence of eigenvectors using the $\TLp{p}$ topology.
They show that eigenvectors of the graph Laplacian converge to eigenfunctions of a weighted continuum Laplacian.
A consequence of this result is that minimizers of
\[ \cE_{n,\eps_n}(u_n) = \frac{1}{\eps^2_n n^2} \sum_{i,j=1}^n w_{\eps_n,ij} (u_n(x_i)-u_n(x_j))^2 \]
converge to the minimizers of
\[ \cE_{\infty}(u) = \begin{cases} \int_\Omega \vert \nabla u(x) \vert^2 \rho(x)^2 \, \dd x & \text{if } u \in \Wkp{1}{2} \\ +\infty & \text{if } u \in \Lp{2} \setminus \Wkp{1}{2}. \end{cases} \]

While the authors in \cite{GARCIATRILLOS2018239} were interested in the machine learning problem of spectral clustering and its consistency, the semi-supervised problem requires pointwise constraints as in \eqref{eq:Intro:LaplaceLearning}. This has led to poor performance of Laplace learning (see \cite{10.5555/2984093.2984243} and \cite{pmlr-v49-elalaoui16}). We briefly elaborate on the intuition stemming from Sobolev embeddings which motivates the use of higher-order variations of Laplace learning to overcome these issues. As was shown in \cite{GARCIATRILLOS2018239} and noted above, minimizers of Laplace learning converge to minimizers of a $\Wkp{1}{2}$ semi-norm. If we now constrain our minimizer $u \in \Wkp{1}{2}$ to take certain values $\ell_i$ at $x_i$ for $i \leq N$, we need $u$ to be at least continuous which is satisfied only for $d = 1$. In most applications, we have $d > 1$ and consequently, with a large data set --- which theoretically resembles the large data limit setting --- the minimizer of Laplace learning will not be regular and, in fact, be an almost constant function developing spikes at the imposed labels (see for example \cite[Figure 1]{98b487bb64994720ba648f45328e2135}). Since our use of the graph structure relies on the idea that connected points in the graph should have similar labels, it is clear that an irregular function cannot be used for the purpose of labelling the rest of the dataset and justifies the need for other learning algorithms. In particular, one approach to solve this problem is to consider methods for which the continuum limit describes a higher-order Sobolev semi-norm since the general Sobolev embeddings guarantee that $u \in \Wkp{s}{p}$ is continuous whenever $sp > d$ allowing for greater flexibility than Laplace learning.

Fixing $s = 1$ and considering $p > 1$ leads to $p$-Laplacian regularization defined as
\begin{equation} \label{eq:Back:Rel:pLapDisProb}
\cE_{n,\eps_n}(u_n) = \frac{1}{\eps^p_n n^2} \sum_{i=1,j=1}^n w_{\eps_n,ij} \vert u_n(x_i) - u_n(x_j) \vert^p \quad \text{subject to } u_n(x_i) = \ell_i \text{ for } i \leq N.
\end{equation}
In~\cite{Slepcev}, it is shown that minimizers of \eqref{eq:Back:Rel:pLapDisProb} converge to the minimizers of either
\begin{equation} \label{eq:Back:Rel:pLapContNoCons}
\cE_{\infty}(u) = \begin{cases} \int_\Omega \vert \nabla u \vert^p \rho^2(x) \,\dd x & \text{if } u \in \Wkp{1}{p}, \\ +\infty & \text{else} \end{cases}
\end{equation}
or
\begin{equation} \label{eq:Back:Rel:pLapContCons}
\cE_{\infty}^{(\mathrm{con})}(u) = \begin{cases} \int_\Omega \vert \nabla u \vert^p \rho^2(x) \, \dd x & \text{if } u \in \Wkp{1}{p} \text{ and } u(x_i) = \ell_i \text{ for } i \leq N, \\ +\infty & \text{else.} \end{cases}
\end{equation} 
The parameter that controls the convergence to either functional is $\eps_n$ (intuitively one might expect that $p>d$ is enough for the constrained minimizers of~\eqref{eq:Back:Rel:pLapDisProb} to converge to minimizers of~\eqref{eq:Back:Rel:pLapContCons} and whilst this is necessary it turns out not to be sufficient~\cite{Slepcev}): indeed, as we will present later on, $\eps_n$ has to tend to $0$ but the rate at which it does so is crucial in the analysis. 

Taking $p \to \infty$ in $p$-Laplacian regularization, one can consider the $\infty$-Laplacian and the associated Lipschitz learning problem \cite{pmlr-v40-Kyng15}. Various convergence results related to the latter are presented in \cite{doi:10.1137/18M1199241}, \cite{Bungert} and \cite{doi.org/10.48550/arxiv.2111.12370}. It is also worth noting that if one decides to fix a constant localization parameter $\eps_n = \eps$ in $p$-Laplace regularization, then we are in the setting of non-local $p$-Laplacian regularization and convergence results related to the limit of \eqref{eq:Back:Rel:pLapDisProb} are discussed in \cite{doi:10.1137/17M1123596}, \cite{refId0} and \cite{https://doi.org/10.48550/arxiv.2010.08697}.

Another higher-order variant of Laplace learning that has been introduced in \cite{pmlr-v15-zhou11b} and studied in \cite{Stuart} is fractional Laplacian regularization (see \eqref{eq:Main:Not:Setting:DisEnergy} for the discrete energy and \eqref{eq:Main:Not:DisMin} for the associated semi-supervised learning problem). In this case, the idea is to fix $p = 2$ and to let $s$ vary. This leads to the continuum problem introduced in \eqref{eq:Main:Not:ContEnergy} which can be shown to be related to a $\Wkp{s}{2}$ semi-norm \cite[Lemma 17]{Stuart}. Similar to the convergence result in \cite{GARCIATRILLOS2018239} for Laplace learning, it is shown in \cite{Stuart} that minimizers of \eqref{eq:Main:Not:Setting:DisEnergy} converge to minimizers of \eqref{eq:Main:Not:ContEnergy}. It is precisely this setting that we aim to extend to semi-supervised learning by adding pointwise constraints. The results of our analysis are stated in Theorem \ref{thm:Main:Res:ConsFracLap}.

The advantage of fractional Laplace regularization over $p$-Laplacian regularization comes from the formulation of the discrete energy \eqref{eq:Main:Not:Setting:DisEnergy} in terms of the eigenpair decomposition of the graph Laplacian as detailed in \eqref{eq:Main:Not:Setting:DisEnergyDecomposition}. The latter allows one to use Lagrange multipliers in order to find an exact discrete minimizer to the semi-supervised problem while one has to rely on gradient descent to compute an approximate solution to $p$-Laplacian regularization as described in \cite{Slepcev}. We refer to Remark \ref{rem:Main:Res:Approximating} for a further discussion on higher-order variations of Laplace learning.

Lastly, we point out other approaches that have been proposed to deal with the regularity problems of Laplace learning in semi-supervised learning. In particular, we refer to game theoretic $p$-Laplacian regularization \cite{Calder_2018} and Poisson learning \cite{98b487bb64994720ba648f45328e2135}.
One can also study estimators which are required to lie in a space spanned by a finite number of eigenvectors of the graph Laplacian as is done in \cite{10.5555/2968618.2968737} and \cite{sslManifolds} or 
properly re-weight the Laplacian as in~\cite{wnll} and~\cite{calderSlepcev}.

\section{Main Results} \label{sec:Main}

In this section we present our main results.
We start by defining the notation we will use, then we state our assumptions followed by the main results.
The proofs of the main results are in the following section.

\subsection{Notation and Setting} \label{subsec:Main:Not}

We now want to formalize the notation used throughout the paper and describe our precise setting. 

\subsubsection{Graph Setting}

Given a space $\Omega\subset\bbR^d$, a measure $\mu$ on $\Omega$ that has density $\rho$, iid samples $\Omega_n = \{x_i\}_{i=1}^n$ from $\mu$, a weight function $\eta$ and a length-scale $\eps$, we will define a random geometric graph as follows.
For $\eps>0$, we write $\eta_\eps(\cdot) = \eps^{-d}\eta(\cdot/\eps)$.
We define a graph $G_{n,\eps} = (\Omega_n,W_{n,\eps})$ where $\Omega_n$ are the vertices and $W_{n,\eps}=(w_{\eps,ij})_{i,j=1}^n$ is the edge weight matrix with entries $w_{\eps,ij} = \eta_\eps(\vert x_i - x_j\vert)$. 

Let $D_{n,\eps}$ be the diagonal matrix with entries $d_{n,\eps,ii} = \sum_{j = 1}^n w_{\eps,ij}$ and define
\[ \sigma_\eta = \frac{1}{d}\int_{\bbR^d} \eta(\vert h \vert) \vert h \vert^2 \,\dd h < \infty. \]
The graph Laplacian is defined as
\begin{equation} \label{eq:Main:Not:Setting:Deltaneps}
\Delta_{n,\eps} := \frac{2}{\sigma_\eta n\eps^2} (D_{n,\eps} - W_{n,\eps}).
\end{equation}
With a small abuse of notation we can interpret $\Delta_{n,\eps}$ as a matrix $\Delta_{n,\eps}\in\bbR^{n\times n}$ or as an operator $\Delta_{n,\eps}:\Lp{2}(\mu_n)\to\Lp{2}(\mu_n)$ where $\mu_n=\frac{1}{n}\sum_{i=1}^n \delta_{x_i}$ is the empirical measure.

\subsubsection{Variational Problems}

Given functions $u_n,v_n: \Omega_n \to \bbR$, we define the $\Lp{2}(\mu_n)$ inner product: 
\begin{equation} \label{eq:Main:Not:Setting:InnerProduct}
\langle u_n, v_n \rangle_{\Lp{2}(\mu_n)} = \frac{1}{n} \sum_{i=1}^n u_n(x_i)v_n(x_i).
\end{equation} 
Such functions can be considered vectors in $\bbR^n$ and with an abuse of notation we will understand $u_n$ as both a function $u_n:\Omega_n\to \bbR$ and a vector $\bbR^n$.

We denote the eigenpairs of $\Delta_{n,\eps}$ by $\{(\lambda_{n,\eps,k},\psi_{n,\eps,k})\}_{k=1}^n$ where $\lambda_{n,\eps,k}$ are in increasing order, $0 = \lambda_{n,\eps,1} < \lambda_{n,\eps,2} \leq \lambda_{n,\eps,3} \leq \hdots \leq \lambda_{n,\eps,n}$, (where strict inequality between $\lambda_{n,\eps,1}$ and $\lambda_{n,\eps,2}$ follows when the graph $G_{n,\eps}$ is connected) and note that $\{\psi_{n,\eps,k}\}_{k=1}^n$ form for a basis of $\Lp{2}(\mu_n)$.

For $s>0$, we will be considering the following energies on the graph:
\begin{equation} \label{eq:Main:Not:Setting:DisEnergy}
\EnergySneps(u_n) = \langle u_n, \Delta_{n,\eps}^s u_n\rangle_{\Lp{2}(\mu_n)}.
\end{equation}
Note that with the eigenfunction decomposition, we can write
\begin{equation}\label{eq:Main:Not:Setting:DisEnergyDecomposition}
\EnergySneps(u_n) = \sum_{k=1}^n \lambda_{n,\eps,k}^s \langle u_n, \psi_{n,\eps,k}\rangle_{\Lp{2}(\mu_n)}^2.
\end{equation} 
\commentOut{
This formulation also allows us to define a truncated version of the energy:
\[
\EnergySnepstrunc(u_n) = \sum_{k=1}^{\lfloor K_n \rfloor - 1} \lambda_{n,\eps,k}^s \langle u_n, \psi_{n,\eps,k}\rangle_{\Lp{2}(\mu_n)}^2
\]
where the definition of $K_n$ will be made precise in Proposition \ref{prop:Proofs:Back:DisReg:EVecAlt}. The truncated energy permits to circumvent many of the difficulties arising from the inclusion of the tail of the eigenvalues: indeed, we investigate the control of $\{\lambda_{n,l}\}_{l=1}^{\lfloor K_n \rfloor}$ in Lemma \ref{lem:Proofs:Back:EvalTails:Jn}, Remark \ref{rem:Proofs:Back:EvalTails:eigenvalueBounds} and Proposition \ref{prop:Proofs:Back:DisReg:EVecAlt}.
}

Suppose now that we are given $\{\ell_i\}_{i=1}^N \subseteq \bbR$ with $N < n$ labels for the first $N$ samples $\{x_i\}_{i=1}^N$.
The problem we consider is to find a function $u_n$ defined as:
\begin{equation} \label{eq:Main:Not:DisMin}
u_n \in \argmin_{v_n \in \Lp{2}(\mu_n)} \EnergySnepsn(v_n) \quad \text{such that } v_n(x_i) = \ell_i \text{ for } i \leq N.
\end{equation}
Using Lagrange multipliers, one can show the well-posedness of \eqref{eq:Main:Not:DisMin}.

This discrete problem has a continuum analogue.
Namely, let $\Delta_\rho$ be the continuum weighted Laplacian operator defined by
\begin{equation} \label{eq:Main:Not:WeightLap}
\Delta_\rho u(x) = -\frac{1}{\rho(x)}\Div(\rho^2\nabla u)(x), \, x \in \Omega \quad \quad \quad \frac{\partial u}{\partial n} = 0, \, x \in \partial \Omega.
\end{equation}
and let $\{(\lambda_k,\psi_k)\}_{k=1}^\infty$ be its associated eigenpairs where $\lambda_1 = 0 < \lambda_2 \leq \lambda_3 \leq \hdots$ and note that $\{\psi_k\}_{k=1}^\infty$ form a basis of $\Lp{2}(\mu)$.
The continuum energy is then defined as
\[
\EnergyS(u) = \langle u, \Delta^s_\rho u \rangle_{\Lp{2}(\mu)}
\]
or, using the eigenfunction decomposition, 
\begin{equation} \label{eq:Main:Not:ContEnergy}
\EnergyS(u) = \sum_{k=1}^\infty \lambda_k^s \langle u, \psi_k \rangle_{\Lp{2}(\mu)}^2.
\end{equation}
We look at the associated problem, namely to find $u$:
\begin{equation} \label{eq:Main:Not:ContMin}
u \in \argmin_{v\in\Lp{2}(\mu_n)} \EnergyS(v) \quad \text{such that } v(x_i) = \ell_i \text{ for } i \leq N.
\end{equation}
Well-posedness of \eqref{eq:Main:Not:ContMin} is derived from Theorem \ref{thm:Main:Res:ConsFracLap} and the convexity of $\EnergyS(\cdot)$.

We lastly define the energies for which we will prove $\Gamma$-convergence.
We consider the set
\begin{equation} \label{eq:Main:Not:Setting:H}
\cH^s(\Omega) = \{ h \in \Lp{2}(\mu) \spaceBar \EnergyS(h) < +\infty \}. 
\end{equation}
As pointed out in Section \ref{subsec:back:related}, $\cH^s(\Omega)$ is closely related to the Sobolev space $\Wkp{s}{2}(\Omega)$.

Consider for $(\nu,v) \in \TLp{2}$:
\begin{equation} \label{eq:Main:Not:Fn}
\cF_{n,\eps_n}( (\nu,v) ) = \begin{cases} \EnergySnepsn(v) &\text{if } \nu = \mu_n \text{ and for } i\leq N, v(x_i) = \ell_i \\ +\infty & \text{else.} \end{cases} 
\end{equation}

\commentOut{and
\begin{equation} \label{eq:Main:Not:FnTrunc}
\cF_{n,\eps_n,\text{trunc}}( (\nu,v) ) = \begin{cases} \EnergySnepsntrunc(v) &\text{if } \nu = \mu_n \text{ and for } i\leq N, v(x_i) = \ell_i \\ +\infty & \text{else.} \end{cases} 
\end{equation}
}

The continuum-limit of this energy is either
\begin{equation} \label{eq:Main:Not:F}
\cF((\nu,v)) = \begin{cases} \EnergyS(v) & \text{if } \nu = \mu, v \in \cH^s(\Omega) \text{ and for } i\leq N, v(x_i) = \ell_i \\ +\infty & \text{else} \end{cases}
\end{equation}
or
\[ \cG((\nu,v)) = \begin{cases} \EnergyS(v) & \text{if } \nu = \mu \text{ and } v \in \cH^s(\Omega), \\ +\infty & \text{else} \end{cases} \]
depending on the asymptotic behaviour of $\eps_n$ which we detail in Theorem \ref{thm:Main:Res:ConsFracLap}.

\subsubsection{Probability setting}

In the sequel, some of our results will depend on the data set $\{x_i\}_{i=1}^\infty$. For example, given $n$, a length-scale $\eps_n$ and a weight construction as in \eqref{eq:Intro:Weights}, as we will point out in Section \ref{subsec:main:assumputions}, the graph is only connected for some subset of all sequences $\{x_i\}_{i=1}^n$. Since we have the modeling assumption $x_i \iid \mu $, we can characterize how large this subset of sequences is in a probabilistic manner. In particular, we can create a probability space $(\Psi,\cF,\bbP)$ in which each element is a sequence $\{x_i\}_{i=1}^\infty$ and our aim will be to show that there exists $\Psi^\prime\subseteq \Psi$ with $\bbP(\Psi') = 1$ such that some property holds for all sequences in $\Psi'$.

We now detail the construction of $(\Psi,\cF,\bbP)$. Consider the probability space $(\Omega, \mathcal{B}(\Omega), \mu)$ where $\cB(\Omega)$ is the Borel $\sigma$-algebra on $\Omega$. Let $\Psi = \Omega^{\mathbb{N}}$.
For any set $J \subseteq \mathbb{N}$ with $ 0 < \vert J \vert < \infty$, we define the coordinate maps $X_J:\Psi \mapsto \Omega^{\vert J \vert}$ by 
\[
X_J(Y) = (\psi_{J_1},\hdots,\psi_{J_{\vert J \vert}})
\]
where $J = \{J_1,\hdots,J_{\vert J \vert}\}$ and $\psi \in \Psi$.

By classical measure theoretical arguments, for any set $J \subseteq \mathbb{N}$ with $\vert J \vert < \infty$ and an appropriate product $\sigma$-field $\mathcal{B}^{\vert J \vert}(\Omega)$ -- the smallest $\sigma$-field that makes the coordinate maps $\{X_{\{j\}}\}_{j \in J}$ measurable, we can define a probability space $(\Omega^{\vert J \vert}, \mathcal{B}^{\vert J \vert }(\Omega),\mathbb{P}^J)$ where $\mathbb{P}^J$ is the product measure satisfying $\mathbb{P}^J(A_{J_1} \times \hdots \times A_{J_{\vert J \vert}}) = \mu(A_{J_1}) \times \hdots \times \mu(A_{J_{\vert J \vert}})$ for $A_{J_i} \in \mathcal{B}(\Omega)$.

By an application of the Daniell--Kolmogorov Theorem \cite[Theorem 2.4.3]{tao2011introduction} (and because $(\Omega,\mathbb{B}(\Omega))$ is Polish), there exists a measure space $(\Psi,\cF,\bbP)$ such that
$\mathbb{P}(X_J^{-1}(A)) = \mathbb{P}^J(A)$ for all $A \in \mathcal{B}^{\vert J \vert}(\Omega)$. It is also straight-forward to show that the coordinate maps $\{X_{\{i\}}: \Psi \mapsto \Omega \}_{i=1}^\infty$ are independent and have distribution $\mu$. For any $\{x_i\}_{i=1}^\infty$, there exists $\psi \in \Psi$ such that $x_i = X_{\{i\}}(\psi)$ for all $i \in \bbN$.

\subsection{Assumptions} \label{subsec:main:assumputions}

We start by listing the assumptions on the space $\Omega$.

\begin{assumptions}
We make the following assumption on the space.

\begin{enumerate}[label=\textbf{S.\arabic*}]
\item The feature vector space $\Omega$ is the unit torus $\sfrac{\bbR^d}{\bbZ^d}$. We write $\dTorus$ for the manifold distance. \label{ass:Main:Ass:S1}
\end{enumerate}
\end{assumptions}

The assumption that $\Omega$ is the torus greatly simplifies our analysis in several places.
Firstly, we don't have to analyse the boundary as we would in a bounded Euclidean domain (the pointwise rates of convergence in the graph-to-continuum Laplacian are different close to the boundary compared with the interior).
And secondly, we don't have to take into account curvature, as we would on a manifold. However, we do not expect significant changes in the proofs if the latter were to be considered.

We make the following assumptions on the measure.

\begin{assumptions}
Assumptions on the measure.
\begin{enumerate}[label=\textbf{M.\arabic*}]
\item The measure $\mu$ is a probability measure on $\Omega$. \label{ass:Main:Ass:M1} 
\item There is a strictly positive Lipschitz continuous Lebesgue density $\rho$ of $\mu$. \label{ass:Main:Ass:M2}
\end{enumerate}
\end{assumptions}

Since the density $\rho$ is strictly positive and continuous on a compact domain then we can infer that $\rho$ is bounded from above and below by strictly positive constants, i.e.
\[ 0< \min_{x\in\Omega} \rho(x) \leq \max_{x\in\Omega} \rho(x) < +\infty. \]
We also have that $\mu$ has finite $p$th moment.

The data consists of feature vectors $\{x_i\}_{i=1}^n$ and labels $\{\ell_i\}_{i=1}^N$ and we make the following assumptions.

\begin{assumptions}
Assumptions on the data.
\begin{enumerate}[label=\textbf{D.\arabic*}]
\item Feature vectors $\Omega_n = \{x_i\}_{i=1}^n$ are iid samples from a measure $\mu$ satisfying \ref{ass:Main:Ass:M1}.
We denote by $\mu_n$ the empirical measure associated to our samples. \label{ass:Main:Ass:D1}
\item There are $N$ labels $\{\ell_i\}_{i=1}^N\subset \bbR$ corresponding to the first $N$ feature vectors $\{x_i\}_{i=1}^N$. \label{ass:Main:Ass:D2}
\end{enumerate}
\end{assumptions}

It is straightforward to generalise to the case where labels are multidimensional, i.e. $\ell_i\in\bbR^m$, with only an additional notational burden on the presentation.

The weight function $\eta$ is assumed to satisfy the following assumptions.

\begin{assumptions}
Assumptions on the weight function or kernel.
\begin{enumerate}[label=\textbf{W.\arabic*}]
\item The function $\eta:[0,\infty) \to [0,\infty)$ is decreasing, with compact support, Lipschitz continuous on $[0,1]$, and has $\eta(0.5) > 0.5$ and $\eta(1) = 0$. 
\label{ass:Main:Ass:W1}
\item The function $\eta$ integrates to unity over $\bbR^d$, i.e. $\int_{\bbR^d} \eta(\vert x \vert) \,\dd x = 1$. \label{ass:Main:Ass:W2}
\end{enumerate}
\end{assumptions}

The assumption that $\eta$ is Lipschitz continuous on $[0,1]$, $\eta(1/2)>0$, and integrates to unity (Assumption~\ref{ass:Main:Ass:W2}) are not important and could be replaced by the assumption that $\eta$ is Lipschitz continuous on $[0,c]$ for some $c>0$ and $\eta(0)>0$.
Fixing $c=1$, and assuming $\eta(1/2) > 0$ along with~\ref{ass:Main:Ass:W2} simplifies the presentation.
We also note that Assumption~\ref{ass:Main:Ass:W1} implies that
\[ \int_0^\infty \eta(r) r^{d+1} \,\dd r < +\infty. \]
The assumption that $\eta$ is Lipschitz continuous in a closed interval around 0 is slightly stronger than what is usually assumed; we include the assumption here so as to be able to apply the results from~\cite{Trillos}.

We finally have the assumptions on the length scale $\eps=\eps_n$ which we scale with the number of feature vectors.

\begin{assumptions}
Assumptions on the length-scale.
\begin{enumerate}[label=\textbf{L.\arabic*}]
\item The length scale $\eps=\eps_n$ is positive and converges to 0, i.e. $0<\eps_n \to 0$. \label{ass:Main:Ass:L1}
\item The length scale $\eps=\eps_n$ satisfies either the lower bound (in the ill-posed case)
\begin{equation} \label{eq:Main:Ass:epsLBIllPosed} \tag{\ref{ass:Main:Ass:L2}.I}
\begin{split} 
\lim_{n \to \infty} \frac{\log(n)}{n \eps_n^{d}} & = 0 \qquad \text{if } d\geq 3 \\
\lim_{n \to \infty} \frac{(\log(n))^{3/2}}{n \eps_n^2} & = 0 \qquad \text{if } d=2
\end{split}
\end{equation}
or the lower bound (in the well-posed case)
\begin{equation} \label{eq:Main:Ass:epsLBWellPosed} \tag{\ref{ass:Main:Ass:L2}.W}
\lim_{n \to \infty} \frac{\log(n)}{n \eps_n^{d+4}} = 0.
\end{equation}
\label{ass:Main:Ass:L2}
\end{enumerate}
\end{assumptions}

Assumption~\ref{ass:Main:Ass:L2} guarantees that (with probability one) that there exists $N_1$ such that for all $n\geq N_1$ the graph $G_{n,\eps_n}=(\Omega_n,W_{n,\eps_n})$ is connected (see \cite{goel} or \cite{DBLP:books/ox/P2003}).

\subsection{Main Results}

The condition on $\eps_n$ stated in \ref{ass:Main:Ass:L2} is a lower bound on $\eps_n$. 
Our results are mostly concerned with finding an upper bound on $\eps_n$, just as is found in~\cite[Theorem 2.1]{Slepcev} for the $p$-Laplacian.
Indeed, it can be shown that if $\eps_n$ goes to $0$ too slowly, the minimizers of~\eqref{eq:Back:Rel:pLapDisProb} converge to a minimizer of~\eqref{eq:Back:Rel:pLapContNoCons}.
Minimizers of~\eqref{eq:Back:Rel:pLapContNoCons} are simply the constant functions and, in particular, do not consider the information from the labels $\{\ell_i\}_{i=1}^N$.
We can therefore consider this regime as degenerate or asymptotically ill-posed.
Conversely, if we have that $\eps_n\to 0$ sufficiently quickly (which for the $p$-Laplacian means $n\eps_n^p \to 0$), then the minimizers of \eqref{eq:Back:Rel:pLapDisProb} converge to the minimizers of \eqref{eq:Back:Rel:pLapContCons}: the latter takes into account the labelling information.

\begin{remark} \label{rem:Main:Res:LowerBoundS}
Lower bounds on $s$.
Suppose that we are able to find an upper bound on $\eps_n$ of the form $n\eps_n^{h(s)} \leq C$ for some $0 \neq h:\bbR \mapsto \bbR$ and $C >0$. Combined with a lower bound of the form $\lim_{n\to \infty} \log(n)n^{-1}\eps_n^{-g(d)} = 0$ for some function $g:\bbN \mapsto (0,\infty)$, this results in an lower bound for $s$. Indeed, the latter two conditions imply that for $n$ large enough
\[
\l \frac{1}{n} \r^{1/g(d)} \ll \eps_n \ll \l \frac{1}{n} \r^{1/h(s)}
\]
or equivalently
\[
h(s) \geq g(d).
\]
In the case of Theorem \ref{thm:Main:Res:ConsFracLap}, with $h(s) = s/2 - 1/2$ and $g(d) = d+4$ (from \eqref{eq:Main:Ass:epsLBWellPosed}), solving the latter yields $s>2d + 9$. We will now implicitly assume that this condition is satisfied whenever we are in the well-posed regime.
\end{remark}

The main result of our work is the following theorem that can be considered the analogue of~\cite[Theorem 2.1]{Slepcev} to the fractional Laplacian case.

\begin{theorem}
\label{thm:Main:Res:ConsFracLap}
Consistency of fractional Laplacian learning.
Assume that~\ref{ass:Main:Ass:S1}, \ref{ass:Main:Ass:M1}, \ref{ass:Main:Ass:M2}, \ref{ass:Main:Ass:D1}, \ref{ass:Main:Ass:D2}, \ref{ass:Main:Ass:W1}, \ref{ass:Main:Ass:W2}, and~\ref{ass:Main:Ass:L1} hold.
Let $(\mu_n,u_n)$ be a sequence of minimizers of $\cF_{n,\eps_n}(\cdot)$, assume that $\rho \in \Ck{\infty}$.
Then, $\bbP$-a.e.:
\begin{enumerate}
\item Let $K_n$ be as in Proposition \ref{prop:Proofs:Back:DisReg:EVecAlt} and assume that $\Vert \psi_{n,k} \Vert_{\Lp{\infty}} \leq C_{\psi} \lambda_{k}^{\alpha}$ for all $k \leq \ceil{K_n}$ and some $\alpha > 0$. Let $s > 2\alpha + 2 + d/2$ and assume that $n\eps_n^{s/2 - 1/2}$ is bounded and that $\eps_n$ satisfies \eqref{eq:Main:Ass:epsLBWellPosed}. Then
\begin{enumerate}
\item  $(\mu_n,u_n)$ has a subsequence $\{(\mu_{n_m},u_{n_m})\}_{m=1}^\infty$ converging to some $(\mu,u)$ in $\TLp{2}$;
\item  $u$ is a continuous function and we have
\begin{equation} \label{eq:Main:Res:mainConvergence}
\max_{i \leq n_m} \vert u_{n_m}(x_i) - u(x_i) \vert \to 0;
\end{equation}
\item $(\mu,u)$ is a minimizer of $\cF(\cdot)$;
\item the whole sequence $(\mu,u_n)$ converges to $(\mu,u)$ in $\TLp{2}$ and as in \eqref{eq:Main:Res:mainConvergence}. 
\end{enumerate}
\item If $n\eps_n^{2s} \to \infty$, $\sup_{n \in \bbN} \Vert u_n \Vert_{\Lp{2}}$ is bounded and~\eqref{eq:Main:Ass:epsLBIllPosed} holds: 
\begin{enumerate}
\item $(\mu_n,u_n)$ has a subsequence $\{(\mu_{n_m},u_{n_m})\}_{m=1}^\infty$ converging to some $(\mu,u)$ in $\TLp{2}$;
\item $(\mu,u)$ is a minimizer of $\cG(\cdot)$.
\end{enumerate}
\end{enumerate}
\end{theorem}

\begin{remark}
Ill-posed case.
We want to comment on the results of Theorem \ref{thm:Main:Res:ConsFracLap} in the ill-posed case. We start by discussing the requirement that $\sup_{n\in \bbN} \Vert u_n \Vert_{\Lp{2}} \leq C$ for some constant $C$. If the latter does not hold, then there must exist at least one subsequence so that $\Vert u_{n_m} \Vert_{\Lp{2}} \to \infty$. In that case, this subsequence clearly cannot converge to a function and this is why we choose to exclude this possibility in our results. 
Regarding the assumption that $n\epsilon^{2s} \to \infty$, we note that it also covers the case where $s \leq d/2$. Indeed, by \eqref{eq:Main:Ass:epsLBIllPosed}, we have 
\[
    \infty = \lim_{n \to \infty} \frac{n^{1/d}}{\log(n)^{1/d}} \eps_n \leq \lim_{n \to \infty} n^{1/d} \eps_n \leq \lim_{n \to \infty} n^{1/2s} \eps_n.
\]
\end{remark}

Theorem \ref{thm:Main:Res:ConsFracLap} shows that fractional Laplacian regularization behaves just like the $p$-Laplacian one, although we do not identify the critical rate here. 
Indeed, there is a gap in the rate at which $\eps_n\to0$ between our established ill-posed and well-posed regimes.
We will further elaborate on the rate in subsequent remarks.

A significant variable in the statement of Theorem \ref{thm:Main:Res:ConsFracLap} is the constant $\alpha$ for which $\Vert \psi_{n,k} \Vert_{\Lp{\infty}} \leq C_{\psi} \lambda_{k}^{\alpha}$ for all $k \leq \ceil{K_n}$. While we expect $\|\psi_{n,k}\|_{\Lp{\infty}}\leq C\lambda_{k}^{\frac{d-1}{4}}$ to be the best we can achieve (see for example~\cite{sogge01} for $\Lp{\infty}$ bounds on a continuum class of Laplacian operators), Proposition \ref{prop:Proofs:Back:DisReg:EVecAlt} shows that we get $\|\psi_{n,k}\|_{\Lp{\infty}}\leq C\lambda_{k}^{d+1}$. We refer to Section \ref{subsec:linear} for numerical experiments investigating the optimal values of $\alpha$ and $K_n$.
This yields the following straight-forward Corollary.

\begin{corollary}
\label{cor:Main:Res:ConsFracLap}
Consistency of fractional Laplacian learning.
Assume that~\ref{ass:Main:Ass:S1}, \ref{ass:Main:Ass:M1}, \ref{ass:Main:Ass:M2}, \ref{ass:Main:Ass:D1}, \ref{ass:Main:Ass:D2}, \ref{ass:Main:Ass:W1}, \ref{ass:Main:Ass:W2}, and~\ref{ass:Main:Ass:L1} hold.
Let $(\mu_n,u_n)$ be a sequence of minimizers of $\cF_{n,\eps_n}(\cdot)$, assume that $\rho \in \Ck{\infty}$.
Then, $\bbP$-a.e.:
\begin{enumerate}
\item Let $s > 5d/2 + 4$ and assume that $n\eps_n^{s/2 - 1/2}$ is bounded and that $\eps_n$ satisfies \eqref{eq:Main:Ass:epsLBWellPosed}. Then
\begin{enumerate}
\item  $(\mu_n,u_n)$ has a subsequence $\{(\mu_{n_m},u_{n_m})\}_{m=1}^\infty$ converging to some $(\mu,u)$ in $\TLp{2}$;
\item  $u$ is a continuous function and we have
\begin{equation} \label{eq:Main:Res:mainConvergenceCor}
\max_{i \leq n_m} \vert u_{n_m}(x_i) - u(x_i) \vert \to 0;
\end{equation}
\item $(\mu,u)$ is a minimizer of $\cF(\cdot)$;
\item the whole sequence $(\mu,u_n)$ converges to $(\mu,u)$ in $\TLp{2}$ and as in \eqref{eq:Main:Res:mainConvergenceCor}. 
\end{enumerate}
\item If $n\eps_n^{2s} \to \infty$, $\sup_{n \in \bbN} \Vert u_n \Vert_{\Lp{2}}$ is bounded and~\eqref{eq:Main:Ass:epsLBIllPosed} holds: 
\begin{enumerate}
\item $(\mu_n,u_n)$ has a subsequence $\{(\mu_{n_m},u_{n_m})\}_{m=1}^\infty$ converging to some $(\mu,u)$ in $\TLp{2}$;
\item $(\mu,u)$ is a minimizer of $\cG(\cdot)$.
\end{enumerate}
\end{enumerate}
\end{corollary}

\begin{remark} \label{rem:Main:Res:Sobolev}
Sobolev embeddings.
Suppose that we had $\alpha = 0$ in the statement of Theorem \ref{thm:Main:Res:ConsFracLap} (equivalently suppose that the discrete eigenfunctions are uniformly bounded) then we would require $s > \max\{d/2 + 2,2d+9\}$. The condition $s > d/2 + 2$ can also be considered from the point of view of Sobolev embeddings. By \cite[Lemma  4]{Stuart}, we know that this implies $\cH^s(\Omega) \subseteq \Ck{0,\gamma}(\Omega)$ for some $\gamma > 0$. This is a natural consequence of being in the well-posed regime as we set pointwise constraints in our continuum problem \eqref{eq:Main:Not:F}, hence requiring the minimizer of the latter in $\cH^s(\Omega)$ to be at least continuous. However, we note that the same embedding would apply with the tighter condition $s > d/2$. We therefore conjecture that the extra +2 term in our result is an artifact of our proof of Theorem \ref{thm:Proofs:Back:DisReg:DisRegAlt}.
\end{remark}

\begin{remark} \label{rem:Main:Res:Relationship}
The relationship between the asymptotics of $\eps_n$ and Sobolev embeddings. 
Assuming the conjecture in Remark \ref{rem:Main:Res:Sobolev}, one is able to deduce what the optimal bounds on $\eps_n$ should be. Indeed, using the equivalence discussed in Remark \ref{rem:Main:Res:LowerBoundS}, we expect 
\[
 \l \frac{1}{n} \r^{1/d}  \ll \eps_n \ll \l \frac{1}{n}\r^{1/2s}.
\]
This implies the following two things about our results for the well-posed regime: firstly, we should be able to only impose the sharper lower bound \eqref{eq:Main:Ass:epsLBIllPosed} which is related only to the connectivity of the graph; secondly, our requirement that $n\eps^{s/2-1/2}$ is bounded is not sharp and should be replaced with $n\eps_n^{2s}$ being bounded. On the other hand, the condition $n\eps_n^{2s} \to \infty$ is sharp for the ill-posed regime. The above considerations imply the existence of a gap in our analysis, namely the regime when
\[
\l \frac{1}{n} \r^{1/(s/2-1/2)} \ll \eps_n \ll \l \frac{1}{n}\r^{1/2s}.
\]
While we postulate that this still corresponds to the well-posed regime, we investigate this hypothesis in Section~\ref{sec:NumExp}.
\end{remark}

\begin{remark}\label{rem:Main:Res:LowerBoundGap}
Lower bound gap.
As was pointed out in Remark \ref{rem:Main:Res:Relationship}, replacing the lower bound $n^{-1/(d+4)}  \ll \eps_n$ with $n^{-1/d}  \ll \eps_n$ amounts to using \eqref{eq:Main:Ass:epsLBIllPosed} instead of \eqref{eq:Main:Ass:epsLBWellPosed}. By considering the proofs in Section \ref{sec:Proofs}, we note that this would require the reformulation of the results of Theorem \ref{thm:Proofs:Back:DisReg:DisRegCalder} and Theorem \ref{thm:Proofs:Back:DisReg:LinftyCalder}. Paraphrasing \cite[Remark 2.7]{Calder}, this does not seem unreasonable but represents non-trivial work. Given the latter fact, we will not consider this in greater detail in Section \ref{sec:NumExp}.  
\end{remark}

\begin{remark} \label{rem:Main:Res:Approximating}
Approximating Sobolev semi-norms and numerical schemes. 
As succinctly mentioned in Section \ref{sec:Intro}, Assumption \ref{ass:Main:Ass:L1} allows one to transition from finite differences to derivatives. Informally (see \cite{Calder_2018} for a rigorous approach to the problem based on Taylor expansion), consider \eqref{eq:Back:Rel:pLapDisProb} and assume that $\eps_n = \eps$ and $u_n = u \in \Ck{1}$. Then, as $n \to\infty$, the expression in \eqref{eq:Back:Rel:pLapDisProb} converges to
\begin{equation} \label{eq:Main:Res:pFiniteDifference}
\frac{1}{\eps^p} \int \int \frac{1}{\eps^d} \eta \l \frac{\vert y - x \vert}{\eps} \r \vert u(y) - u(x) \vert^p \, \dd y \dd x = \frac{1}{\eps^p} \int \int \eta(\vert z \vert) \vert u(x + \eps z) - u(x) \vert^p \, \dd z \dd x. 
\end{equation}
Now, assuming that $\eta(x) = 1/x^p$, we obtain
\[
\frac{1}{\eps^p} \int \int \frac{\vert u(x + \eps z) - u(x) \vert^p}{\vert z \vert^p} \, \dd z \dd x \approx \int \vert \nabla u(x) \vert^p \, \dd x.  
\]
This shows that \eqref{eq:Back:Rel:pLapDisProb} is essentially an approximation of a $\Wkp{1}{p}$ semi-norm by finite differences on a non-regular grid -- the graph: this result is made rigorous in \cite{Slepcev}. Disregarding the discrete-to-continuum aspect as well as the pointwise constraints in \cite{Slepcev}, i.e. if we consider \eqref{eq:Main:Res:pFiniteDifference} directly, it is well-known (see \cite{Bourgain01anotherlook}) that the latter approximates the $\Wkp{1}{p}$ norm with minimal conditions on the kernel $\eta$ which we partly require in Assumptions \ref{ass:Main:Ass:W1} and \ref{ass:Main:Ass:W2}. In fact, the finite differences can be used as a characterization of $\Wkp{1}{p}$ (see \cite[Theorem 11.75]{leoni2009first}).

Analogously, in the case of the fractional Laplacian, we know by \cite[Lemma 4]{Stuart} that $\EnergyS(\cdot)$ is essentially a $\Wkp{s}{2}$ semi-norm and therefore, by Theorem \ref{thm:Main:Res:ConsFracLap}, $\EnergySnepsn(\cdot)$ can be viewed as a finite difference approximation of $\Wkp{s}{2
}$ semi-norm on a non-regular grid. Continuing the parallel with $p$-Laplacian regularization, it is the subject of further research to generalize the results of \cite[Theorem 2]{Bourgain01anotherlook} to $\Wkp{s}{2}$ (or even $\Wkp{k}{p}$ with arbitrary $k$ and $p$).

Considering the above, the authors believe that both \cite[Theorem 2.1]{Slepcev} and Theorem \ref{thm:Main:Res:ConsFracLap} can be considered justifications for the use of numerical schemes based on discretizations on a graph for solving certain variational problems. Extending these results to more general variational problems will be undertaken in future research.
Furthermore, the above-mentioned results are asymptotic and it would of interest to obtain convergence rates (see \cite{Calder_2018}, \cite{calder2020rates} or \cite{https://doi.org/10.48550/arxiv.2010.08697} for examples of quantitative rates for such numerical schemes).
\end{remark}

\section{Proofs} \label{sec:Proofs}

In this section we present the proofs of the main result, Theorem~\ref{thm:Main:Res:ConsFracLap}.
We start by including some background results on Weyl's law, convergence of eigenvalues and discrete regularity. 
We then prove compactness of minimisers followed by the $\Gamma$-convergence of the energy $\cF_{n,\eps_n}$ in the well-posed and ill-posed regimes.
In Section~\ref{subsec:Proofs:Linfty} we prove minimisers are bounded in $\Lp{\infty}$. 
The final part of this section proves the main result, Theorem~\ref{thm:Main:Res:ConsFracLap}.

\subsection{Background Results} 

In this section we include some background results that will be useful in the sequel.
In particular, we start by recalling Weyl's law for the scaling of eigenvalues for weighted Laplacian's.
We then adapt some results on the convergence of eigenvalues to our setting.
In Section~\ref{subsubsec:Proofs:Back:DisReg} we include some discrete regularity results, in particular a discrete Morrey's type inequality.

\subsubsection{Weyl's Law} 

Weyl's law is well known in Euclidean and manifold settings with uniform density.
It states that eigenvalues $\{\lambda_k\}_{k\in\bbN}$ of the unweighted Laplacian, i.e. in Euclidean domains $\Delta:=\Delta_1=\sum_{i=1}^d \frac{\partial^2}{\partial x_i^2}$ and on manifolds $\Delta$ is the Laplace-Beltrami operator, scale as $k^{2/d}$.
The version of Weyl's law proven below is analogous to~\cite[Lemma 28]{Stuart} but adapted to our setting.

\begin{proposition}
\label{prop:Proofs:Back:Weyl:Law}
Weyl's law.
Let Assumptions~\ref{ass:Main:Ass:S1} and~\ref{ass:Main:Ass:M2} hold.
Define $\Delta_\rho$ by~\eqref{eq:Main:Not:WeightLap} and let $\{\lambda_k\}_{k\in\bbN}$ be the eigenvalues of $\Delta_\rho$ arranged in increasing order.
Then, there exists constants $0 < c_W \leq C_W$ such that \[ c_W k^{2/d} \leq  \lambda_k \leq C_W k^{2/d}. \]
\end{proposition}

\begin{proof}
Define $c = \min_{x\in\Omega}\rho(x)>0$ and $C = \max_{x\in\Omega} \rho(x)<+\infty$.
Define
\[ \cV_{k-1} = \{ V \subset \Wkp{1}{2}(\Omega) \spaceBar \dim(V) = k-1\} \]
and let $V \in \cV_{k-1}$.
For $u \in (\Wkp{1}{2} \cap V^\perp)\setminus \{0\} $, we therefore have
\[ \frac{c^2}{C} \frac{\int_\Omega \vert \nabla u \vert^2 \,\dd x  }{\int_\Omega u^2 \,\dd x} \leq \frac{\int_\Omega \vert \nabla u \vert^2 \rho^2 \,\dd x}{\int_\Omega u^2 \rho \, \dd x} \leq \frac{C^2}{c} \frac{\int_\Omega \vert \nabla u \vert^2 \, \dd x  }{\int_\Omega u^2 \,\dd x}. \]
From the latter, we can deduce that:
\begin{align*}
\frac{c^2}{C} \sup_{V \in \cV_{k-1}} \inf_{u \in (\Wkp{1}{2} \cap V^\perp)\setminus \{0\}} \frac{\int_\Omega \vert \nabla u \vert^2 \, \dd x  }{\int_\Omega u^2 \, \dd x} & \leq \sup_{V \in \cV_{k-1}} \inf_{u \in (\Wkp{1}{2} \cap V^\perp)\setminus \{0\}} \frac{\int_\Omega \vert \nabla u \vert^2 \rho^2 \, \dd x  }{\int_\Omega u^2 \rho \, \dd x} \\
 & \leq \frac{C^2}{c} \sup_{V \in \cV_{k-1}} \inf_{u \in (\Wkp{1}{2} \cap V^\perp)\setminus \{0\}} \frac{\int_\Omega \vert \nabla u \vert^2 \, \dd x}{\int_\Omega u^2 \, \dd x}.
\end{align*}
Since~\ref{ass:Main:Ass:S1} holds then the unweighted Laplacian, $\Delta_1$, is a compact self-adjoint operator, and hence we can apply the Courant–Fisher characterization of eigenvalues \cite[Max-Min theorem]{Chavel} to infer
\begin{equation} \label{eq:Proofs:Back:Weyl:EValBound} 
\frac{c^2}{C} \lambda_k^{\rmu} \leq \lambda_k \leq \frac{C^2}{c} \lambda_k^{\rmu}, 
\end{equation}
where $\lambda_k^{\rmu}$ is the $k$-th eigenvalue of $\Delta_1$. 
By~\cite[Corrolary p.~218]{Craioveanu}, we have that
\begin{equation} \label{eq:Proofs:Back:Weyl:UnweightedManifold}
\lambda_k^{\rmu} \sim k^{2/d}
\end{equation}
and hence, injecting~\eqref{eq:Proofs:Back:Weyl:UnweightedManifold} into \eqref{eq:Proofs:Back:Weyl:EValBound}, we obtain for $k$ large enough:\[ c_W k^{2/d} \leq  \lambda_k \leq C_W k^{2/d} \]
for some constants $c_W$ and $C_W$. 
\end{proof}

\subsubsection{Convergence Of Eigenvalues} 

Results, such as \cite[Theorem 4]{Trillos}, give convergence of eigenvalues $\lambda_{n,k}\to \lambda_k$ as $n\to\infty$ and moreover, the rate of convergence can be bounded for $k\ll n$.
In other words, there exists $J_n$ such that the limit as $n\to\infty$ of $\sup_{k\in\{1,2,\dots,J_n\}} \vert\lambda_{n,k} - \lambda_k\vert$ can be controlled.
In this section we will derive bounds on $J_n$.
In particular we will show that $J_n\gtrsim \eps_n^{-d}$. 

\begin{lemma}
\label{lem:Proofs:Back:EvalTails:Jn}
Existence of $J_n$.
Let Assumptions~\ref{ass:Main:Ass:S1}, \ref{ass:Main:Ass:M1}, \ref{ass:Main:Ass:M2}, \ref{ass:Main:Ass:W1}, \ref{ass:Main:Ass:W2}, \ref{ass:Main:Ass:D1}, and~\ref{ass:Main:Ass:L1} hold and assume $\eps_n$ satisfies the lower bound in~\eqref{eq:Main:Ass:epsLBIllPosed}. 
Then, for $n$ large enough, there exists an integer $J_n$ and positive constants $C_0,C_1, C_2, C_3$ such that:
\begin{enumerate}
\item $C_0 \eps_n^{-d} \geq J_n \geq C_1\eps_n^{-d}$;
\item $\vert \lambda_{n,k} - \lambda_{k} \vert \leq C_2\lambda_k\l\eps_n + \sqrt{\lambda_{k}}\eps_n + \frac{\dWp{\infty}(\mu_n,\mu)}{\eps_n}\r$ for all $k\in\{1,\hdots,J_n\}$, $\bbP$-a.e.;
\item $n \lambda_{n,J_n}^{-s} \leq C_3 n \eps^{2s}$, $\bbP$-a.e.
\end{enumerate}
\end{lemma}

\begin{proof}
By Assumptions~\ref{ass:Main:Ass:S1}, \ref{ass:Main:Ass:M1}, \ref{ass:Main:Ass:M2}, \ref{ass:Main:Ass:W1}, \ref{ass:Main:Ass:W2}, and~\ref{ass:Main:Ass:D1}, 
 we can apply \cite[Theorem 4]{Trillos}. 
In particular, there exists positive constants $c_{\mathrm{eig}},C_{\mathrm{eig}}$, $\eps_0$ and $c$ dependent on $\Omega$, $\rho$ and $\eta$ such that
\begin{equation} \label{eq:Proofs:Back:EvalTails:Trillos4}
\sqrt{\lambda_{k}}\eps < c_{\mathrm{eig}} \text{ and } c\dWp{\infty}(\mu_n,\mu) \leq \eps_n \leq \eps_0 \Rightarrow \vert \lambda_{n,k} - \lambda_k \vert \leq C_{\mathrm{eig}} \lambda_k \ls \eps_n + \sqrt{\lambda_k}\eps_n + \frac{\dWp{\infty}(\mu_n,\mu)}{\eps_n}  \rs.
\end{equation}
By Theorem~\ref{thm:Back:TLp:LinftyMapsRate} and Assumption~\ref{ass:Main:Ass:L1} and~\eqref{eq:Main:Ass:epsLBIllPosed} we can assume that $c\dWp{\infty}(\mu_n,\mu) \leq \eps_n \leq \eps_0$ holds for $n$ sufficiently large.
Without loss of generality we may assume that $C_{\mathrm{eig}} c_{\mathrm{eig}} \geq 1$ (we may always increase the value of $C_{\mathrm{eig}}$ in~\eqref{eq:Proofs:Back:EvalTails:Trillos4}).
By Assumptions~\ref{ass:Main:Ass:S1} and~\ref{ass:Main:Ass:M2} there exists constant $c_{\rmW}$ and $C_{\rmW}$ such that the conclusions of Proposition~\ref{prop:Proofs:Back:Weyl:Law} hold.
We choose constants $0 < \delta < 1$ and $D>0$ such that $\delta C_{\mathrm{eig}} c_{\mathrm{eig}} < 1$ and
\begin{equation} \label{eq:Proofs:Back:EvalTails:firstCaseK}
D < \delta c_{\mathrm{eig}}\frac{\sqrt{c_{\rmW}}}{\sqrt{C_{\rmW}}}.
\end{equation}
We define
\begin{equation} \label{eq:Proofs:Back:EvalTails:firstCaseJn}
J_n = \ceil{ \frac{D}{\eps_n\sqrt{c_{\rmW}}}}^d + 1 \lb 
\begin{array}{l} 
\geq \l\frac{D}{\eps_n \sqrt{c_{\rmW}}} \r^d =: \frac{C_1}{\eps_n^d} \\
\leq 2 \l\frac{D}{\eps_n \sqrt{c_{\rmW}}} \r^d =: \frac{C_0}{\eps_n^d}
\end{array} \rd
\end{equation}
which proves the first statement of the lemma. 

For the second statement we can estimate as follows: for $k\in\{1,\hdots,J_n\}$,
\begin{align}
\sqrt{\lambda_{k}} \eps_n & \leq \sqrt{C_{\rmW}} J_n^{1/d} \eps_n\label{eq:Proofs:Back:EvalTails:firstCaseUBineq1} \\
 & \leq \l \frac{D}{\sqrt{c_{\rmW}}\eps_n} + 1 \r \sqrt{C_{\rmW}} \eps_n \label{eq:Proofs:Back:EvalTails:firstCaseUBineq2} \\
 & = \frac{D\sqrt{C_{\rmW}}}{\sqrt{c_{\rmW}}} + \eps_n \sqrt{C_{\rmW}} \notag \\
 & < \delta c_{\mathrm{eig}} + \eps_n \sqrt{C_{\rmW}} \label{eq:Proofs:Back:EvalTails:firstCaseUBineq4}
\end{align}
where we used Proposition~\ref{prop:Proofs:Back:Weyl:Law} for~\eqref{eq:Proofs:Back:EvalTails:firstCaseUBineq1}, the upper bound in~\eqref{eq:Proofs:Back:EvalTails:firstCaseJn} for~\eqref{eq:Proofs:Back:EvalTails:firstCaseUBineq2}, and~\eqref{eq:Proofs:Back:EvalTails:firstCaseK} for~\eqref{eq:Proofs:Back:EvalTails:firstCaseUBineq4}.
Recall that $0<\delta<1$ so that $\delta c_{\mathrm{eig}} < c_{\mathrm{eig}}$ and, as $\eps_n \to 0$, for $n$ large enough, the second term in~\eqref{eq:Proofs:Back:EvalTails:firstCaseUBineq4} can be made arbitrarily small, so that we achieve:
\begin{equation} \label{eq:Proofs:Back:EvalTails:firstCaseUB}
\sqrt{\lambda_{k}} \eps_n < \delta^\prime c_{\mathrm{eig}} < c_{\mathrm{eig}}
\end{equation} 
for some $\delta <\delta^\prime<1$.
Inequality~\eqref{eq:Proofs:Back:EvalTails:firstCaseUB} will allow us to apply \eqref{eq:Proofs:Back:EvalTails:Trillos4} and hence proves the second statement of the lemma.
Also, note that $\delta^\prime$ can be chosen arbitrarily close to $\delta$ so that we might assume that
\begin{equation} \label{eq:Proofs:Back:EvalTails:firstCaseDeltaPrime}
\delta C_{\mathrm{eig}} c_{\mathrm{eig}} < \delta^\prime C_{\mathrm{eig}}c_{\mathrm{eig}} < 1.
\end{equation} 

For the final statement we let 
$\{T_n\}_{i=1}^\infty$ be the sequence of transport maps from $\mu$ to $\mu_n$ (that exists $\bbP$ almost surely by Theorem~\ref{thm:Back:TLp:LinftyMapsRate}) such that
\[ \begin{cases}
\limsup_{n \to \infty} \frac{n^{1/2} \Vert \Id - T_n \Vert_{\Lp{\infty}} }{\log(n)^{3/4}} \leq C & \text{if } d = 2, \\
\limsup_{n \to \infty} \frac{n^{1/d} \Vert \Id - T_n \Vert_{\Lp{\infty}}}{\log(n)^{1/d}} \leq C &\text{if } d \geq 3.
\end{cases} \]
In particular, we have the following estimate for $d \geq 3$ (and analogously for $d = 2$):
\begin{equation} \label{eq:Proofs:Back:EvalTails:ratioLimit}
\limsup_{n \to \infty} \frac{\dWp{\infty}(\mu_n,\mu)}{\eps_n} \leq \limsup_{n\to \infty} \frac{\Vert \Id- T_n \Vert_{\Lp{\infty}} n^{1/d}}{\log(n)^{1/d}} \frac{\log(n)^{1/d}}{n^{1/d}\eps_n} = 0
\end{equation}
where we used the fact that $ \dWp{\infty}(\mu,\mu_n) \leq \Vert \Id-T_n \Vert_{\Lp{\infty}}$ for the inequality and~\eqref{eq:Main:Ass:epsLBIllPosed} and~\eqref{eq:Back:TLp:LinftyMapsRate} for the equality.

We directly verify that
\begin{equation} \label{eq:Proofs:Back:EvalTails:firstCaseLB}
\sqrt{\lambda_{J_n}} \eps_n \geq \sqrt{c_{\rmW}} J_n^{1/d} \eps_n \geq D
\end{equation}
where we used Proposition~\ref{prop:Proofs:Back:Weyl:Law} for the first inequality and~\eqref{eq:Proofs:Back:EvalTails:firstCaseJn} for the second one. 
Let us now estimate as follows:
\begin{align}
A & := \frac{n}{\lambda_{n,J_n}^{s}} \leq n \lambda_{J_n}^{-s}\ls (1-o(1) - C_{\mathrm{eig}}\sqrt{\lambda_{J_n}}\eps_n\rs^{-s} \label{eq:Proofs:Back:EvalTails:firstCaseLimitIneq1} \\
 & \leq \frac{n \eps_n^{2s}}{D^{2s}} \ls 1 -  o(1)  - C_{\mathrm{eig}}\sqrt{\lambda_{J_n}}\eps_n\rs^{-s}\label{eq:Proofs:Back:EvalTails:firstCaseLimitIneq2}
\end{align}
where we used \eqref{eq:Proofs:Back:EvalTails:Trillos4} and \eqref{eq:Proofs:Back:EvalTails:ratioLimit} for \eqref{eq:Proofs:Back:EvalTails:firstCaseLimitIneq1}, and~\eqref{eq:Proofs:Back:EvalTails:firstCaseLB} for~\eqref{eq:Proofs:Back:EvalTails:firstCaseLimitIneq2}.
Now, using \eqref{eq:Proofs:Back:EvalTails:firstCaseUB} and \eqref{eq:Proofs:Back:EvalTails:firstCaseDeltaPrime} we deduce that for $n$ large enough
\begin{equation} \label{eq:Proofs:Back:EvalTails:firstCaseBoundSecondTerm}
1 - o(1) - C_{\mathrm{eig}}\sqrt{\lambda_{J_n}}\eps_n > 1 - o(1) - \delta^\prime c_{\mathrm{eig}}C_{\mathrm{eig}} > \delta_0
\end{equation}
for some $\delta_0>0$.
This is equivalent to
\begin{equation} \label{eq:Proofs:Back:EvalTails:firstCaseBound}
(1 - o(1) -C_{\mathrm{eig}}\sqrt{\lambda_{J_n}}\eps_n)^{-s} < \delta_0^{-s} 
\end{equation}
Finally, using \eqref{eq:Proofs:Back:EvalTails:firstCaseBound}, we obtain
\[ A \leq \frac{n\eps_n^{2s}}{D^{2s}\delta_0^s} =: C_3 n\eps_n^{2s} \]
as required.
\end{proof}

\begin{remark}
\label{rem:Proofs:Back:EvalTails:eigenvalueBounds}
Eigenvalue bounds for $k \leq J_n$.
Let $k\in\{1,\hdots,J_n\}$ then 
we can apply the second statement in Lemma~\ref{lem:Proofs:Back:EvalTails:Jn}, i.e. 
\begin{equation} \label{eq:Proofs:Back:EvalTails:theorem4Remark}
\lambda_{n,k} \geq \lambda_{k}\l 1 - o(1) -  C_2 \sqrt{\lambda_{k}}\eps_n \r \geq \lambda_{k}\l 1 - o(1) -  C_2 \sqrt{\lambda_{J_n}}\eps_n \r.
\end{equation}
Inserting~\eqref{eq:Proofs:Back:EvalTails:firstCaseBoundSecondTerm} (recalling that $C_2=C_{\mathrm{eig}}$) 
into~\eqref{eq:Proofs:Back:EvalTails:theorem4Remark}, we obtain
\[
\lambda_{n,k} \geq \lambda_{k}\delta_0.
\]
\end{remark}

\begin{corollary}
\label{cor:Proofs:Back:EvalTails:Kn}
Let Assumptions~\ref{ass:Main:Ass:S1}, \ref{ass:Main:Ass:M1}, \ref{ass:Main:Ass:M2}, \ref{ass:Main:Ass:W1}, \ref{ass:Main:Ass:W2}, \ref{ass:Main:Ass:D1}, and~\ref{ass:Main:Ass:L1} hold and assume $\eps_n$ satisfies the lower bound in~\eqref{eq:Main:Ass:epsLBIllPosed}. Let $K_n = \alpha \eps_n^{-d/2} + 1$ for some constant $\alpha > 0$ and assume that $s > 1$. Then, for $n$ large enough and a positive constant $C$: 
\[
\lambda_{n,\ceil{K_n}}^{-1} \leq C \eps_n, \text{ } \bbP\text{-a.e.}
\]
\end{corollary}

\begin{proof}
In the proof $C>0$ ($c > 0$) will denote a constant that can be arbitrarily large (small), independent of $n$ and $k$ that may change from line to line.

Since for $n$ large enough $\ceil{K_n} \leq K_n \leq J_n$, we have that $\lambda_{n,\ceil{K_n}} \geq \lambda_{\ceil{K_n}} \delta_0$ by Remark \ref{rem:Proofs:Back:EvalTails:eigenvalueBounds}. Furthermore, by Proposition \ref{prop:Proofs:Back:Weyl:Law}: 
\[ 
\lambda_{\ceil{K_n}}^{1/2} \geq c \ceil{K_n}^{1/d} \geq c \eps_n^{-1/2}. 
\]
Rearranging the latter, we obtain: 
\[
\lambda_{n,\ceil{K_n}}^{-1} \leq C \lambda_{\ceil{K_n}}^{-1} \leq C \eps_n. \qedhere
\]
\end{proof}

\subsubsection{Discrete Regularity} \label{subsubsec:Proofs:Back:DisReg}

Similarly to \cite[Lemma 4.1]{Slepcev}, we can show a regularity result for the discrete functions.
Our result can be seen as a discrete analogue of a Morrey-type inequality.
We start by recalling the regularity result in~\cite{Calder} (more precisely we apply the Borel--Cantelli lemma to Theorem 2.1 in~\cite{Calder} -- as we do in the proof of Proposition \ref{prop:Proofs:Back:DisReg:EVecAlt} -- to deduce the same conclusions with probability one, rather than a high probability bound).

\begin{theorem}
\label{thm:Proofs:Back:DisReg:DisRegCalder}
Regularity of functions I~\cite[Theorem 2.1]{Calder}.
Assume Assumptions~\ref{ass:Main:Ass:S1}, \ref{ass:Main:Ass:M1}, \ref{ass:Main:Ass:M2}, \ref{ass:Main:Ass:D1}, \ref{ass:Main:Ass:W1} ~\ref{ass:Main:Ass:W2}, \ref{ass:Main:Ass:L1} hold and $\rho\in\Ck{2}$. Then, $\bbP$-a.s., there exists $C > 0$ such that for $n$ large enough, we have
\[
\vert u(x) - u(y) \vert \leq C(\Vert u \Vert_{\Lp{\infty}} + \Vert \Delta_{n,\eps_n}u \Vert_{\Lp{\infty}}) (\dTorus(x,y) + \eps_n)
\]
for any $u:\Omega_n \mapsto \bbR$ and $x,y \in \Omega_n$.
\end{theorem}

We recall a second result which will allow us to provide an $\Lp{\infty}$ bound on the eigenfunctions $\{\psi_{n,k}\}_{k=1}^{J_n}$.

\begin{theorem}
\label{thm:Proofs:Back:DisReg:LinftyCalder}
Bounds in $\Lp{\infty}$~\cite[Corollary 2.5]{Calder}.
Assume Assumptions~\ref{ass:Main:Ass:S1}, \ref{ass:Main:Ass:M1}, \ref{ass:Main:Ass:M2}, \ref{ass:Main:Ass:D1}, \ref{ass:Main:Ass:W1},\ref{ass:Main:Ass:W2} and~\ref{ass:Main:Ass:L1} hold and $\rho\in\Ck{2}$.
There exists $\eps_0>0$, $C,c>0$ such that for any $\Lambda>0$, with probability at least $1 - C\eps^{-6d}e^{-cn\eps^{d+4}} - 2ne^{-cn(\Lambda+1)^{-d}}$, we have $\|u\|_{\Lp{\infty}}\leq C(\Lambda+1)^{d+1}\|u\|_{\Lp{1}(\mu_n)}$ for all $0<\eps\leq \frac{\eps_0}{\Lambda+1}$ and $u:\Omega_n\to \bbR$ satisfying
\[ \frac{\|\Delta_{n,\eps} u\|_{\Lp{\infty}}}{\|u\|_{\Lp{\infty}}} \leq \Lambda. \]
\end{theorem}

We now use Theorem~\ref{thm:Proofs:Back:DisReg:LinftyCalder} to derive an $\Lp{\infty}$ bound on the first $K_n$ eigenvectors, where we scale $K_n\sim \eps^{-d/2}_n$.

\begin{proposition}
\label{prop:Proofs:Back:DisReg:EVecAlt}
Assume Assumptions~\ref{ass:Main:Ass:S1}, \ref{ass:Main:Ass:M1}, \ref{ass:Main:Ass:M2}, \ref{ass:Main:Ass:D1}, \ref{ass:Main:Ass:W1}, \ref{ass:Main:Ass:W2} and~\ref{ass:Main:Ass:L1} hold, $\eps_n$ satisfies~\eqref{eq:Main:Ass:epsLBWellPosed}, and $\rho\in\Ck{2}$.
Let $\psi_{n,k}$ be the ordered eigenfunctions of $\Delta_{n,\eps_n}$ defined by~\eqref{eq:Main:Not:Setting:Deltaneps}.
Let $K_n=\alpha\eps_n^{-d/2} + 1$.
Then, $\bbP$-a.s., there exists $C>0$ and $\alpha_0$ such that for $n$ sufficiently large, and for all $k\in\{2,\hdots,\ceil{K_n}\}$, $\alpha\in(0,\alpha_0]$ we have $\|\psi_{n,k}\|_{\Lp{\infty}}\leq C\lambda_k^{d+1}$.
\end{proposition}

\begin{proof}
In the proof $C>0$ ($c>0$) will denote a constant that can be arbitrarily large (small), independent of $n$ and $k$ that may change from line to line.
Our choice of $K_n$ implies that (for $n$ sufficiently large, $\bbP$-a.s.) $K_n\leq J_n$ where $J_n$ is defined in Lemma~\ref{lem:Proofs:Back:EvalTails:Jn} as well as $\lambda_{\ceil{K_n}} \sim \alpha^{2/d}\eps_n^{-1}$ by Proposition \ref{prop:Proofs:Back:Weyl:Law}.
By Lemma~\ref{lem:Proofs:Back:EvalTails:Jn} and~Proposition~\ref{prop:Proofs:Back:Weyl:Law} we have, for any $k\in\{1,\hdots,\ceil{K_n}\}$,
\begin{align*}
\lambda_{n,k} & \leq \lambda_k \l 1 + C \l \eps_n + \sqrt{\lambda_k}\eps_n + \frac{\dWp{\infty}(\mu_n,\mu)}{\eps_n}\r \r \\
 & \leq \lambda_k \l 1 + \l \eps_n + \sqrt{\lambda_{\ceil{K_n}}}\eps_n + \frac{\dWp{\infty}(\mu_n,\mu)}{\eps_n}\r \r \\
 & \leq \lambda_k  \l 1 + C\l \eps_n + \alpha^{1/d}\sqrt{\eps_n} + \frac{\dWp{\infty}(\mu_n,\mu)}{\eps_n}\r \r \\
 & \leq C\lambda_k
\end{align*}
for $n$ sufficiently large.
We choose 
$\Lambda=C\lambda_k$
in Theorem~\ref{thm:Proofs:Back:DisReg:LinftyCalder}.
Therefore, assuming 
$\eps_n\leq \frac{\eps_0}{C\lambda_k + 1}$
and since 
\[ \frac{\|\Delta_{n,\eps_n}\psi_{n,k}\|_{\Lp{\infty}}}{\|\psi_{n,k}\|_{\Lp{\infty}}} = \lambda_{n,k} \leq C\lambda_k = \Lambda, \] 
we have
\[ \|\psi_{n,k}\|_{\Lp{\infty}} \leq C^\prime \l C\lambda_k +1 \r^{d+1} \] 
with probability at least $1 - C\eps_n^{-6d}e^{-cn\eps_n^{d+4}} - 2ne^{-cn(C\lambda_{k}+1)^{-d}}$. 
Using the fact that $\lambda_{\ceil{K_n}} \sim \alpha^{2/d}\eps_n^{-1}$, we can simplify. 
Indeed, the condition on $\eps_n$ is implied by 
\begin{equation} \label{eq:Proofs:Back:DisReg:epsConAlt2}
C\alpha_0^{\frac{2}{d}}+\eps_n\leq \eps_0.
\end{equation}
Since $\eps_n\to 0$ then we can choose $\alpha_0$ sufficiently small so that~\eqref{eq:Proofs:Back:DisReg:epsConAlt2} holds for $n$ sufficiently large.
We now fix $\alpha_0$ and absorb it into our constants $C,c$.
We can write
\[ \|\psi_{n,k}\|_{\Lp{\infty}} \leq C \lambda_k^{d+1} \]
with probability at least $1 - C\eps_n^{-6d}e^{-cn\eps_n^{d+4}} - 2ne^{-cn(C\lambda_{k}+1)^{-d}}$. 
We note that
\begin{align*}
\frac{n}{(C\lambda_{k}+1)^{d}} & \geq \frac{n}{(C\lambda_{\ceil{K_n}}+1)^{d}} \\
 & \geq \frac{n}{(C\eps_n^{-1} +1)^{d}} \\
 & \geq Cn\eps^{d}.
\end{align*}
The assumption that $\frac{n\eps_n^{d+4}}{\log n}\gg 1$ implies that we can bound $\eps_n^{-6d} e^{-cn\eps_n^{d+4}} \leq \eps_n^A$ for any $A$ we choose and for $n$ sufficiently large (where ``sufficiently large'' depends on the choice of $A$).
We choose $A$ so that 
$\sum_{n=1}^d \eps_n^A<+\infty$
and therefore by the Borel-Cantelli lemma we can conclude that, $\bbP$-a.s. the result holds for $n$ sufficiently large.
\end{proof}

Using the above results we derive a second regularity result better suited to our setting. While in Proposition \ref{prop:Proofs:Back:DisReg:EVecAlt} we showed that $\|\psi_{n,k}\|_{\Lp{\infty}}\leq C\lambda_{k}^{d+1}$, we state the following results for an estimate of the form $\|\psi_{n,k}\|_{\Lp{\infty}}\leq C_\psi\lambda_{k}^{\alpha}$ for some $\alpha > 0$.

\begin{theorem}
\label{thm:Proofs:Back:DisReg:DisRegAlt}
Regularity of functions II.
Assume Assumptions~\ref{ass:Main:Ass:S1}, \ref{ass:Main:Ass:M1}, \ref{ass:Main:Ass:M2}, \ref{ass:Main:Ass:D1}, \ref{ass:Main:Ass:W1}, \ref{ass:Main:Ass:W2} and~\ref{ass:Main:Ass:L1} hold, $\eps_n$ satisfies~\eqref{eq:Main:Ass:epsLBWellPosed}, and $\rho\in\Ck{2}$.
Let $K_n$ be as in Proposition \ref{prop:Proofs:Back:DisReg:EVecAlt} and assume that $\Vert \psi_{n,k} \Vert_{\Lp{\infty}} \leq C_{\psi} \lambda_{k}^{\alpha}$ for all $k \leq \ceil{K_n}$ and some $\alpha > 0$. Let $s > 2\alpha + 2 + d/2$.
Then, for any $\gamma > 0$, there exists $C>0$ such that, $\bbP$-a.s., for $n$ sufficiently large we have 
\begin{equation} \label{eq:Proofs:Back:DisReg:DisReg}
\vert u_n(x_i) - u_n(x_j)  \vert \leq C (\dTorus(x_i,x_j) + \eps_n) \l \sqrt{\EnergySnepsn(u_n)} + \Vert u_n \Vert_{\Lp{2}} \r + C \gamma n \eps_n^{s/2 + 1/2} \sqrt{\EnergySnepsn(u_n)}
\end{equation} 
for all $x_i,x_j\in\Omega_n$ with $ \dTorus(x_i,x_j) \leq \gamma \eps_n$ and any $u_n:\Omega_n\to \bbR$.
\end{theorem}

\begin{proof}
In the proof $C>0$ will denote a constant that can be arbitrarily large, independent of $n$ and $k$ that may change from line to line.
With probability one, we can assume that the conclusions of Lemma~\ref{lem:Proofs:Back:EvalTails:Jn}, Remark~\ref{rem:Proofs:Back:EvalTails:eigenvalueBounds},
Corollary~\ref{cor:Proofs:Back:EvalTails:Kn}, 
Theorem~\ref{thm:Proofs:Back:DisReg:DisRegCalder}, Theorem~\ref{thm:Proofs:Back:DisReg:LinftyCalder}
and Proposition \ref{prop:Proofs:Back:DisReg:EVecAlt} hold.

Let $x_i,x_j \in \Omega_n$ with $ \dTorus(x_i,x_j) \leq \gamma \eps_n$. We note that $\dTorus(x_i,x_j)$ is the length of the shortest path between the equivalence classes of $x_i$ and $x_j$ in $\bbR^d/\mathbb{Z}^d$. Hence, there exists points $x_i^* \in x_i + \mathbb{Z}^d \subseteq \mathbb{R}^d$ and $x_j^* \in x_j + \mathbb{Z}^d \subseteq \mathbb{R}^d$ such that $\dTorus(x_i,x_j) = \Vert x_i^* - x_j^* \Vert$.

With $K_n$ from Proposition \ref{prop:Proofs:Back:DisReg:EVecAlt}, we start by estimating: 
\begin{align}
\vert u_n(x_i^*) - u_n(x_j^*) \vert & = \la \sum_{k=1}^n \langle u_n, \psi_{n,k}\rangle (\psi_{n,k}(x_i^*) - \psi_{n,k}(x_j^*)) \ra \notag \\
 & \leq \sum_{k=1}^{\ceil{K_n}-1} \vert \langle u_n, \psi_{n,k}\rangle \vert \vert \psi_{n,k}(x_i^*) - \psi_{n,k}(x_j^*) \vert + \sum_{k=\ceil{K_n} }^{n} \vert \langle u_n, \psi_{n,k}\rangle \vert \vert \psi_{n,k}(x_i^*) - \psi_{n,k}(x_j^*) \vert \notag \\
 & \leq C \sum_{k=1}^{\ceil{K_n}-1}\Vert \psi_{n,k}\Vert_{\Lp{\infty}} \vert \langle u_n, \psi_{n,k}\rangle \vert (\Vert x_i^* - x_j^* \Vert + \eps_n) \notag \\
 & \qquad + C \sum_{k=1}^{\ceil{K_n}-1} \lambda_{n,k} \Vert \psi_{n,k}\Vert_{\Lp{\infty}} \vert \langle u_n, \psi_{n,k}\rangle \vert (\Vert x_i^* - x_j^* \Vert + \eps_n) \label{eq:UijSecondInequality}\\
 & \qquad + \sum_{k=\ceil{K_n} }^{n} \vert \langle u_n, \psi_{n,k}\rangle \vert \vert \psi_{n,k}(x_i^*) - \psi_{n,k}(x_j^*) \vert \notag \\
 &=: A + B + D, \notag
\end{align}
where we use Theorem~\ref{thm:Proofs:Back:DisReg:DisRegCalder} 
for the eigenfunctions $\psi_{n,k}$ in~\eqref{eq:UijSecondInequality}.

We now proceed to bound the terms $A$, $B$ and $D$ individually.
Starting with $A$, we have
\begin{align}
A &= C \sum_{k=1}^{\ceil{K_n}-1} \Vert \psi_{n,k}\Vert_{\Lp{\infty}} \vert \langle u_n, \psi_{n,k}\rangle \vert (\Vert x_i^* - x_j^* \Vert + \eps_n) \notag \\
&= C (\Vert x_i^* - x_j^* \Vert + \eps_n) \left[ \sum_{k=2}^{\ceil{K_n}-1} \Vert \psi_{n,k}\Vert_{\Lp{\infty}} \vert \langle u_n, \psi_{n,k}\rangle \vert  + \Vert \psi_{n,1} \Vert_{\Lp{\infty}} \vert \langle u_n, \psi_{n,1}\rangle \vert \right] \notag \\
&\leq C (\Vert x_i^* - x_j^* \Vert + \eps_n) \left[ \sum_{k=2}^{\ceil{K_n}-1} \lambda_k^{\alpha} \vert \langle u_n, \psi_{n,k}\rangle \vert  + \Vert u_n \Vert_{\Lp{2}} \right] \label{eq:Proofs:Back:DisReg:A3} \\
&= C (\Vert x_i^* - x_j^* \Vert + \eps_n) \left[ \sum_{k=2}^{\ceil{K_n}-1} \lambda^{s/2}_{n,k} \vert \langle u_n, \psi_{n,k}\rangle \vert \lambda_k^{\alpha} \lambda_{n,k}^{-s/2} + \Vert u_n \Vert_{\Lp{2}} \right] \notag  \\
&\leq C (\Vert x_i^* - x_j^* \Vert + \eps_n) \left[  \sqrt{\EnergySnepsn(u_n)} \sqrt{\sum_{k=2}^{\ceil{K_n}-1} \lambda_{n,k}^{-s} \lambda_k^{2\alpha}}  + \Vert u_n \Vert_{\Lp{2}} \right] \notag\\
&\leq C (\Vert x_i^* - x_j^* \Vert + \eps_n) \left[ \sqrt{\EnergySnepsn(u_n)} \sqrt{\sum_{k=2}^{\infty} \lambda_k^{2\alpha - s}} +  \Vert u_n \Vert_{\Lp{2}} \right] \label{eq:Proofs:Back:DisReg:A6} \\
&\leq C (\Vert x_i^* - x_j^* \Vert + \eps_n) \left[\sqrt{\EnergySnepsn(u_n)} \sqrt{\sum_{k=2}^{\infty} k^{(2/d)(2\alpha - s)}} + \Vert u_n \Vert_{\Lp{2}} \right] \label{eq:Proofs:Back:DisReg:A7}
\end{align}
where we use the fact that $\Vert \psi_{n,k} \Vert_{\Lp{\infty}} \leq C_\psi \lambda^{\alpha}_{k}$ and Theorem \cite[Theorem 2.6]{Calder} for \eqref{eq:Proofs:Back:DisReg:A3},
the fact that $K_n \leq J_n$ and Remark \ref{rem:Proofs:Back:EvalTails:eigenvalueBounds} for \eqref{eq:Proofs:Back:DisReg:A6}
and Proposition \ref{prop:Proofs:Back:Weyl:Law} for \eqref{eq:Proofs:Back:DisReg:A7}. Since our assumption on $s$ implies $s > 2\alpha + d/2$, we finally obtain: \begin{equation} \label{eq:Proofs:Back:Disreg:A}
    A \leq C (\Vert x_i^* - x_j^* \Vert + \eps_n) \left[\sqrt{\EnergySnepsn(u_n)} + \Vert u_n \Vert_{\Lp{2}} \right]. 
\end{equation}

Similarly,
\begin{align}
    B  &=  C \sum_{k=1}^{\ceil{K_n}-1} \lambda_{n,k} \Vert \psi_{n,k}\Vert_{\Lp{\infty}} \vert \langle u_n, \psi_{n,k}\rangle \vert (\Vert x_i^* - x_j^* \Vert + \eps_n) \notag \\
    &\leq  C (\Vert x_i^* - x_j^* \Vert + \eps_n)  \sum_{k=1}^{\ceil{K_n}-1} \lambda_{n,k}^{s/2} \lambda_k^{\alpha} \vert \langle u_n, \psi_{n,k}\rangle \vert \lambda_{n,k}^{1-s/2} \label{eq:Proofs:Back:DisReg:B2} \\
    &\leq C  (\Vert x_i^* - x_j^* \Vert + \eps_n) \sqrt{\EnergySnepsn(u_n)} \sqrt{\sum_{k=1}^{\ceil{K_n}-1} \lambda_{k}^{2\alpha} \lambda_{n,k}^{2-s}} \notag\\
    &\leq C  (\Vert x_i^* - x_j^* \Vert + \eps_n) \sqrt{\EnergySnepsn(u_n)} \sqrt{\sum_{k=1}^{\infty} \lambda_k^{2\alpha + 2 - s} } \label{eq:Proofs:Back:DisReg:B4}\\
    &\leq C (\Vert x_i^* - x_j^* \Vert + \eps_n) \sqrt{\EnergySnepsn(u_n)} \sqrt{\sum_{k=1}^{\infty} k^{(2/d)(2\alpha + 2 - s)} } \label{eq:Proofs:Back:DisReg:B5}
\end{align}
where we use $\Vert \psi_{n,k} \Vert_{\Lp{\infty}} \leq C_\psi \lambda^\alpha_k$ for \eqref{eq:Proofs:Back:DisReg:B2},
the fact that $K_n \leq J_n$ and Remark \ref{rem:Proofs:Back:EvalTails:eigenvalueBounds} for \eqref{eq:Proofs:Back:DisReg:B4}
and Proposition \ref{prop:Proofs:Back:Weyl:Law} for \eqref{eq:Proofs:Back:DisReg:B5}. Since by assumption $s > 2\alpha + 2 + d/2$, we obtain: \begin{equation} \label{eq:Proofs:Back:Disreg:B}
    B \leq C (\Vert x_i^* - x_j^* \Vert + \eps_n) \sqrt{\EnergySnepsn(u_n)}. 
\end{equation}

Finally, for $D$, we have
\begin{align}
    D &= \sum_{k=\ceil{K_n}}^{n} \vert \langle u_n, \psi_{n,k}\rangle \vert \vert \psi_{n,k}(x_i^*) - \psi_{n,k}(x_j^*) \vert \notag \\
    &\leq C\gamma \sum_{k=\ceil{K_n}}^{n} \vert \langle u_n, \psi_{n,k}\rangle \vert \sqrt{n} \eps_n \lambda_{n,k}^{1/2} \label{eq:Proofs:Back:DisReg:D2} \\
    &= C \gamma n^{3/2} \eps_n \sqrt{ \l \frac{1}{n} \sum_{k = \ceil{K_n}}^n \lambda_{n,k}^{1/2} \vert \langle u_n, \psi_{n,k}\rangle \vert \r^2 } \notag\\
    &\leq C \gamma n \eps_n \l \sum_{k= \ceil{K_n}}^n \lambda_{n,k}^{s/2} \vert \langle u_n, \psi_{n,k}\rangle \vert \lambda_{n,k}^{1-s/2} \vert \langle u_n, \psi_{n,k}\rangle \vert \r^{1/2} \notag \\
    &\leq C \gamma n \eps_n \l \sqrt{\EnergySnepsn(u_n)} \l \sum_{k=\ceil{K_n}}^{n} \lambda_{n,k}^{2-2s} \lambda_{n,k}^{s} \vert \langle u_n, \psi_{n,k}\rangle \vert^2 \r^{1/2}  \r^{1/2}  \notag\\
    &\leq C \gamma n \eps_n \sqrt{\EnergySnepsn(u_n)}  \lambda_{n,\ceil{K_n}}^{\frac{2-2s}{4}} \label{eq:Proofs:Back:DisReg:D6}\\
    &\leq C \gamma n \eps_n^{s/2 + 1/2} \sqrt{\EnergySnepsn(u_n)}  \label{eq:Proofs:Back:DisReg:D7}
\end{align}
where we use \cite[Lemma 4.1]{Slepcev} for \eqref{eq:Proofs:Back:DisReg:D2}, the fact that $s > 2$ for \eqref{eq:Proofs:Back:DisReg:D6} 
and Corollary \ref{cor:Proofs:Back:EvalTails:Kn} for \eqref{eq:Proofs:Back:DisReg:D7}. 

Combining, \eqref{eq:Proofs:Back:Disreg:A}, \eqref{eq:Proofs:Back:Disreg:B} and \eqref{eq:Proofs:Back:DisReg:D7}, we have:
\[
\vert u_n(x_i^*) - u_n(x_j^*)  \vert \leq C (\Vert x_i^* - x_j^* \Vert + \eps_n) \l \sqrt{\EnergySnepsn(u_n)} + \Vert u_n \Vert_{\Lp{2}} \r + C \gamma n \eps_n^{s/2 + 1/2} \sqrt{\EnergySnepsn(u_n)}
\]
for $\Vert x_i^* - x_j^* \Vert \leq \gamma \eps_n$, which, since $\Vert x_i^* - x_j^* \Vert = \dTorus(x_i,x_j)$, yields the claim of the proposition.
\end{proof}


The next result is essential to show uniform compactness of our sequence of minimizers in Section \ref{subsec:compacteness}. In particular, our strategy in Proposition \ref{prop:Proofs:Compactness:L2Compactness} will be to use the Ascoli-Arzel\`a theorem on a sequence of mollified mimimizers. For the latter, we will need equicontinuity of our sequence which we deduce by the regularity properties of the minimizers proven in Corollary \ref{cor:Proofs:Back:DisReg:GlobReg}.

\begin{corollary}
\label{cor:Proofs:Back:DisReg:GlobReg}
Global regularity of functions.
Assume Assumptions~\ref{ass:Main:Ass:S1}, \ref{ass:Main:Ass:M1}, \ref{ass:Main:Ass:M2}, \ref{ass:Main:Ass:D1}, \ref{ass:Main:Ass:W1}, \ref{ass:Main:Ass:W2} and~\ref{ass:Main:Ass:L1} hold, $\eps_n$ satisfies~\eqref{eq:Main:Ass:epsLBWellPosed}, and $\rho\in\Ck{2}$.
Let $K_n$ be as in Proposition \ref{prop:Proofs:Back:DisReg:EVecAlt} and assume that $\Vert \psi_{n,k} \Vert_{\Lp{\infty}} \leq C_{\psi} \lambda_{k}^{\alpha}$ for all $k \leq \ceil{K_n}$ and some $\alpha > 0$. Let $s > 2\alpha + 2 + d/2$.
Then, there exists $C>0$ such that, $\bbP$-a.s., for $n$ sufficiently large we have
\begin{equation} \label{eq:Proofs:Back:DisReg:GlobReg}
\vert u_n(x_i) - u_n(x_j)  \vert \leq C \l \sqrt{\EnergySnepsn(u_n)} + \Vert u_n \Vert_{\Lp{2}} \r \dTorus(x_i,x_j) + C \sqrt{\EnergySnepsn(u_n)} n \eps_n^{s/2 -1/2} \dTorus(x_i,x_j)
\end{equation} 
for all $x_i,x_j\in\Omega_n$ and any $u_n:\Omega_n\to \bbR$.
\end{corollary}

\begin{proof}
In the proof $C>0$ will denote a constant that can be arbitrarily large, independent of $n$ that may change from line to line. Depending on context, we will write $z$ for the point $z \in \bbR^d$ and for its equivalence class $z + \mathbb{Z}^d \in \bbR^d/\mathbb{Z}^d$.
With probability one, we can assume that the conclusions of Theorem~\ref{thm:Back:TLp:LinftyMapsRate} and Theorem~\ref{thm:Proofs:Back:DisReg:DisRegAlt} hold. 

For $z \in \mathbb{R}^d$, let $B(z,\delta)$ and $B_{\dTorus}(z,\delta)$ be the balls of radius $\delta > 0$ centered at $z$ respectively in the Euclidean and manifold metric. Then, for $\alpha > 0$ and $n$ large enough, 
we recall the following density result from \cite[Lemma 4.1]{Slepcev}:
\begin{equation} \label{eq:Proofs:Back:DisReg:GlobReg:DensityAlt}
\mu_n \l B \l z,\frac{\alpha\eps_n}{2} \r\r > 0.
\end{equation}
By assumption \ref{ass:Main:Ass:S1}, for $\delta \ll 1/2$, for all $w \in B_{\dTorus}(z,\delta)$ there exists $w^* \in w + \mathbb{Z}^d$ with $w^* \in B(z,\delta)$ such that $\dTorus(z,w) = \Vert z - w^* \Vert$. Conversely, for all $w \in B(z,\delta)$, $w + \mathbb{Z}^d \in B_{\dTorus}(z,\delta)$ and $\dTorus(z,w) = \Vert z - w \Vert$ implying that \eqref{eq:Proofs:Back:DisReg:GlobReg:DensityAlt} holds for $B_{\dTorus}$ as well.

Let $x_i,x_j \in \Omega_n$, choose $\gamma > 0$ and let $n$ be large enough so that $\gamma \eps_n \ll 1/2$, \eqref{eq:Proofs:Back:DisReg:GlobReg:DensityAlt} and  \eqref{eq:Proofs:Back:DisReg:DisReg} hold. Let $x_i^* \in x_i + \mathbb{Z}^d \subseteq \mathbb{R}^d$ and $x_j^* \in x_j + \mathbb{Z}^d \subseteq \mathbb{R}^d$ be such that $\Vert x_i^* - x_j^* \Vert = \dTorus(x,y)$ and $r$ the straight path between $x_i^*$ and $x_j^*$. Define $C_n = \lceil \frac{2\Vert x_i^* - x_j^* \Vert}{\gamma \eps_n} \rceil$. Then, there exist points $\{x^{(i)}\}_{i=1}^{C_n+1}$ on $r$ with $x^{(1)} =x_i$, $x^{(C_n + 1)} = x_j$ and $\dTorus(x^{(i+1)},x^{(i)}) = \Vert x^{(i+1)} - x^{(i)} \Vert \leq \frac{\gamma \eps_n}{2}$ for $i = 1, \cdots, C_n$. Using the density result \eqref{eq:Proofs:Back:DisReg:GlobReg:DensityAlt}, we find points $\{x^{(i)}_{\Omega_n}\}_{i=2}^{C_n}$ such that $x^{(i)}_{\Omega_n} \in \Omega_n$ and $x^{(i)}_{\Omega_n} \in B\l x^{(i)}, \frac{\gamma \eps_n}{4} \r = B_{\dTorus}\l x^{(i)}, \frac{\gamma \eps_n}{4}\r$. Define $x^{(1)}_{\Omega_n} = x_i$ and $x_{\Omega_n}^{(C_n + 1)} = x_j$ and note that $\dTorus \l x_{\Omega_n}^{(i+1)},x_{\Omega_n}^{(i)}\r = \Vert x_{\Omega_n}^{(i+1)} - x_{\Omega_n}^{(i)} \Vert \leq \gamma \eps_n$. We can estimate as follows: 
\begin{align}
    \vert u_n(x_i^*) - u_n(x_j^*) \vert &\leq \sum_{i = 1}^{C_n} \vert u_n\l x^{(i)}_{\Omega_n} \r - u_n \l x_{\Omega_n}^{(i+1)} \r \vert \notag\\
    &\leq C \sum_{i=1}^{C_n} (\dTorus \l x_{\Omega_n}^{(i)}, x_{\Omega_n}^{(i+1)} \r + \eps_n ) \l \sqrt{\EnergySnepsn(u_n)} + \Vert u_n \Vert_{\Lp{2}} \r \label{eq:Proofs:Back:DisReg:GlobReg:Ineq2Alt} \\
& \qquad + C \sum_{i=1}^{C_n} \gamma n \eps_n^{s/2 + 1/2} \l  \EnergySnepsn(u_n)\r^{1/2} \notag\\
    &\leq C \l \sqrt{\EnergySnepsn(u_n)} + \Vert u_n \Vert_{\Lp{2}} \r C_n \eps_n + C \l  \EnergySnepsn(u_n) \r^{1/2} C_n n \eps_n^{s/2 + 1/2} \notag\\
    &\leq C \l \sqrt{\EnergySnepsn(u_n)} + \Vert u_n \Vert_{\Lp{2}} \r \Vert x_i^* - x_j^* \Vert \notag\\
    & \qquad + C  \sqrt{\EnergySnepsn(u_n)} n \eps_n^{s/2 - 1/2} \Vert x_i^*-x_j^* \Vert \notag
\end{align}
where we used \eqref{eq:Proofs:Back:DisReg:DisReg} for \eqref{eq:Proofs:Back:DisReg:GlobReg:Ineq2Alt}. Recalling $\Vert x_i^* - x_j^* \Vert = \dTorus(x_i,x_j)$ completes the proof of the proposition.
\end{proof}

\begin{remark}
Tail of the eigenvalues.
From the proof of Theorem \ref{thm:Proofs:Back:DisReg:DisRegAlt}, it is apparent that it is the lack of control on the tail of the eigenvalues, i.e. on $\{\lambda_{n,k}\}_{k \geq \lfloor K_n \rfloor + 1}$, that induces the term 
\[C \sqrt{\EnergySnepsn(u_n)} n \eps_n^{s/2 -1/2} \dTorus(x_i,x_j)
\]
in \eqref{eq:Proofs:Back:DisReg:GlobReg}. In turn, it is the latter that will imply the upper bound on $\eps_n$ as we will require $n \eps_n^{s/2 -1/2}$ to be bounded for Proposition \ref{prop:Proofs:Compactness:L2Compactness}.

In order to circumvent this, one could imagine only considering functions in the set $\mathrm{Tr}(n) = \{u:\Omega_n \to \bbR \spaceBar \langle u, \psi_{n,k} \rangle_{\Lp{2}(\Omega_n)} = 0 \text{ for $k \geq \lfloor K_n \rfloor + 1$}  \}$ and substitute the truncated energy $\EnergySnepsntrunc(u_n) = \sum_{k=1}^{\lfloor K_n \rfloor} \lambda_{n,k}^s \langle u, \psi_{n,k} \rangle_{\Lp{2}(\Omega_n)}^2$ for $\EnergySnepsn(\cdot)$. The authors believe the analysis to be analogous to the one presented here besides a few changes in Proposition \ref{prop:Proofs:GConvergence:WellPosed:limsup}. The particularly nice feature of this truncated problem is that $\eps_n$ would no longer have an upper bound: we expect to be in the well-posed regime whenever $s > 2\alpha + 2 + d/2$ (with the sharp bound still being $s>d/2$ as explained in Remark \ref{rem:Main:Res:Relationship}) and in the ill-posed regime when $s < d/2$.
\end{remark}

\subsection{Compactness} \label{subsec:compacteness}

In order to show compactness of our discrete functions, we use the following regularity result for a mollification of our discrete function extended to the continuum domain.


\begin{lemma} \label{lem:Proofs:Compactness:RegAlt}
Regularity of mollified sequences.
Assume Assumptions~\ref{ass:Main:Ass:S1}, \ref{ass:Main:Ass:M1}, \ref{ass:Main:Ass:M2}, \ref{ass:Main:Ass:D1}, \ref{ass:Main:Ass:W1}, \ref{ass:Main:Ass:W2} and~\ref{ass:Main:Ass:L1} hold, $\eps_n$ satisfies~\eqref{eq:Main:Ass:epsLBWellPosed}, and $\rho\in\Ck{2}$.
Let $K_n$ be as in Proposition \ref{prop:Proofs:Back:DisReg:EVecAlt} and assume that $\Vert \psi_{n,k} \Vert_{\Lp{\infty}} \leq C_{\psi} \lambda_{k}^{\alpha}$ for all $k \leq \ceil{K_n}$ and some $\alpha > 0$. Let $s > 2\alpha + 2 + d/2$.
Let $J$ be a radially symmetric, positive mollifier supported in the unit ball with $\Vert J \Vert_{\Lp{1}} = 1$ and write $J_{\eps_n}(\cdot) = \eps_n^{-d}J(\cdot / \eps_n)$. For any $u_n: \Omega_n \mapsto \bbR$, let $\tilde{u}_n = J_\epsilon \ast (u_n \circ T_n)$ where $\{T_n\}_{n=1}^\infty$ are the transport maps from Theorem~\ref{thm:Back:TLp:LinftyMapsRate}. Then, $\bbP$-a.s., for $n$ sufficiently large we have
\begin{align}
\vert \tilde{u}_n(x) - \tilde{u}_n(y)  \vert &\leq C \l \sqrt{\EnergySnepsn(u_n)} + \Vert u_n \Vert_{\Lp{2}} \r \l 2 \Vert \Id-T_n \Vert_{\Lp{\infty}} + \dTorus(x,y)  \r \label{eq:Proofs:Compactness:Reg}\\
&+ \sqrt{\EnergySnepsn(u_n)} n \eps_n^{s/2 -1/2} \l 2 \Vert \Id-T_n \Vert_{\Lp{\infty}} + \dTorus(x,y) \r \notag
\end{align}
for all $x, y \in \Omega$.
\end{lemma}

\begin{proof}
In the proof $C>0$ will denote a constant that can be arbitrarily large, independent of $n$ that may change from line to line.
With probability one, we can assume that the conclusion of
Corollary~\ref{cor:Proofs:Back:DisReg:GlobReg} holds.

For any $x,y,z \in \Omega$ we have, 
\begin{equation} \label{eq:Proofs:Compactness:Reg:TranportDiff}
     \dTorus(T_n(x - z),T_n(y - z)) \leq 2 \Vert \Id-T_n \Vert_{\Lp{\infty}} + \dTorus(x,y).
\end{equation}
Let $n$ be large enough so that \eqref{eq:Proofs:Back:DisReg:GlobReg} holds. Then, we can estimate as follows:
\begin{align}
    \vert \tilde{u}_n(x) - \tilde{u}_n(y) \vert \notag &= \left| \int J_{\eps_n}(z) [u_n \circ T_n (x - z ) - u_n \circ T_n(y - z)] \, \dd z \right| \notag\\
    &\leq C \int J_{\eps_n}(z) \l \sqrt{\EnergySnepsn(u_n)} + \Vert u_n \Vert_{\Lp{2}} \r \l 2 \Vert \Id-T_n \Vert_{\Lp{\infty}} + \dTorus(x, y) \r \, \dd z  \label{eq:Proofs:Compactness:Reg:Ineq1} \\
    &\qquad + C \int J_{\eps_n}(z) \sqrt{\EnergySnepsn(u_n)} n \eps_n^{s/2 - 1/2} \l 2 \Vert \Id-T_n \Vert_{\Lp{\infty}} + \dTorus(x,y) \r \, \dd z \notag\\
    &\leq C \l \sqrt{\EnergySnepsn(u_n)} + \Vert u_n \Vert_{\Lp{2}} \r \l 2 \Vert \Id-T_n \Vert_{\Lp{\infty}} + \dTorus(x,y)  \r \notag\\
    &\qquad+ \sqrt{\EnergySnepsn(u_n)} n \eps_n^{s/2 - 1/2} \l 2 \Vert \Id-T_n \Vert_{\Lp{\infty}} + \dTorus( x,y) \r \notag
\end{align}
where we used \eqref{eq:Proofs:Back:DisReg:GlobReg} and \eqref{eq:Proofs:Compactness:Reg:TranportDiff} for \eqref{eq:Proofs:Compactness:Reg:Ineq1}.
\end{proof}


We will use a variant of the Ascoli-Arzel\`a theorem in order to prove compactness of the mollified discrete functions in the $\Lp{\infty}$-norm. 
We state the result in the setting of Assumption~\ref{ass:Main:Ass:S1} (i.e. on the torus) but the result is true in more general compact manifolds and sets.

\begin{theorem} \label{thm:Proofs:Compactness:ArzelaAscoli}
Asymptotic Ascoli-Arzel\`a.
Let $\Omega$ be the unit torus and $\{u_n:\Omega \mapsto \mathbb{R}\}_{n=1}^\infty$ be a set of continuous functions such that
\begin{enumerate}
    \item $\sup_{n \in \mathbb{N}} \Vert u_n \Vert_{\Lp{\infty}} < \infty$;
    \item for all $\bar{\epsilon} > 0$, there exists $\delta > 0$ and $N$ such that, if $\dTorus(x, y) < \delta $ and $n \geq N$:
    \[
        \vert u_n(x) - u_n(y) \vert < \bar{\epsilon}.
    \]
\end{enumerate}
Then, there exists a continuous function $u:\Omega \mapsto \bbR$ and a subsequence $\{u_{n_k}\}_{k=1}^\infty$ such that $\Vert u_{n_k} - u \Vert_{\Lp{\infty}} \to 0 $ as $k \to \infty$. 
\end{theorem}

\begin{proof}
Let $\bar{\epsilon} > 0$. Then, by assumption, there exists $\delta > 0$ and $N$ such that $\dTorus(x,y) < \delta $ and $n \geq N$ implies that
\[ \vert u_n(x) - u_n(y) \vert < \bar{\epsilon}. \]
For all $i < N$, since the functions $u_i$ are uniformly continuous, there exists $\delta_i > 0$ such that $\dTorus(x,y) < \delta_i$ implies that $\vert u_i(x) - u_i(y) \vert < \bar{\epsilon}$. Hence, if we define $\bar{\delta} = \min\{\delta_1,\hdots,\delta_{N-1},\delta\}$, then for all $n$, if $\dTorus(x,y) < \bar{\delta}$ we have $\vert u_n(x) - u_n(y) \vert < \bar{\epsilon}$.

This means that the set $\{u_n\}_{n=1}^\infty$ is uniformly bounded and equicontinuous and, by the Ascoli-Arzel\`a theorem \cite[Theorem 8.2.10]{engelking1989general}, we deduce the existence of a continuous function $u:\Omega \mapsto \bbR$ and a subsequence $\{u_{n_k}\}_{k=1}^\infty$ such that $\Vert u_{n_k} - u \Vert_{\Lp{\infty}} \to 0 $.
\end{proof}

We are now able to prove compactness of our discrete functions. 


\begin{proposition} \label{prop:Proofs:Compactness:L2Compactness}
$\TLp{2}$ and uniform compactness.
Assume Assumptions~\ref{ass:Main:Ass:S1}, \ref{ass:Main:Ass:M1}, \ref{ass:Main:Ass:M2}, \ref{ass:Main:Ass:D1}, \ref{ass:Main:Ass:W1}, \ref{ass:Main:Ass:W2} and~\ref{ass:Main:Ass:L1} hold, $\eps_n$ satisfies~\eqref{eq:Main:Ass:epsLBWellPosed}, and $\rho\in\Ck{2}$.
Let $K_n$ be as in Proposition \ref{prop:Proofs:Back:DisReg:EVecAlt} and assume that $\Vert \psi_{n,k} \Vert_{\Lp{\infty}} \leq C_{\psi} \lambda_{k}^{\alpha}$ for all $k \leq \ceil{K_n}$ and some $\alpha > 0$. Let $s > 2\alpha + 2 + d/2$ and assume that $n\eps_n^{s/2 - 1/2}$ is bounded.
Let $\{u_n:\Omega_n \mapsto \mathbb{R}\}_{n=1}^\infty$ be a set of functions with $\sup_n \Vert u_n \Vert_{\Lp{\infty}(\mu_n)} < M$ and $\sup_n \EnergySnepsn(u_n) < M$ for some constant $M$. Then, $\mathbb{P}$-a.s., there exists a continuous function $u : \Omega \mapsto \bbR$ and a subsequence $\{u_{n_k}\}_{k=1}^\infty$ such that 
\begin{equation} \label{eq:Proofs:Compactness:L2Compactness}
\max_{i = 1, \dots, n_k} \vert u_{n_k}(x_i) - u(x_i) \vert \to 0
\end{equation}
and $(\mu_{n_k},u_{n_k})$ converges to $(\mu,u)$ in $\TLp{2}$.
\end{proposition}

\begin{proof}
In the proof $C>0$ will denote a constant that can be arbitrarily large, independent of $n$ that may change from line to line.
With probability one, we can assume that the conclusion of
Lemma~\ref{lem:Proofs:Compactness:RegAlt} holds.
Let the functions $\tilde{u}_n = J_{\eps_{n}} \ast (u_n \circ T_n)$ be defined as in Lemma \ref{lem:Proofs:Compactness:RegAlt}. By assumption on $\{u_n\}_{n=1}^\infty$ we have
\[
       \Vert \tilde{u}_n \Vert_{\Lp{\infty}} = \frac{1}{\eps_n^d} \Vert J(\cdot/\eps_n) \ast (u_n \circ T_n) \Vert_{\Lp{\infty}} \leq \frac{1}{\eps_n^d} \Vert J(\cdot/\eps_n) \Vert_{\Lp{1}} \Vert u_n \circ T_n \Vert_{\Lp{\infty}(\mu)} <C.
\]
Moreover, using \eqref{eq:Proofs:Compactness:Reg} and the fact that $n\eps_n^{s/2 - 1/2}$ is a bounded sequence, for any  $x, y \in \Omega$ we obtain, for $n$ sufficiently large:
\begin{align}
\vert \tilde{u}_n(x) - \tilde{u}_n(y)  \vert &\leq C \l\Vert \Id-T_n \Vert_{\Lp{\infty}} + \dTorus(x,y)\r (1 + n \eps_n^{s/2-1/2}) \notag \\
&\leq C \l\Vert \Id-T_n \Vert_{\Lp{\infty}} + \dTorus( x,y) \r.  \label{eq:Proofs:Compactness:Linfty:mollifierReg}
\end{align}
Let $\bar{\eps} > 0$. By the estimates in Theorem \ref{thm:Back:TLp:LinftyMapsRate}, for $n$ sufficiently large, we have $C(\Vert \Id-T_n \Vert_{\Lp{\infty}}) < \bar{\eps}/2$. Let $x,y$ be such that $\dTorus(x,y) < \bar{\eps}/2C =: \delta$. Inserting the latter two estimates in \eqref{eq:Proofs:Compactness:Linfty:mollifierReg}, we have
\[
    \vert \tilde{u}_n(x_i) - \tilde{u}_n(x_j)  \vert < \bar{\eps}.
\]
We now apply Theorem \ref{thm:Proofs:Compactness:ArzelaAscoli} to deduce the existence of a subsequence $\{\tilde{u}_{n_k}\}_{k=1}^\infty$ and a continuous function $u:\Omega \mapsto \bbR$ such that \begin{equation}
   \Vert \tilde{u}_{n_k} - u \Vert_{\Lp{\infty}} \to 0.
    \label{eq:Proofs:Compactness:Linfty:mollifierConvergence}
\end{equation}
For any $i \leq n_k$, we can write 
\[
\vert u_{n_k}(x_i) - u(x_i) \vert \leq \vert u_{n_k}(x_i) - \tilde{u}_{n_k}(x_i) \vert + \vert \tilde{u}_{n_k}(x_i) - u(x_i) \vert =: A + B.
\]
By \eqref{eq:Proofs:Compactness:Linfty:mollifierConvergence}, we know that $B$ tends to 0. Let $k$ be large enough so that~\eqref{eq:Proofs:Back:DisReg:GlobReg} holds for $u_{n_k}$. For the term $A$, we estimate as follows: 
\begin{align}
    A &= \left| u_{n_k}(x_i) - \int J_{\eps_n}(x_i - y) u_{n_k}(T_{n_k}(y)) \,\dd y \right| \notag \\
    &\leq \int J_{\eps_n}(x_i - y) \left|  u_{n_k}(T_{n_k}(y)) - u_{n_k}(x_i)   \right| \, \dd y \notag \\
    &\leq C \int J_{\eps_n}(x_i - y) \dTorus(T_{n_k}(y),x_i) \, \dd y \label{eq:Proofs:Compactness:Linfty:Ineq1} \\
    & \leq C\int J_{\eps_n}(x_i-y) \l \dTorus(T_{n_k}(y),y) + \dTorus(y,x_i)\r \notag \\
    &\leq C \l \Vert T_n - \Id \Vert_{\Lp{\infty}} +\eps_n\r \notag
\end{align}
where we used \eqref{eq:Proofs:Back:DisReg:GlobReg} and the fact that $n\eps_n^{s/2-1/2}$ is bounded for \eqref{eq:Proofs:Compactness:Linfty:Ineq1}. By Theorem \ref{thm:Back:TLp:LinftyMapsRate}, we deduce that $A$ tends to 0 and obtain the claim of the proposition.
\end{proof}

\subsection{\texorpdfstring{$\Gamma$}{Gamma} Convergence}

We now proceed to prove the $\Gamma$-convergence for our energy functionals.
Our proofs use the following $\Gamma$-convergence result.

\begin{theorem}\label{thm:Proofs:GConvergence:withoutConstraints}
$\Gamma$-convergence without constraints.
Assume that Assumptions \ref{ass:Main:Ass:S1}, \ref{ass:Main:Ass:M1}, \ref{ass:Main:Ass:M2}, \ref{ass:Main:Ass:D1}, \ref{ass:Main:Ass:W1}, \ref{ass:Main:Ass:W2} and \ref{ass:Main:Ass:L1} hold. Furthermore, assume that $\eps_n$ satisfies \eqref{eq:Main:Ass:epsLBIllPosed}. Then, $\bbP$-a.s., we have
\begin{enumerate}
    \item $\EnergySnepsn{}$ $\Gamma$-converges to $\EnergyS$;
    \item Any sequence of functions $\{u_n:\Omega_n \mapsto \bbR\}$ with $\sup_{n\in \bbN} \Vert u_n \Vert_{\Lp{2}} \leq C $ and $\sup_{n \in \bbN} \EnergySnepsn(u_n) \leq C$ is pre-compact in the $\TLp{2}$ topology.
\end{enumerate}
\end{theorem}

We will split the $\Gamma$ convergence proofs in two separate propositions for convenience.

\subsubsection{Well-Posed Case} 

The $\liminf$-inequality is a consequence of the compactness Proposition~\ref{prop:Proofs:Compactness:L2Compactness} and the following Poincaré inequality which is similar to \cite[Lemma 3.4]{DBLP:journals/corr/abs-1909-10221}.

\begin{lemma} \label{lem:Proofs:GConvergence:IllPosed:Minkowski}
Minkowski inequality for the discrete energy.
Assume Assumptions~\ref{ass:Main:Ass:S1}, \ref{ass:Main:Ass:M1}, \ref{ass:Main:Ass:M2}, \ref{ass:Main:Ass:D1} hold.
Then, we have \[
    \sqrt{\EnergySnepsn(u+v)} \leq \sqrt{\EnergySnepsn(u)} + \sqrt{\EnergySnepsn(v)}
\]
for any $u,v:\Omega_n \mapsto \mathbb{R}$ and $n \in \bbN$.
\end{lemma}

\begin{proof}
The proof is a computation:\begin{align*}
    \EnergySnepsn(u+v) &= \sum_{k=1}^n \lambda_{n,k}^s \langle u + v, \psi_{n,k} \rangle^2\\
    &= \sum_{k=1}^n \lambda_{n,k}^{s} \langle u , \psi_{n,k} \rangle \langle u + v, \psi_{n,k} \rangle + \sum_{k=1}^n \lambda_{n,k}^{s} \langle v , \psi_{n,k} \rangle \langle u + v, \psi_{n,k} \rangle\\
    &\leq \sqrt{\sum_{k=1}^n \lambda_{n,k}^s \langle u , \psi_{n,k} \rangle^2 } \sqrt{\sum_{k=1}^n \lambda_{n,k}^{s} \langle u+v , \psi_{n,k} \rangle^2} \\
    & \qquad \qquad + \sqrt{\sum_{k=1}^n \lambda_{n,k}^s \langle v , \psi_{n,k} \rangle^2 } \sqrt{\sum_{k=1}^n \lambda_{n,k}^{s} \langle u+v , \psi_{n,k} \rangle^2}\\
    &= \sqrt{\EnergySnepsn(u+v)}\left( \sqrt{\EnergySnepsn(u)} + \sqrt{\EnergySnepsn(v)} \right). \qedhere
\end{align*}
\end{proof}

\begin{remark}
Minkowski inequality for the continuum energy.
We note that the proof of Lemma \ref{lem:Proofs:GConvergence:IllPosed:Minkowski} can equally be applied to prove the same inequality for $\EnergyS(\cdot)$ on the set $\mathcal{H}^s$.
\end{remark}

\begin{proposition} \label{prop:Proofs:Bounds:Poincaré}
Discrete Poincaré inequality.
Assume Assumptions~\ref{ass:Main:Ass:S1}, \ref{ass:Main:Ass:M1}, \ref{ass:Main:Ass:M2}, \ref{ass:Main:Ass:D1}, \ref{ass:Main:Ass:W1}, \ref{ass:Main:Ass:W2} and~\ref{ass:Main:Ass:L1} hold, $\eps_n$ satisfies~\eqref{eq:Main:Ass:epsLBWellPosed}, and $\rho\in\Ck{2}$.
Let $K_n$ be as in Proposition \ref{prop:Proofs:Back:DisReg:EVecAlt} and assume that $\Vert \psi_{n,k} \Vert_{\Lp{\infty}} \leq C_{\psi} \lambda_{k}^{\alpha}$ for all $k \leq \ceil{K_n}$ and some $\alpha > 0$. Let $s > 2\alpha + 2 + d/2$ and assume that $n\eps_n^{s/2 - 1/2}$ is bounded.
For a function $u:\Omega \mapsto \mathbb{R}$, we define $\bar{u} = \frac{1}{N}\sum_{i=1}^N u(x_i)$. Then, there exists $C >0$ such that, $\mathbb{P}$-a.e, we have
\begin{equation}
    \Vert u_n - \bar{u}_n \Vert_{\Lp{\infty}(\mu_n)} \leq C \sqrt{\EnergySnepsn(u_n)}
    \label{eq:Proofs:Bounds:Poincaré}
\end{equation}
for all $n$ and any $u_n:\Omega_n \mapsto \mathbb{R}$.
\end{proposition}

\begin{proof}
In the proof $C>0$ will denote a constant that can be arbitrarily large, independent of $n$ that may change from line to line. With probability one, we can assume that the conclusions of Proposition \ref{prop:Proofs:Compactness:L2Compactness}, Theorem \ref{thm:Proofs:GConvergence:withoutConstraints} hold.

Assume that \eqref{eq:Proofs:Bounds:Poincaré} does not hold. Then, there exists a sequence $\{n_m\}_{m=1}^\infty \subseteq \mathbb{N}$ and functions $\{u_{n_m}:\Omega_{n_m} \mapsto \mathbb{R}\}_{m=1}^\infty$ such that \begin{equation}
    \Vert u_{n_m} - \bar{u}_{n_m} \Vert_{\Lp{\infty}(\mu_{n_m})} > m \sqrt{\mathcal{E}^{(s)}_{n_m,\eps_{n_m}}(u_{n_m})}.
    \label{eq:Proofs:Bounds:Poincaré:contradiction}
\end{equation}

Define 
\[
v_{n_m} = \frac{u_{n_m} - \bar{u}_{n_m}}{\Vert u_{n_m} - \bar{u}_{n_m} \Vert_{\Lp{\infty}(\mu_{n_m})}}
\]
and note that $\Vert v_{n_m} \Vert_{\Lp{\infty}(\mu_{n_m})} = 1$ as well as 
\[
    \bar{v}_{n_m} 
    = \frac{\bar{v}_{n_m} - \bar{v}_{n_m}} {\Vert u_{n_m} - \bar{u}_{n_m} \Vert_{\Lp{\infty}(\mu_{n_m})}} = 0.
\]
Using Lemma \ref{lem:Proofs:GConvergence:IllPosed:Minkowski} and \eqref{eq:Proofs:Bounds:Poincaré:contradiction}, we furthermore obtain
\[
    \sqrt{\mathcal{E}_{n_m}^{(s)}(v_{n_m})} \leq \frac{1}{\Vert u_{n_m} - \bar{u}_{n_m} \Vert_{\Lp{\infty}(\mu_{n_m})}} \left(\sqrt{ \mathcal{E}^{(s)}_{n_m}(u_{n_m})} + \sqrt{\mathcal{E}^{(s)}_{n_m}(\bar{u}_{n_m})}\right)
    < \frac{1}{m}.
\]
Hence, we can apply Proposition \ref{prop:Proofs:Compactness:L2Compactness} to deduce the existence of a subsequence $\{n_{m_k}\}_{k=1}^\infty$ and a continuous function $v$ such that $ \max_{i \leq n_{m_k}} \vert v_{n_{m_k}}(x_i) - v(x_i) \vert \to 0$.

We estimate as follows:
\begin{align*}
    \vert \Vert v_{n_{m_k}} \Vert_{\Lp{\infty}(\mu_{n_{m_k}})} - \Vert v \Vert_{\Lp{\infty}(\mu)} \vert &= \vert \Vert v_{n_{m_k}} \circ T_{n_{m_k}} \Vert_{\Lp{\infty}(\mu)} - \Vert v \Vert_{\Lp{\infty}(\mu)} \vert \\
    &\leq \Vert v_{n_{m_k}} \circ T_{n_{m_k}} - v \Vert_{\Lp{\infty}(\mu)} \\
    &\leq \Vert v_{n_{m_k}} \circ T_{n_{m_k}} - v \circ T_{n_{m_k}} \Vert_{\Lp{\infty}(\mu)} + \Vert v \circ T_{n_{m_k}} - v\Vert_{\Lp{\infty}(\mu)} \\
    &=: A + B.
\end{align*}
The $A$ term tends to 0 by \eqref{eq:Proofs:Compactness:L2Compactness}.
For the $B$ term, let $\bar{\epsilon} > 0$. Since $v$ is a uniformly continuous function, there exists $\delta > 0$ such that $\vert x - y \vert < \delta$ implies $\vert v(x) - v(y) \vert < \bar{\epsilon}$.  By Theorem~\ref{thm:Back:TLp:LinftyMapsRate}, we know that $\Vert \Id - T_n \Vert_{\Lp{\infty}} \to 0$ and hence, there exists a $n_0$ such that for $n \geq n_0$, $\Vert \Id - T_n \Vert_{\Lp{\infty}} < \delta$. Therefore, for $n_{m_k} \geq n_0$, we have $\vert v(T_{n_{m_k}}(x)) - v(x) \vert < \bar{\epsilon}$ which implies that
$B$ tends to 0. 
Combining the latter, we obtain $\Vert v \Vert_{\Lp{\infty}(\mu)} = \lim_{k \to \infty} \Vert v_{n_{m_k}} \Vert_{\Lp{\infty}(\mu_{n_{m_k}})} = 1.$
Furthermore,
\[
    \vert \bar{v} - \bar{v}_{n_{m_k}} \vert \leq \frac{1}{N} \sum_{i=1}^N \vert v(x_i) - v_{n_{m_k}}(x_i) \vert \leq \max_{i \leq n_{m_k}} \vert v(x_i) - v_{n_{m_k}}(x_i) \vert
\]
which tends to 0 by \eqref{eq:Proofs:Compactness:L2Compactness} and we have $\bar{v} = \lim_{k \to \infty} \bar{v}_{n_{m_k}} = 0$.

Finally, by an application of Theorem \ref{thm:Proofs:GConvergence:withoutConstraints}, we obtain 
\[
    0 = \liminf_{k \to \infty} \frac{1}{m_k^2} > \liminf_{k \to \infty} \mathcal{E}^{(s)}_{n_{m_k}} (v_{n_{m_k}}) \geq \EnergyS(v).
\]
We conclude that $v$ has to be a constant function (as with probability one the graph is connected) and, since $\bar{v} = 0$, $v = 0$ which contradicts $\Vert v \Vert_{\Lp{\infty}} = 1$.

\end{proof}

\begin{proposition} \label{prop:Proofs:GConvergence:WellPosed:liminf}
$\liminf$-inequality in the well-posed case. 
Assume Assumptions~\ref{ass:Main:Ass:S1}, \ref{ass:Main:Ass:M1}, \ref{ass:Main:Ass:M2}, \ref{ass:Main:Ass:D1}, \ref{ass:Main:Ass:W1}, \ref{ass:Main:Ass:W2} and~\ref{ass:Main:Ass:L1} hold, $\eps_n$ satisfies~\eqref{eq:Main:Ass:epsLBWellPosed}, and $\rho\in\Ck{2}$.
Let $K_n$ be as in Proposition \ref{prop:Proofs:Back:DisReg:EVecAlt} and assume that $\Vert \psi_{n,k} \Vert_{\Lp{\infty}} \leq C_{\psi} \lambda_{k}^{\alpha}$ for all $k \leq \ceil{K_n}$ and some $\alpha > 0$. Let $s > 2\alpha + 2 + d/2$ and assume that $n\eps_n^{s/2 - 1/2}$ is bounded.
Then, $\bbP$-a.s., we have 
\begin{equation} \label{eq:Proofs:GConvergence:WellPosed:liminf}
    \liminf_{n \to \infty} \cF_{n,\eps_n}((\nu_n,v_n)) \geq \cF(\nu,v)
\end{equation} 
for any $(\nu,v) \in \TLp{2}(\Omega)$ and $\{(\nu_n,v_n)\}_{n=1}^\infty \subseteq \TLp{2}(\Omega)$ such that $(\nu_n, v_n) \to (\nu, v)$ in $\TLp{2}$. 
\end{proposition}

\begin{proof}
With probability one, we can assume that the conclusions of Proposition ~\ref{prop:Proofs:Compactness:L2Compactness} and Theorem \ref{thm:Proofs:GConvergence:withoutConstraints} hold.

We start by noting that if $\liminf_{n \to \infty} \mathcal{F}_{n,\eps_n}((\nu_n,v_n)) = \infty$, then \eqref{eq:Proofs:GConvergence:WellPosed:liminf} is trivial. We therefore assume that $\liminf_{n \to \infty} \mathcal{F}_{n,\eps_n}((\nu_n,v_n)) < \infty$ and hence, there exists a subsequence $\{n_k\}_{k=1}^\infty$ such that 
\[
\lim_{k \to \infty} \mathcal{F}_{n_k,\eps_{n_k}}((\nu_{n_k},v_{n_k})) = \liminf_{n \to \infty} \mathcal{F}_{n,\eps_n}((\nu_n,v_n)) < \infty.
\]
Hence, \eqref{eq:Proofs:GConvergence:WellPosed:liminf} is equivalent to showing 
\[
    \liminf_{n \to \infty} \mathcal{F}_{n,\eps_n}((\nu_n,v_n)) = \lim_{k \to \infty} \mathcal{F}_{n_k,\eps_{n_k}}((\nu_{n_k},v_{n_k})) \geq \mathcal{F}((\nu,v)).
\]
In particular, this shows that we might assume that $\{\mathcal{F}_{n,\eps_n}((\nu_n,v_n))\}_{n=1}^\infty$ is uniformly bounded by a constant~$C$. By~\eqref{eq:Main:Not:Fn}, this also means that $\nu_n = \mu_n$ for all $n \in \bbN$ and that for $i \leq N$, $v_n(x_i) = \ell_i$. Consequently, we have 
\[
    \mathcal{F}_{n,\eps_n}((\nu_n,v_n)) = \EnergySnepsn(v_n)
\]
and, by Proposition \ref{prop:Proofs:Bounds:Poincaré},
\[ \Vert u_n \Vert_{\Lp{\infty}} \leq \Vert u_n - \frac{1}{N} \sum_{i=1}^N u_n(x_i) \Vert_{\Lp{\infty}} + \Vert \frac{1}{N} \sum_{i=1}^N u_n(x_i) \Vert_{\Lp{\infty}} \leq C \sqrt{\EnergySnepsn(u_n)} + \frac{1}{N} \sum_{i=1}^N \vert\ell_i\vert \leq C \]
which implies $\sup_{n\in\bbN}\Vert u_n \Vert_{\Lp{\infty}}<+\infty$.
By Proposition \ref{prop:Back:TLp}, $(\nu_n,v_n) \to (\nu,v)$ in $\TLp{2}$ implies that $\{\nu_n\}_{n=1}^\infty = \{\mu_n\}_{n=1}^\infty$ converges weakly to $\nu$. As weak limits are unique and $\{\mu_n\}_{n=1}^\infty$ converge weakly to $\mu$, we deduce that $\mu = \nu$. Hence, there exists transport maps $\{T_n\}_{n=1}^\infty$ from $\mu$ to $\mu_n$ satisfying the rates in Theorem \ref{thm:Back:TLp:LinftyMapsRate}.

We apply Proposition \ref{prop:Proofs:Compactness:L2Compactness} in order to deduce the existence of a continuous function $\hat{v}$ and a subsequence $\{n_k\}_{k=1}^\infty$ such that $\max_{i \leq n_k} \vert v_{n_k}(x_i) - \hat{v}(x_i) \vert \to 0$. In particular, this means that for $i \leq N$ and any $n_k$: 
\[
    \vert \hat{v}(x_i) - \ell_i \vert = \vert \hat{v}(x_i) - v_{n_k}(x_i) \vert \leq \max_{j \leq n_k} \vert \hat{v}(x_j) - v_{n_k}(x_j) \vert \to 0,
\]
from which we deduce that $\hat{v}(x_i) = \ell_i $ for all $i \leq N$. 

For any $n_k$, we estimate as follows: 
\begin{align}
    \int_\Omega \vert v(x) - \hat{v}(x) \vert^2 \, \dd\mu(x) &\leq 2\int_\Omega \vert v(x) - v_{n_k}(T_{n_k}(x))  \vert^2 \,\dd\mu(x) + 2 \int_\Omega \vert v_{n_k}(T_{n_k}(x)) - \hat{v}(T_{n_k}(x)) \vert^2 \, \dd\mu(x) \notag\\
    & \qquad + 2\int_\Omega \vert \hat{v}(T_{n_k}(x)) - \hat{v}(x) \vert^2 \,\dd\mu(x) \notag\\
    &=: 2(A + B + D). \notag
\end{align}
Note that $ B \leq 2\max_{i \leq n_k} \vert v_{n_k}(x_i) - \hat{v}(x_i) \vert^2$ and the latter tends to 0 by \eqref{eq:Proofs:Compactness:L2Compactness}. The $D$ term goes to 0 as shown in Proposition \ref{prop:Proofs:Bounds:Poincaré}.
Finally, $A$ tends to 0 by the assumption that $v_n \to v$ in $\TLp{2}$. Combining the latter three results, we deduce that $v = \hat{v} $ $\mu$-almost everywhere and since $\hat{v}$ is continuous then in fact $v=\hat{v}$ everywhere, in particular $v$ satisfies the constraints.
By Theorem \ref{thm:Proofs:GConvergence:withoutConstraints}

 \[ \liminf_{n \to \infty} \mathcal{F}_{n,\eps_n}((\nu_n,v_n)) = \liminf_{n \to \infty} \EnergySnepsn(v_n) \geq \EnergyS(v) \qedhere \] 
\end{proof}


The $\limsup$-inequality requires the following technical lemma.

\begin{lemma} \label{lem:Proofs:GConvergence:WellPosed:Bound}
Bounded energies.
Assume Assumptions~\ref{ass:Main:Ass:S1}, \ref{ass:Main:Ass:M1}, \ref{ass:Main:Ass:M2}, \ref{ass:Main:Ass:D1}, \ref{ass:Main:Ass:W1}, \ref{ass:Main:Ass:W2} and~\ref{ass:Main:Ass:L1} hold. 
Assume that $\rho \in \Ck{\infty}$ and that $\eps_n$ satisfies \eqref{eq:Main:Ass:epsLBWellPosed}. 
For $u \in \Ck{0}(\Omega)$, let $u_n$ be the restrictions of $u$ to $\Omega_n$.
Then, $\bbP$-a.s., for any $k\in \bbN$ and $u \in \Ck{\infty}(\Omega)$, there exists a constant $C=C(k,u) > 0$ such that  
\begin{equation} \label{eq:Proofs:GConvergence:WellPosed:Bound}
\sup_n \cE_{n,\eps_n}^{(2k)}(u_n) \leq C.
\end{equation}
\end{lemma}

\begin{proof}
In the proof $C>0$ will denote a constant that can be arbitrarily large, independent of $n$ that may change from line to line.

By Assumption \ref{eq:Main:Ass:epsLBWellPosed}, with probability one, we can assume that the conclusion of \cite[Theorem 5]{Calder_2018}, Theorem \ref{thm:Back:TLp:LinftyMapsRate} and Theorem \ref{thm:Proofs:GConvergence:withoutConstraints} hold.

We proceed to the proof by induction. Let $k = 0$. Then,  $\cE_{n,\eps_n}^{(0)} = \Vert u_n \Vert_{\Lp{2}(\Omega_n)} \leq \Vert u \Vert_{\Lp{\infty}}$ so that \eqref{eq:Proofs:GConvergence:WellPosed:Bound} is satisfied. We now assume that for any $u \in \Ck{\infty}(\Omega)$, \eqref{eq:Proofs:GConvergence:WellPosed:Bound} holds for all $l \leq 2k-2$. Recalling that the operator $\Delta_{n,\eps_n}$ is self-adjoint, we start by estimating as follows:
\begin{align}
    \cE_{n,\eps_n}^{(2k)}(u_n) &= \langle \Delta_{n,\eps_n}^k u, \Delta_{n,\eps_n}^k u \rangle_{\Lp{2}(\Omega_n)} \notag\\
    &\leq \langle \Delta_\rho^k u, \Delta_\rho^k u \rangle_{\Lp{2}(\Omega_n)} + \vert \langle \Delta_\rho^k u - \Delta_{n,\eps_n}^k u, \Delta_{n,\eps_n}^k u + \Delta_\rho^k u \rangle_{\Lp{2}(\Omega_n)} \vert \notag \\
    &\leq \langle \Delta_\rho^k u, \Delta_\rho^k u \rangle_{\Lp{2}(\Omega_n)} + \vert \langle \Delta_\rho^k u - \Delta_{n,\eps_n}^k u, \Delta_\rho^k u \rangle_{\Lp{2}(\Omega_n)} \vert \notag \\
    & \qquad + \vert \langle \Delta_\rho^k u - \Delta_{n,\eps_n}^k u, \Delta_{n,\eps_n}^k u \rangle_{\Lp{2}(\Omega_n)} \vert \notag \\
    &=: T_1 + \vert T_2 \vert + \vert T_3 \vert. \notag
\end{align}

Since $\rho \in \Ck{\infty}(\Omega)$, we have that $u\Delta_\rho^ku \in \Ck{\infty}$ and the latter is bounded on $\Omega$ by Assumption  \ref{ass:Main:Ass:S1}. Hence, as $\mu_n$ converges weakly to $\mu$, by the Portmanteau lemma \cite[Theorem 13.16]{klenke2013probability} we obtain
\[
T_1 = \int_\Omega u \Delta_\rho^{2k} u \, \dd\mu_n \to \int_\Omega u \Delta_\rho^{2k} u \, \dd\mu = \langle u, \Delta_\rho^{2k} u \rangle_{\Lp{2}(\Omega)}
\]
and so $T_1 \leq C$.

For the $T_2$ term, we note that:
\begin{align}
    T_2 &= \langle \Delta_\rho^k u , (\Delta_\rho - \Delta_{n,\eps_n}) \Delta_\rho^{k-1} u \rangle_{\Lp{2}(\mu_n)} + \langle \Delta_\rho^k u , \Delta_{n,\eps_n}(\Delta_\rho^{k-1} - \Delta_{n,\eps_n}^{k-1}) u \rangle_{\Lp{2}(\mu_n)} \label{eq:Proofs:GConvergence:WellPosed:Bound:T2}\\
    &= \sum_{i=0}^{l} \langle \Delta_\rho^k u , \Delta_{n,\eps_n}^i(\Delta_\rho - \Delta_{n,\eps_n}) \Delta_\rho^{k-1-i} u \rangle_{\Lp{2}(\mu_n)} + \langle \Delta_\rho^k u , \Delta_{n,\eps_n}^{l+1}(\Delta_\rho^{k-1-l} - \Delta_{n,\eps_n}^{k-1-l}) u \rangle_{\Lp{2}(\mu_n)} \label{eq:Proofs:GConvergence:WellPosed:Bound:T2Iterated}
\end{align}
where $l \leq k - 1$ and we iterate \eqref{eq:Proofs:GConvergence:WellPosed:Bound:T2} for \eqref{eq:Proofs:GConvergence:WellPosed:Bound:T2Iterated}. Picking $l = k - 1 $ in \eqref{eq:Proofs:GConvergence:WellPosed:Bound:T2Iterated}, we obtain:
\begin{align}
    T_2 &= \sum_{i=0}^{k-1}  \langle \Delta_\rho^k u , \Delta_{n,\eps_n}^i(\Delta_\rho - \Delta_{n,\eps_n}) \Delta_\rho^{k-1-i} u \rangle_{\Lp{2}(\mu_n)} \notag \\
    &\leq \sum_{i=0}^{k-1} \Vert \Delta_{n,\eps_n}^i \Delta_\rho^k u \Vert_{\Lp{2}(\mu_n)} \Vert (\Delta_\rho - \Delta_{n,\eps_n}) \Delta_\rho^{k-1-i} u \Vert_{\Lp{2}(\mu_n)} \notag \\
    &= \sum_{i=0}^{k-1} \sqrt{\cE_{n,\eps_n}^{(2i)}(\Delta_\rho^{k}u)} \Vert (\Delta_\rho - \Delta_{n,\eps_n}) \Delta_\rho^{k-1-i} u \Vert_{\Lp{2}(\mu_n)} \notag \\
    &\leq C \sum_{i=0}^{k-1} \Vert (\Delta_\rho - \Delta_{n,\eps_n}) \Delta_\rho^{k-1-i} u \Vert_{\Lp{2}(\mu_n)} \label{eq:Proofs:GConvergence:WellPosed:Bound:T2Hypothesis} \\
    &= C \sum_{i=0}^{k-1} \left[ \frac{1}{n}  \sum_{j=1}^n \vert \Delta_\rho(\Delta_\rho^{k-1-i}u)(x_j) - \Delta_{n,\eps_n}(\Delta_\rho^{k-1-i}u)(x_j) \vert^2 \right]^{1/2} \notag \\
    &\leq C \sum_{i=0}^k \left[ \frac{C}{n} \sum_{j=1}^n \Vert \Delta_{\rho}^{k-1-i}u \Vert_{\Ck{3}(\Omega)}^2 \right]^{1/2} \label{eq:Proofs:GConvergence:WellPosed:Bound:T2Calder}\\
    &\leq C \label{eq:Proofs:GConvergence:WellPosed:Bound:T2Final}
\end{align}
where we used the induction hypothesis \eqref{eq:Proofs:GConvergence:WellPosed:Bound} for $i \leq 2k-2$ for \eqref{eq:Proofs:GConvergence:WellPosed:Bound:T2Hypothesis} and \cite[Theorem 5]{Calder_2018} as well as Assumption \ref{ass:Main:Ass:L1} for \eqref{eq:Proofs:GConvergence:WellPosed:Bound:T2Calder}.

Similarly to the iteration process that lead to \eqref{eq:Proofs:GConvergence:WellPosed:Bound:T2Iterated} and picking $l = k-1$, we can write: \begin{align}
    T_3 &= \sum_{i=0}^{k-2} \langle \Delta_{n,\eps_n}^{k}u, \Delta_{n,\eps_n}^i(\Delta_\rho - \Delta_{n,\eps_n})\Delta_{\rho}^{k-1-i}u \rangle_{\Lp{2}(\mu_n)} + \langle \Delta_{n,\eps_n}^{k}u, \Delta_{n,\eps_n}^{k-1}(\Delta_\rho - \Delta_{n,\eps_n}) u \rangle_{\Lp{2}(\mu_n)} \notag \\
    &\leq \sqrt{\cE_{n,\eps_n}^{(2k)}(u_n)}\sum_{i=0}^{k-2} \Vert \Delta_{n,\eps_n}^i(\Delta_\rho - \Delta_{n,\eps_n})\Delta_{\rho}^{k-1-i}u \Vert_{\Lp{2}(\mu_n)} + \langle \Delta_{n,\eps_n}^{k}u, \Delta_{n,\eps_n}^{k-1}(\Delta_\rho - \Delta_{n,\eps_n}) u \rangle_{\Lp{2}(\mu_n)} \notag \\
    &\leq \sqrt{\cE_{n,\eps_n}^{(2k)}(u_n)}\sum_{i=0}^{k-2} \ls \Vert \Delta_{n,\eps_n}^{i}(\Delta_\rho^{k-i}u) \Vert_{\Lp{2}(\mu_n)} + \Vert \Delta_{n,\eps_n}^{i+1}(\Delta_\rho^{k-1-i}u)  \Vert_{\Lp{2}(\mu_n)} \rs \notag \\
    &+ \langle \Delta_{n,\eps_n}^{k}u, \Delta_{n,\eps_n}^{k-1}(\Delta_\rho - \Delta_{n,\eps_n}) u \rangle_{\Lp{2}(\mu_n)} \notag \\
    &= \sqrt{\cE_{n,\eps_n}^{(2k)}(u_n)}\sum_{i=0}^{k-2} \left[ \sqrt{\cE_{n,\eps_n}^{(2i)}(\Delta_{\rho}^{k-i}u)} + \sqrt{\cE_{n,\eps_n}^{(2i+2)}(\Delta_{\rho}^{k-1-i}u)} \right] + \langle \Delta_{n,\eps_n}^{k}u, \Delta_{n,\eps_n}^{k-1}(\Delta_\rho - \Delta_{n,\eps_n}) u \rangle_{\Lp{2}(\mu_n)} \notag \\
    &\leq C \sqrt{\cE_{n,\eps_n}^{(2k)}(u_n)} + \langle \Delta_{n,\eps_n}^{k}u, \Delta_{n,\eps_n}^{k-1}(\Delta_\rho - \Delta_{n,\eps_n}) u \rangle_{\Lp{2}(\mu_n)} \label{eq:Proofs:GConvergence:WellPosed:Bound:T3Hypothesis} \\
    &= \frac{1}{2}\left[ \Vert \Delta_{n,\eps_n}^k u + \Delta_{n,\eps_n}^{k-1}(\Delta_\rho - \Delta_{n,\eps_n}) u \Vert_{\Lp{2}(\mu_n)}^2  - \Vert \Delta_{n,\eps_n}^k u \Vert_{\Lp{2}(\mu_n)}^2 - \Vert \Delta_{n,\eps_n}^{k-1}(\Delta_\rho - \Delta_{n,\eps_n}) u  \Vert_{\Lp{2}(\mu_n)}^2   \right] \notag \\
    &+ C \sqrt{\cE_{n,\eps_n}^{(2k)}(u_n)} \notag\\
    &= \frac{1}{2}\left[ \Vert \Delta_{n,\eps_n}^{k-1}(\Delta_\rho u) \Vert_{\Lp{2}(\mu_n)}^2  - \Vert \Delta_{n,\eps_n}^k u \Vert_{\Lp{2}(\mu_n)}^2 - \Vert \Delta_{n,\eps_n}^{k-1}(\Delta_\rho - \Delta_{n,\eps_n}) u  \Vert_{\Lp{2}(\mu_n)}^2   \right] + C \sqrt{\cE_{n,\eps_n}^{(2k)}(u_n)} \notag \\
    &= \frac{1}{2}\left[ \cE_{n,\eps_n}^{(2k-2)}(\Delta_\rho u)  - \cE_{n,\eps_n}^{(2k)}(u) - \cE_{n,\eps_n}^{(2k-2)}((\Delta_\rho-\Delta_{n,\eps_n})u)   \right] + C \sqrt{\cE_{n,\eps_n}^{(2k)}(u_n)} \label{eq:Proofs:GConvergence:WellPosed:Bound:T3}
\end{align}
where we used the induction hypothesis \eqref{eq:Proofs:GConvergence:WellPosed:Bound} for $i \leq 2k-2$ for \eqref{eq:Proofs:GConvergence:WellPosed:Bound:T3Hypothesis}. 

Now, by \cite[Theorem 5]{Calder_2018}, we have
\[
\int_\Omega \vert (\Delta_\rho - \Delta_{n,\eps_n})(u)(T_n(x)) \vert^2 \, \dd \mu \leq C \eps_n^2 \to 0
\]
where $\{T_n\}_{n = 1}^\infty$ are the transport maps from Theorem \ref{thm:Back:TLp:LinftyMapsRate} which shows that $(\Delta_\rho - \Delta_{n,\eps_n})u \to 0$ in $\TLp{2}(\Omega)$ by Proposition \ref{prop:Back:TLp}. By Theorem \ref{thm:Proofs:GConvergence:withoutConstraints}, we therefore obtain:
\[
\lim_{n\to \infty} \cE_{n,\eps_n}^{(2k-2)}((\Delta_\rho-\Delta_{n,\eps_n})u) \geq \liminf_{n \to \infty} \cE_{n,\eps_n}^{(2k-2)}((\Delta_\rho-\Delta_{n,\eps_n})u) \geq \cE_\infty^{(2k-2)}(0) = 0
\]
or equivalently 
\[
- \lim_{n\to \infty} \cE_{n,\eps_n}^{(2k-2)}((\Delta_\rho-\Delta_{n,\eps_n})u) \leq 0.
\]
In the latter, if $\lim_{n\to \infty} \cE_{n,\eps_n}^{(2k-2)}((\Delta_\rho-\Delta_{n,\eps_n})u) = 0$, then we know that the sequence $\{ \cE_{n,\eps_n}^{(2k-2)}((\Delta_\rho-\Delta_{n,\eps_n})u) \}_{n=1}^\infty$ is bounded. Else, we have $\lim_{n\to \infty} -\cE_{n,\eps_n}^{(2k-2)}((\Delta_\rho-\Delta_{n,\eps_n})u) < 0$ which implies that there are only finitely many $n$ such that $ -\cE_{n,\eps_n}^{(2k-2)}((\Delta_\rho-\Delta_{n,\eps_n})u) > 0$. Again, in that case, there exists $C > 0$ such that $ -\cE_{n,\eps_n}^{(2k-2)}((\Delta_\rho-\Delta_{n,\eps_n})u) \leq C$. Inserting this in \eqref{eq:Proofs:GConvergence:WellPosed:Bound:T3}, using our induction hypothesis \eqref{eq:Proofs:GConvergence:WellPosed:Bound} and $u=u_n$ on $\Omega_n$, we obtain
\begin{equation} \label{eq:Proofs:GConvergence:WellPosed:Bound:T3Final}
    \vert T_3 \vert \leq C\sqrt{\cE_{n,\eps_n}^{(2k)}(u_n)} + \frac{1}{2} \left[\cE_{n,\eps_n}^{(2k-2)}(\Delta_\rho u) + \cE_{n,\eps_n}^{(2k)}(u_n) + C \right] \leq  C\sqrt{\cE_{n,\eps_n}^{(2k)}(u_n)} + \frac{1}{2} \cE_{n,\eps_n}^{(2k)}(u_n) + C.  
\end{equation}

Combining the result about $T_1$, \eqref{eq:Proofs:GConvergence:WellPosed:Bound:T2Final} and \eqref{eq:Proofs:GConvergence:WellPosed:Bound:T3Final}, we have
\[
    \frac{1}{2}\cE_{n,\eps_n}^{(2k)}(u_n) \leq C\l 1 + \sqrt{\cE_{n,\eps_n}^{(2k)}(u_n)} \r. 
\]
The above implies
\[ \frac12 \cE_{n,\eps_n}^{(2k)}(u_n) \leq 2C \max\lb 1,\sqrt{\cE_{n,\eps_n}^{(2k)}(u_n)} \rb\]
which in turn yields
\[ \cE_{n,\eps_n}^{(2k)}(u_n) \leq 4C \max\lb 1, 4C\rb. \qedhere \]
\end{proof}

\begin{remark}
Finite-difference approximations of $\Wkp{s}{2}$-norms.
Recalling Remark \ref{rem:Main:Res:Approximating}, one could interpret Lemma \ref{lem:Proofs:GConvergence:WellPosed:Bound} as the statement that our finite-difference approximations of the $\Wkp{s}{2}$-norm are bounded by a constant. Drawing the parallel with \cite[Theorem 2]{Bourgain01anotherlook}, this is not surprising as we informally expect $\cE_{n,\eps_n}^{(s)}(u_n) \to \tilde{C}(s,d) \Vert u \Vert_{\Wkp{s}{2}}$ for some $\tilde{C}(s,d) > 0$ and any $s$. In turn, since we have $u \in \Ck{\infty}(\Omega)$ and assume \ref{ass:Main:Ass:S1}, $\Vert u \Vert_{\Wkp{s}{2}} \leq C(s,d)$. Furthermore, with the above intuition, one expects to be able to prove Lemma \ref{lem:Proofs:GConvergence:WellPosed:Bound} for any $s \in (0,\infty)$.
\end{remark}

\begin{proposition} \label{prop:Proofs:GConvergence:WellPosed:limsup}
$\limsup$-inequality in the well-posed case.
Assume Assumptions~\ref{ass:Main:Ass:S1}, \ref{ass:Main:Ass:M1}, \ref{ass:Main:Ass:M2}, \ref{ass:Main:Ass:D1}, \ref{ass:Main:Ass:W1}, \ref{ass:Main:Ass:W2} and~\ref{ass:Main:Ass:L1} hold. Assume that $\rho \in \Ck{\infty}$ and $\eps_n$ satisfies \ref{eq:Main:Ass:epsLBWellPosed}. Let $(\nu,v) \in \TLp{2}(\Omega)$. Then, $\mathbb{P}$-a.s., there exists $\{(\nu_n,v_n)\} \subseteq \TLp{2}(\Omega)$ such that $(\nu_n,v_n) \to (\nu,v)$ in $\TLp{2}$ and
\begin{equation}\label{eq:Proofs:GConvergence:WellPosed:limsup}
    \limsup_{n\to \infty} \cF_{n,\eps_n}((\nu_n,v_n)) \leq \cF((\nu,v)).
\end{equation}
\end{proposition}

\begin{proof}
In the proof $C>0$ will denote a constant that can be arbitrarily large, independent of $n$ and $k$ that may change from line to line.

With probability one, we can assume that the conclusion Lemma \ref{lem:Proofs:GConvergence:WellPosed:Bound} holds.

Let us start by noticing that if $\mathcal{F}((v,\nu)) = \infty$, \eqref{eq:Proofs:GConvergence:WellPosed:limsup} is trivial. We therefore assume that $\mathcal{F}((\nu,v))<\infty$ which, by \eqref{eq:Main:Not:F}, implies that we have to prove \eqref{eq:Proofs:GConvergence:WellPosed:limsup} on the set
\[
    S = \{(\nu,v) \spaceBar \nu = \mu, v \in \mathcal{H}^s(\Omega) \text{ and } v(x_i) = \ell_i \text{ for $i \leq N$} \} \subseteq \{\mu\} \times \Lp{2}(\mu) \subseteq \TLp{2}(\Omega).
\]
We will begin by considering a dense subset of $S$ and assume that $v \in \Ck{\infty}$ and that for $ i \leq N $, $v(x_i) = \ell_i$. Let $v_n$ be the restriction of $v$ to $\Omega_n$ and let us consider $\{(v_n,\mu_n)\}_{n=1}^\infty \subseteq \TLp{2}(\Omega)$. 

We start by showing that $(\mu_n,v_n) \to (\mu,v)$ in $\TLp{2}$. Let $\{T_n\}_{n=1}^\infty$ be the transport maps from Theorem \ref{thm:Back:TLp:LinftyMapsRate}. We note that
\[
    \int_\Omega \vert v_n \circ T_n - v \vert^2 \, \dd\mu = \int_\Omega \vert v \circ T_n - v \vert^2 \,\dd\mu
\]
and the latter tends to 0 as is already shown in the proof of Proposition \ref{prop:Proofs:GConvergence:WellPosed:liminf}.

We now show that that $ \limsup_{n \to \infty}\mathcal{F}_{n,\eps_n}((\mu_n,v_n)) = \limsup_{n \to \infty} \EnergySnepsn(v_n) \leq \mathcal{F}((\mu,v)) = \EnergyS(v)$. Let $K \in \bbN$ and recall $K_n$ from Proposition \ref{prop:Proofs:Back:DisReg:EVecAlt}. Since $K_n$ tends to infinity, for $n$ large enough, $\lfloor K_n \rfloor \geq K$. Let $n$ be large enough so that the latter holds and we can apply Remark \ref{rem:Proofs:Back:EvalTails:eigenvalueBounds}. For $\gamma > 0$ such that $s + \gamma = 2m$ for some $m \in \bbN$, we can estimate as follows:
\begin{align}
    \EnergySnepsn(v_n) &= \sum_{k=1}^K \lambda_{n,k}^s \langle v_n, \psi_{n,k} \rangle^2 +  \sum_{k=K+1}^n \lambda_{n,k}^s \langle v_n, \psi_{n,k} \rangle^2 \notag \\
    &\leq \sum_{k=1}^K \lambda_{n,k}^s \langle v_n, \psi_{n,k} \rangle^2 + \lambda_{n,K}^{-\gamma} \sum_{k=K+1}^n \lambda_{n,k}^{s+\gamma} \langle v_n, \psi_{n,k} \rangle^2 \notag \\
    &\leq \sum_{k=1}^K \lambda_{n,k}^s \langle v_n, \psi_{n,k} \rangle^2 + C\lambda_{K}^{-\gamma} \cE_{n,\eps_n}^{(s+\gamma)}(v_n) \label{eq:Proofs:GConvergence:WellPosed:limsup:Remark} \\
    &\leq \sum_{k=1}^K \lambda_{n,k}^s \langle v_n, \psi_{n,k} \rangle^2 + CK^{-2\gamma/d} \cE_{n,\eps_n}^{(s+\gamma)}(v_n) \label{eq:Proofs:GConvergence:WellPosed:limsup:Weyl} \\
    &\leq \sum_{k=1}^K \lambda_{n,k}^s \langle v_n, \psi_{n,k} \rangle^2 + CK^{-2\gamma/d} \label{eq:Proofs:GConvergence:WellPosed:limsup:Bound}
\end{align}
where we used Remark \ref{rem:Proofs:Back:EvalTails:eigenvalueBounds} for \eqref{eq:Proofs:GConvergence:WellPosed:limsup:Remark}, Proposition \ref{prop:Proofs:Back:Weyl:Law} for \eqref{eq:Proofs:GConvergence:WellPosed:limsup:Weyl} and Lemma \ref{lem:Proofs:GConvergence:WellPosed:Bound} for \eqref{eq:Proofs:GConvergence:WellPosed:limsup:Bound}. In the proof of Theorem \ref{thm:Proofs:GConvergence:withoutConstraints}, it is proven that $\lambda_{n,k}^s \langle v_n, \psi_{n,k} \rangle^2 \to \lambda_k^s \langle v, \psi_{k} \rangle^2$ since $(\mu_n,v_n) \to (\mu,v)$ in $\TLp{2}(\Omega)$. Inserting the latter in \eqref{eq:Proofs:GConvergence:WellPosed:limsup:Bound}, we obtain:
\begin{align}
    \limsup_{n \to \infty} \EnergySnepsn(v_n) &\leq \sum_{k=1}^K \lambda_k^s \langle v, \psi_k \rangle^2 + CK^{-2\gamma/d} \notag \\
    &\leq \EnergyS(v) + CK^{-2\gamma/d}. \label{eq:Proofs:GConvergence:WellPosed:limsup:limitn}
\end{align}
Finally, taking the limit as $K$ tends to $\infty$ in \eqref{eq:Proofs:GConvergence:WellPosed:limsup:limitn} yields 
\[
\limsup_{n \to \infty} \EnergySnepsn(v_n) \leq \EnergyS(v)
\]
which proves \eqref{eq:Proofs:GConvergence:WellPosed:limsup}. Using \cite[Remark 2.7]{Trillos3}, we extend this result to the whole space $S$ which concludes the proof.
\end{proof}

\subsubsection{Ill-Posed Case} 

We start with the liminf inequality.

\begin{proposition} \label{prop:Proofs:GConvergence:IllPosed:liminf}
$\liminf$-inequality in the ill-posed case. 
Assume Assumptions~\ref{ass:Main:Ass:S1}, \ref{ass:Main:Ass:M1}, \ref{ass:Main:Ass:M2}, \ref{ass:Main:Ass:D1}, \ref{ass:Main:Ass:W1}, \ref{ass:Main:Ass:W2} and~\ref{ass:Main:Ass:L1} hold, $\eps_n$ satisfies~\eqref{eq:Main:Ass:epsLBIllPosed}.
Then, $\bbP$-a.s., we have 
\[
    \liminf_{n \to \infty} \cF_{n,\eps_n}((\nu_n,v_n)) \geq \cG(\nu,v)
\]
for any $(\nu,v) \in \TLp{2}(\Omega)$ and $\{(\nu_n,v_n)\}_{n=1}^\infty \subseteq \TLp{2}(\Omega)$ such that $(\nu_n, v_n) \to (\nu, v)$ in $\TLp{2}$. 
\end{proposition}

\begin{proof}
With probability one, we can assume that the conclusion of Theorem \ref{thm:Proofs:GConvergence:withoutConstraints} holds.

As in the proof of Proposition \ref{prop:Proofs:GConvergence:WellPosed:liminf}, without loss of generality, assume that $\cF_{n,\eps_n}((\nu_n,v_n)) \leq C$ for some $C > 0$. In particular, this implies that $\nu_n = \mu_n$, and by Proposition \ref{prop:Back:TLp}, $\mu_n$ converges weakly to $\nu$. By the uniqueness of weak limits, we must have that $\nu = \mu$. 

We have \[
  C\geq  \liminf_{n \to \infty} \mathcal{F}_{n,\eps_n}( (v_n,\nu_n) ) \geq \liminf_{n \to \infty} \EnergySnepsn(v_n) \geq \EnergyS(v) = \mathcal{G}((\nu,v))
\]
where the last inequality follows from Theorem \ref{thm:Proofs:GConvergence:withoutConstraints}. 
\end{proof}

The $\limsup$ inequality requires two computational Lemmas. 

\begin{lemma} \label{lem:Proofs:GConvergence:IllPosed:diracEnergies}
Energy estimates of dirac deltas. 
Assume Assumptions~\ref{ass:Main:Ass:S1}, \ref{ass:Main:Ass:M1}, \ref{ass:Main:Ass:M2}, \ref{ass:Main:Ass:D1}, \ref{ass:Main:Ass:W1}, \ref{ass:Main:Ass:W2} and~\ref{ass:Main:Ass:L1} hold, $\eps_n$ satisfies~\eqref{eq:Main:Ass:epsLBIllPosed}. Then, there exists $C > 0$ such that, $\bbP$-a.s., for $n$ large enough we have
\begin{equation} \label{eq:Proofs:GConvergence:IllPosed:diracEnergies}
    \EnergySnepsn(\delta_{x_i}) \leq \frac{C}{n \eps_n^{2s}}
\end{equation}
for all $x_i \in \Omega_n$.
\end{lemma}

\begin{proof}
With probability one, we can assume that the conclusion of Theorem \ref{thm:Back:TLp:LinftyMapsRate} holds.

By Assumption \ref{eq:Main:Ass:epsLBIllPosed} and Theorem \ref{thm:Back:TLp:LinftyMapsRate}, we can apply \cite[Lemma 22]{Stuart} for $n$ large enough. We recall the variational definition of eigenvalues:
\[
    \lambda_{n,n}^s = \sup_{\Vert u \Vert_{\Lp{2}} = 1,\: u \in \mathbb{R}^n} \langle u, \Delta_{n,\eps_n}^s u \rangle_{\Lp{2}(\mu_n)}.
\]
Furthermore, for $x_i \in \Omega_n$, 
\[
    \Vert \sqrt{n}\delta_{x_i} \Vert_{\Lp{2}} = \sqrt{\frac{1}{n} \sum_{j=1}^n n \delta_{x_i}(x_j)^2} = 1
\]
so that we can estimate: 
\begin{align}
    \EnergySnepsn(\delta_{x_i}) &= \langle \delta_{x_i}, \Delta_{n,\eps_n}^s \delta_{x_i} \rangle_{\Lp{2}(\mu_n)} \notag\\
    &= \frac{1}{n} \langle \sqrt{n} \delta_{x_i} , \Delta_{n,\eps_n}^s \sqrt{n}\delta_{x_i} \rangle_{\Lp{2}(\mu_n)} \notag\\
    &\leq \frac{1}{n}\sup_{\Vert u \Vert_{\Lp{2}} = 1,\: u \in \mathbb{R}^n} \langle u, \Delta_{n,\eps_n}^s u \rangle_{\Lp{2}(\mu_n)} \notag\\
    &= \frac{\lambda_{n,n}^s}{n} \notag\\
    &\leq \frac{C}{n\eps_n^{2s}}\label{eq:Proofs:GConvergence:IllPosed:dirac:eigenvalueBound}
\end{align}
where we used \cite[Lemma 22]{Stuart} for \eqref{eq:Proofs:GConvergence:IllPosed:dirac:eigenvalueBound}.
\end{proof}

\begin{proposition} \label{prop:Proofs:GConvergence:IllPosed:limsup}
$\limsup$-inequality in the ill-posed case. 
Assume Assumptions~\ref{ass:Main:Ass:S1}, \ref{ass:Main:Ass:M1}, \ref{ass:Main:Ass:M2}, \ref{ass:Main:Ass:D1}, \ref{ass:Main:Ass:W1}, \ref{ass:Main:Ass:W2} and~\ref{ass:Main:Ass:L1} hold, $\eps_n$ satisfies~\eqref{eq:Main:Ass:epsLBIllPosed}.
Assume that $n\eps_n^{2s} \to \infty$ and $\rho \in \Ck{\infty}$. Let $(\nu,v) \in \TLp{2}(\Omega)$. Then, $\bbP$-a.s., there exists $\{(\nu_n,v_n)\} \subseteq \TLp{2}(\Omega)$ such that $(\nu_n,v_n) \to (\nu,v)$ in $\TLp{2}$ and
\begin{equation}\label{eq:Proofs:GConvergence:IllPosed:limsup}
    \limsup_{n\to \infty} \cF_{n,\eps_n}((\nu_n,v_n)) \leq \cG((\nu,v)).
\end{equation}
\end{proposition}

\begin{proof}
In the proof $C>0$ will denote a constant that can be arbitrarily large, independent of $n$ that may change from line to line. With probability one, we can assume that the conclusions of Theorem \ref{thm:Back:TLp:LinftyMapsRate} and Lemma \ref{lem:Proofs:GConvergence:IllPosed:diracEnergies} hold.

We note that \eqref{eq:Proofs:GConvergence:IllPosed:limsup} is trivial if $\mathcal{G}((\nu,v)) = \infty$. Hence, we might assume that $\mathcal{G}((\nu,v)) < \infty$ and in particular, this implies that $\nu = \mu$ and $\mathcal{G}((\nu,v)) = \EnergyS(v)$. We therefore need to prove \eqref{eq:Proofs:GConvergence:IllPosed:limsup} on $\{\mu\}\times\mathcal{H}^s(\Omega) \subseteq \TLp{2}$. We start by assuming that $v \in \Ck{\infty}$, which is dense in $\mathcal{H}^s(\Omega)$. Let $\bar{v}_n$ be the restriction of $v$ to $\Omega_n$ and consider the recovery sequence $\{(\mu_n,\bar{v}_n)\}_{i=1}^\infty \subseteq \TLp{2}(\Omega)$. By Theorem \ref{thm:Back:TLp:LinftyMapsRate}, we get transport maps $\{T_n\}_{n=1}^\infty$ from $\mu$ to $\mu_n$. Repeating the proof of Proposition \ref{prop:Proofs:GConvergence:WellPosed:limsup}, we can show that $(\mu_n,\bar{v}_n) \to (\mu,v)$ in $\TLp{2}$ and that $\limsup_{n \to \infty} \EnergySnepsn(\bar{v}_n) \leq \EnergyS(v)$.

We define the functions \[
    v_n(x_i) = \begin{cases}
    \bar{v}_n(x_i) &\text{if $i \geq N + 1$,}\\
    \ell_i & \text{if $i \leq N$.}
    \end{cases}
\]
and estimate as follows: 
\[
    \int_\Omega \vert v_n \circ T_n - v \vert^2 \,\dd\mu \leq 2 \int_\Omega \vert v_n \circ T_n - v \circ T_n \vert^2 \,\dd\mu + 2 \int_\Omega \vert v \circ T_n - v \vert^2 \,\dd\mu =: 2(A + B).
\]
As in the proof of Proposition \ref{prop:Proofs:GConvergence:WellPosed:liminf}, we see that $B \to 0$. For the $A$ term, we have
\begin{align}
    A &\leq \int_{\{x \spaceBar T_n(x) \neq x_i \text{ for $i \leq N$}\}} \vert v_n \circ T_n - v \circ T_n \vert^2 \, \dd\mu + \sum_{i=1}^N \int_{\{x \spaceBar T_n(x) = x_i\}} \vert v_n \circ T_n - v \circ T_n \vert^2 \, \dd\mu \notag\\
    &=\sum_{i=1}^N \vert v(x_i) - \ell_i \vert^2 \mu(\{x \spaceBar T_n(x) = x_i\}) \notag\\
    &= \sum_{i=1}^N \vert v(x_i) - \ell_i \vert^2 \mu_n(\{x_i\}) \notag\\
    &= \sum_{i=1}^N \frac{\vert v(x_i) - \ell_i \vert^2}{n} \notag
\end{align}
from which we deduce that $(\mu_n,v_n) \to (\mu,v)$ in $\TLp{2}$.

Let $n$ be large enough so that \eqref{eq:Proofs:GConvergence:IllPosed:diracEnergies} holds. We now show that 
\[
\limsup_{n \to \infty} \mathcal{F}_{n,\eps_n}(v_n,\mu_n) = \limsup_{n \to \infty} \EnergySnepsn(v_n) \leq \EnergyS(v).
\]
To this purpose, we note that 
\[
v_n = \bar{v}_n + \sum_{i = 1}^N \delta_{x_i} (\ell_i - \bar{v}_n(x_i))
\]
and applying Lemma \ref{lem:Proofs:GConvergence:IllPosed:Minkowski}, we obtain:
\begin{align}
    \sqrt{\EnergySneps{}(v_n)} &\leq \sqrt{\EnergySnepsn(\bar{v}_n)} + \sum_{i=1}^N \vert \ell_i - \bar{v}_n(x_i) \vert \sqrt{\EnergySnepsn(\delta_{x_i})} \notag\\
    &\leq \sqrt{\EnergySnepsn(\bar{v}_n)} + C \sum_{i=1}^N \sqrt{\EnergySnepsn(\delta_{x_i})} \label{eq:Proofs:GConvergence:IllPosed:limsup:linfty}\\
    &\leq \sqrt{\EnergySnepsn(\bar{v}_n)} + C \l \frac{1}{n\eps_n^{2s}} \r^{1/2} \label{eq:Proofs:GConvergence:IllPosed:limsup:energy}
\end{align}
where we used the fact that $\Vert v \Vert_{\Lp{\infty}}$ is bounded on $\Omega$ for \eqref{eq:Proofs:GConvergence:IllPosed:limsup:linfty} (and $\bar{v}_n$ is the restriction of $v$ to $\Omega_n$ and so is also bounded in $\Lp{\infty}$) and \eqref{eq:Proofs:GConvergence:IllPosed:diracEnergies} for \eqref{eq:Proofs:GConvergence:IllPosed:limsup:energy}. Taking $n \to \infty$ on the latter right hand side, we have 
\[
\limsup_{n \to \infty} \sqrt{\EnergySneps(v_n)} \leq \limsup_{n \to \infty} \sqrt{\EnergySnepsn(\bar{v}_n)} + \limsup_{n \to \infty} C \l \frac{1}{n\eps_n^{2s}} \r^{1/2} \leq \sqrt{\EnergyS(v)}
\]
since by assumption $n\eps_n^{2s} \to \infty$ and $\limsup_{n \to \infty} \EnergySnepsn(\bar{v}_n) \leq \EnergyS(v)$. 

Having shown \eqref{eq:Proofs:GConvergence:IllPosed:limsup} on $\{\mu\}\times \Ck{\infty}$, we now use the fact that it is sufficient to establish the existence of a recovery sequence on a dense subset \cite[Remark 2.7]{Trillos3}. 
\end{proof}

\subsection{Bounds on Minimizers} \label{subsec:Proofs:Linfty}

\begin{lemma} \label{lem:Proofs:Bounds:UniformBounds}
Uniform bound of energies for minimizers.
Assume Assumptions~\ref{ass:Main:Ass:S1}, \ref{ass:Main:Ass:M1}, \ref{ass:Main:Ass:M2}, \ref{ass:Main:Ass:D1}, \ref{ass:Main:Ass:W1}, \ref{ass:Main:Ass:W2} and~\ref{ass:Main:Ass:L1} hold. Assume $\rho \in \Ck{\infty}$. Let $\{(\mu_n,u_n)\}_{n=1}^\infty$ be the minimizers of $\cF_{n,\eps_n}(\cdot)$. Then, there exists $C >0$ such that, $\bbP$-a.s., we have
\[
    \sup_{n} \cF_{n,\eps_n}( (\mu_n,u_n) ) \leq C.
\]
\end{lemma}

\begin{proof}
In the proof $C>0$ will denote a constant that can be arbitrarily large, independent of $n$ that may change from line to line.
With probability one, we can assume that the conclusion of Proposition~\ref{prop:Proofs:GConvergence:WellPosed:limsup} holds.

Let $v \in C^\infty(\Omega)$ be a function that interpolates all the points $\{(x_i,\ell_i)\}_{i=1}^N$. Since $\rho \in C^\infty$, we have that $\Delta^s_\rho v \in C^\infty$ implying that $v \Delta^s_\rho v \in C^\infty$. 
We have
\[
    \EnergyS(v) = \int_\Omega v \Delta^s_\rho v \,\dd\mu < K.
\]
for some $K > 0$. 

By Proposition \ref{prop:Proofs:GConvergence:WellPosed:limsup} there exists a sequence $v_n$ converging to $v$ and such that 
\[
    \lim_{n \to \infty} h_n := \lim_{n \to \infty} \sup_{m \geq n } \mathcal{E}^{(s)}_{m,\eps_m}(v_m) =  \limsup_{n \to \infty} \EnergySnepsn(v_n) = \limsup_{n \to \infty} \mathcal{F}_{n,\eps_n}((\mu_n,v_n)) \leq \EnergyS(v) < K.
\]
Let $h := \limsup_{n \to \infty} \EnergySnepsn(v_n)$ and let $\bar{\eps} = K - h > 0$. Then, there exists $n_0$ such that for all $n \geq n_0$, $\vert h_n - h \vert < \bar{\eps}/2$, which is equivalent to $h_n = \sup_{m \geq n} \mathcal{E}^{(s)}_{m,\eps_m}(v_m) < h + \bar{\eps}/2 < K$. Using the latter, we have 
\[
    \sup_{n} \EnergySnepsn(v_n) = \max\{\mathcal{E}_{1,\eps_1}^{(s)}(v_1),\dots,\mathcal{E}_{n_0,\eps_{n_0}}^{(s)}(v_{n_0}),\sup_{n \geq n_0} \EnergySnepsn(v_n)\} 
    \leq C.
\]
Since $\{u_n\}_{n=1}^\infty$ are minimizers, we have $\EnergySnepsn(u_n) \leq \EnergySnepsn(v_n)$ which implies 
\[
   \sup_n \mathcal{F}_{n,\eps_n}((\mu_n,u_n)) =  \sup_n \EnergySnepsn(u_n) \leq \sup_n \EnergySnepsn(v_n) \leq C. \qedhere
\]
\end{proof}

\subsection{Proof of Theorem~\ref{thm:Main:Res:ConsFracLap}}

\begin{proof}[Proof of Theorem \ref{thm:Main:Res:ConsFracLap}]
In the proof $C>0$ will denote a constant that can be arbitrarily large, independent of $n$ that may change from line to line.
\paragraph{Well-Posed Case.}
With probability one, we can assume that the conclusions of Lemma \ref{lem:Proofs:Bounds:UniformBounds}, Proposition \ref{prop:Proofs:Bounds:Poincaré}, Theorem \ref{thm:Proofs:GConvergence:withoutConstraints}, Proposition \ref{prop:Proofs:Compactness:L2Compactness} and
Proposition \ref{prop:Proofs:GConvergence:WellPosed:liminf} hold. 
Using Lemma \ref{lem:Proofs:Bounds:UniformBounds} and Proposition \ref{prop:Proofs:Bounds:Poincaré}, we have
\[ \Vert u_n \Vert_{\addmaths{\Lp{\infty}}{\removemaths{\Lp{2}}}} \leq \Vert u_n - \frac{1}{N} \sum_{i=1}^N u_n(x_i) \Vert_{\addmaths{\Lp{\infty}}{\removemaths{\Lp{2}}}} + \Vert \frac{1}{N} \sum_{i=1}^N u_n(x_i) \Vert_{\addmaths{\Lp{\infty}}\removemaths{\Lp{2}}} \leq C \sqrt{\EnergySnepsn(u_n)} + \frac{1}{N} \sum_{i=1}^N \vert\ell_i\vert \leq C. \]
Hence minimisers are bounded in $\TLp{2}$ and $\max\l\sup_{n} \Vert u_{n} \Vert_{\Lp{\infty}},\sup_{n} \mathcal{E}^{(s)}_{n}(u_{n}\r \leq C$.
By Proposition~\ref{prop:Proofs:Compactness:L2Compactness} there exists $u$ and a subsequence converging uniformly and in $\TLp{2}$.
By Propositions~\ref{prop:Proofs:GConvergence:WellPosed:liminf} and \ref{prop:Proofs:GConvergence:WellPosed:limsup} it follows that $(\mu,u)$ is a minimiser of $\cF$.
By uniqueness of the minimiser of $\cF$ it follows that the whole sequence $\{( \mu_n,u_n)\}$ converges in $\TLp{2}$ to $(\mu,u)$.

\paragraph{Ill-Posed Case.}
With probability one, we can assume that the conclusions of Theorem \ref{thm:Proofs:GConvergence:withoutConstraints}, Proposition \ref{prop:Proofs:GConvergence:IllPosed:liminf} and Proposition \ref{prop:Proofs:GConvergence:IllPosed:limsup}  hold.

By Proposition \ref{prop:Proofs:GConvergence:IllPosed:liminf} and Proposition \ref{prop:Proofs:GConvergence:IllPosed:limsup}, we know that $\mathcal{F}_{n,\eps_n}$ $\Gamma$-converges to $\mathcal{G}(\cdot)$.
By Theorem \ref{thm:Proofs:GConvergence:withoutConstraints} $\mathcal{F}_{n,\eps_n}(\cdot)$ satisfies the compactness property.
Hence, by Proposition~\ref{prop:Back:Gamma:minimizers} we can conclude the result.    
\end{proof}

\section{Numerical Experiments} \label{sec:NumExp}

\begin{figure}
\centering
\includegraphics[width = \textwidth]{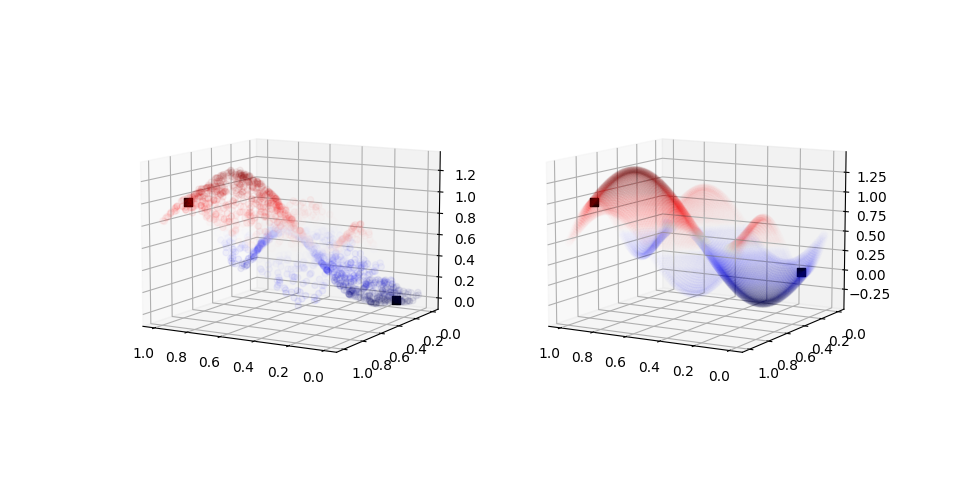}
\caption{Plots of the discrete and continuum minimizers with the setting described in Section \ref{subsec:criticalBoundary}. The values at the points $(0.1,0.1)$ and $(0.9,0.9)$ are marked with black squares. 
\textit{Left}: Discrete minimizer computed with $n = 1733$ points.
\textit{Right}: Continuum minimizer.
}
\end{figure}

\subsection{Critical boundary between well-posed and ill-posed regimes} \label{subsec:criticalBoundary}

In this section, we will investigate the gap in the upper bound alluded to in Remark \ref{rem:Main:Res:Relationship}. In particular, we will rely on the same methodology as in \cite{Slepcev}.

To test the gap between the well-posed and ill-posed regime, i.e. $n^{-\frac{2}{s-1}}\lesssim \eps_n\lesssim n^{-\frac{1}{2s}}$ (by Corollary~\ref{cor:Main:Res:ConsFracLap} we know that $\eps_n\lesssim n^{-\frac{2}{s-1}}$ implies asymptotic well-posedness and $\eps_n\gg n^{-\frac{1}{2s}}$ implies asymptotic ill-posedness)
we will consider the following setting: we choose the uniform measure on $[0,1]^2$ with periodic boundary conditions, the kernel function $\eta(t) = 1$ if $t\leq 1$ and $\eta(t) = 0$ else and $s = 16$. This choice of $s$ satisfies the constraint $s > 2d + 9 = 13$ from Remark \ref{rem:Main:Res:LowerBoundS}. The training set consists of the points $(0.1,0.1)$ and $(0.9,0.9)$ labelled $0$ and $1$ respectively. 

We will vary the number of points in our graph from $n = 100$ to $n = 5000$. For each $n$, we will also consider a wide range of $\eps_n$ ranging from roughly the connectivity radius to above $n^{-1/2s}$. For each combination of $(n,\eps_n)$, we then compute the discrete minimizer $u_n$ using Lagrange multipliers and compare it to the continuum solution $u$. The continuum solution is computed using a finite-difference scheme on a regular grid of 10000 points on $[0,1]^2$.
The error considered is 
\begin{equation} \label{eq:numExp:error}
\text{err}(n,\eps_n,u_n) = \Vert u_n - u \Vert_{\Lp{2}(\mu_n)}
\end{equation}
where, in order to evalutate $u$ on the graph, we use spline interpolation.
Finally, by re-sampling the points for each $n$, we repeat the experiments one hundred times and average the error. 

We are interested in finding the value of $\eps_n$ where fractional Laplacian learning switches from the well-posed to the ill-posed regime. In order to compute the latter, we start by smoothing the function $\eps_n \mapsto \text{err}(n,\eps_n,u_n)$ and compute the maximizer of its first derivative and the minimizer of its second derivative which we denote by $\hat{\eps}_n$ and $\eps_n^{*}$ respectively. 
We choose $\hat{\eps}_n$ and $\eps_n^*$ 
greater than the minimizer of $\text{err}(n,\cdot,u_n)$. 
Both $\hat{\eps}_n$ and $\eps_n^*$ could be taken as reasonable definitions of the transition point between the well-posed and ill-posed regime and, it is therefore interesting to understand how they scale with $n$.  Using the five largest values of $n$, the best linear fits between  $\log(\hat{\eps}_n)$, $\log(\eps_n^*)$ and $\log(n)$ yield that
\[
   \hat{\eps}_n \approx \frac{0.6541}{n^{0.05}} \quad \text{and} \quad \eps_n^* \approx \frac{0.7312}{n^{0.06}}.
\]

For $s = 16$, we have $1/(2s) = 0.03125$ and $2/(s-1) \approx 0.134$. Given the top plots in Figure \ref{fig:error}, we observe that both $\hat{\eps}_n$ and $\eps_n^*$ seem to scale accurately with powers that are different to $1/(2s)$. In fact, we note that $0.06 \approx 0.05 \approx 1/s = 0.0625$. On one hand, this indicates that we should be able to extend the well-posed regime of Theorem \ref{thm:Main:Res:ConsFracLap} to 
\[
\l \frac{1}{n} \r^{2/(s-1)} \ll \eps_n \ll \l \frac{1}{n}\r^{1/s}
\]
and, by Remark \ref{rem:Main:Res:LowerBoundS}, we could relax our assumption of $s > 2d + 9$ to $s > d + 4$. If we were able to also tackle the lower bound gap (see Remark \ref{rem:Main:Res:LowerBoundGap}), we could furthermore have $s > d$ in contrast to $s > d/2$ which is conjectured in Remark \ref{rem:Main:Res:Sobolev}.
On the other hand, we are not able to fully confirm the conjecture made in Remark \ref{rem:Main:Res:Relationship} numerically  and it remains an open question to accurately characterize the behaviour of fractional Laplacian semi-supervised learning when 
\[
\l \frac{1}{n} \r^{1/s} \ll \eps_n \ll \l \frac{1}{n}\r^{1/(2s)}.
\]

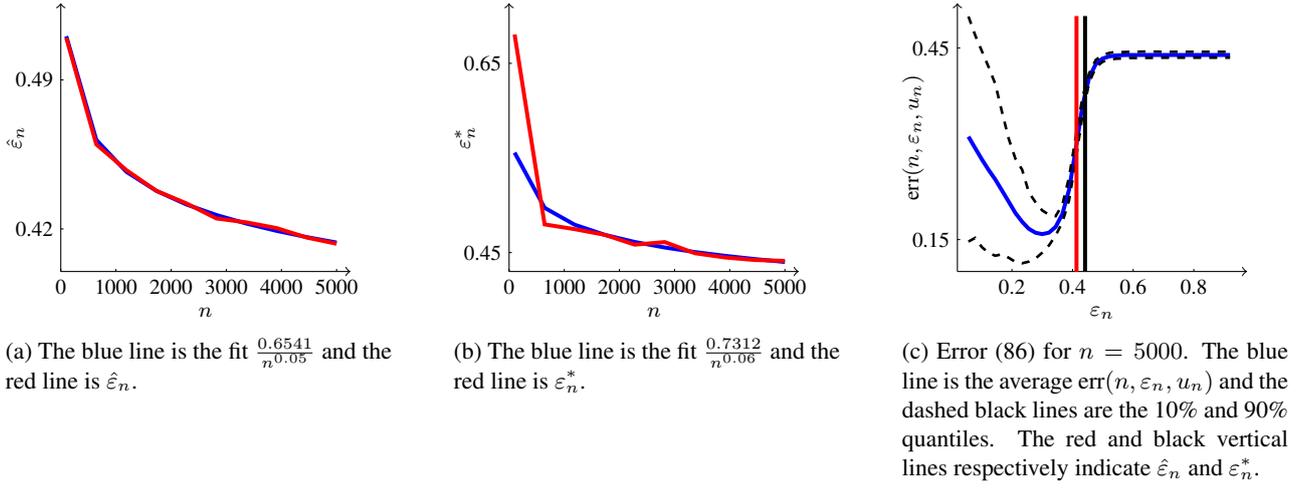
\begin{figure}[t!]
\setlength\figureheight{0.25\textwidth}
\setlength\figurewidth{0.3\textwidth}
\centering
\scriptsize
 
\begin{subfigure}[t]{0.3\textwidth}
\centering
\scriptsize
\begin{tikzpicture}
\def\xl{0}
\def\xu{5000}
\def\yl{0.40}
\def\yu{0.52}

\begin{axis}[%
width=\figurewidth,
height=\figureheight,
scale only axis,
xmin={\xl-(\xu-\xl)/5},
xmax={\xu+(\xu-\xl)/5},
ymin={\yl-(\yu-\yl)/5},
ymax={\yu+(\yu-\yl)*0.05},
hide axis,
axis background/.style={fill=white!100}
]
\draw [->] (axis cs: \xl,\yl) -- (axis cs: {\xu+(\xu-\xl)*0.05},\yl);
\draw [->] (axis cs: \xl,\yl) -- (axis cs: \xl,{\yu+(\yu-\yl)*0.05});
\node at (axis cs: {(\xu+(\xu-\xl)*0.05+\xl)/2},{\yl-0.16*(\yu-\yl)}) {$n$};
\node[rotate=90] at (axis cs: {\xl-0.165*(\xu-\xl)},{(\yu+(\yu-\yl)*0.05+\yl)/2}) {$\hat{\eps}_n$};
\foreach \yValue in {0.42,0.49} {
    \edef\temp{\noexpand\draw [-] ({\xl+(\xu-\xl)/100},\yValue) -- (\xl,\yValue) node[left] {\yValue};} 
    \temp
}
\foreach \xValue in {0,1000,2000,3000,4000,5000} {
	\edef\temp{\noexpand\draw [-] (\xValue,{\yl+(\yu-\yl)/100}) -- (\xValue,\yl) node[below] {\xValue};}    
    \temp
}
\addplot [
color=blue,
solid,
mark options={solid},
line width=1.5pt,
forget plot
]
table[row sep=crcr]{
100 0.5105180044902877 \\
644 0.4618272279200647 \\
1189 0.4468362265447278 \\
1733 0.4378678635984564 \\
2278 0.4314714530723095 \\
2822 0.4265274061822768 \\
3367 0.42249342252377287 \\
3911 0.41910174027509495 \\
4456 0.4161696013335309 \\
5000 0.4135977676403563 \\
};
\addplot [
color=red,
solid,
mark options={solid},
line width=1.5pt,
]
table[row sep=crcr]{
100 0.509749 \\
644 0.459700 \\
1189 0.447607 \\
1733 0.437894 \\
2278 0.431985 \\
2822 0.424900 \\
3367 0.423074 \\
3911 0.420404 \\
4456 0.415974 \\
5000 0.413030 \\
};
\end{axis}
\end{tikzpicture}%
\caption{
The blue line is the fit $\frac{0.6541}{n^{0.05}}$ and the red line is $\hat{\eps}_n$. 
}
\end{subfigure}
\hspace*{0.04\textwidth}
\begin{subfigure}[t]{0.3\textwidth}
\centering
\scriptsize
\begin{tikzpicture}
\def\xl{0}
\def\xu{5000}
\def\yl{0.43}
\def\yu{0.70}

\begin{axis}[%
width=\figurewidth,
height=\figureheight,
scale only axis,
xmin={\xl-(\xu-\xl)/5},
xmax={\xu+(\xu-\xl)/5},
ymin={\yl-(\yu-\yl)/5},
ymax={\yu+(\yu-\yl)*0.05},
hide axis,
axis background/.style={fill=white!100}
]
\draw [->] (axis cs: \xl,\yl) -- (axis cs: {\xu+(\xu-\xl)*0.05},\yl);
\draw [->] (axis cs: \xl,\yl) -- (axis cs: \xl,{\yu+(\yu-\yl)*0.05});
\node at (axis cs: {(\xu+(\xu-\xl)*0.05+\xl)/2},{\yl-0.16*(\yu-\yl)}) {$n$};
\node[rotate=90] at (axis cs: {\xl-0.16*(\xu-\xl)},{(\yu+(\yu-\yl)*0.05+\yl)/2}) {$\eps_n^*$};
\foreach \yValue in {0.45,0.65} {
    \edef\temp{\noexpand\draw [-] ({\xl+(\xu-\xl)/100},\yValue) -- (\xl,\yValue) node[left] {\yValue};} 
    \temp
}
\foreach \xValue in {0,1000,2000,3000,4000,5000} {
	\edef\temp{\noexpand\draw [-] (\xValue,{\yl+(\yu-\yl)/100}) -- (\xValue,\yl) node[below] {\xValue};}    
    \temp
}
\addplot [
color=blue,
solid,
mark options={solid},
line width=1.5pt,
forget plot
]
table[row sep=crcr]{
100 0.5556510885972593 \\
644 0.4972689052922075 \\
1189 0.47942386056165354 \\
1733 0.4687786589913229 \\
2278 0.4612005735642463 \\
2822 0.4553514216272993 \\
3367 0.4505843304065218 \\
3911 0.44658005046187366 \\
4456 0.4431211207493321 \\
5000 0.4400893842346224 \\
};
\addplot [
color=red,
solid,
mark options={solid},
line width=1.5pt,
]
table[row sep=crcr]{
100 0.680498 \\
644 0.479771 \\
1189 0.474747 \\
1733 0.468762 \\
2278 0.458361 \\
2822 0.461102 \\
3367 0.449320 \\
3911 0.444848 \\
4456 0.442129 \\
5000 0.441219 \\
};
\end{axis}
\end{tikzpicture}%
\caption{
The blue line is the fit $\frac{0.7312}{n^{0.06}}$ and the red line is $\eps_n^*$. 
}
\end{subfigure}
\hspace*{0.04\textwidth}
\begin{subfigure}[t]{0.3\textwidth}
\centering
\scriptsize
\begin{tikzpicture}
\def\xl{0.02}
\def\xu{0.93}
\def\yl{0.1}
\def\yu{0.5}

\begin{axis}[%
width=\figurewidth,
height=\figureheight,
scale only axis,
xmin={\xl-(\xu-\xl)/5},
xmax={\xu+(\xu-\xl)/5},
ymin={\yl-(\yu-\yl)/5},
ymax={\yu+(\yu-\yl)*0.05},
hide axis,
axis background/.style={fill=white!100}
]
\draw [->] (axis cs: \xl,\yl) -- (axis cs: {\xu+(\xu-\xl)*0.05},\yl);
\draw [->] (axis cs: \xl,\yl) -- (axis cs: \xl,{\yu+(\yu-\yl)*0.05});
\node at (axis cs: {(\xu+(\xu-\xl)*0.05+\xl)/2},{\yl-0.16*(\yu-\yl)}) {$\eps_n$};
\node[rotate=90] at (axis cs: {\xl-0.16*(\xu-\xl)},{(\yu+(\yu-\yl)*0.05+\yl)/2}) {$\text{err}(n,\eps_n,u_n)$};
\foreach \yValue in {0.15,0.45} {
    \edef\temp{\noexpand\draw [-] ({\xl+(\xu-\xl)/100},\yValue) -- (\xl,\yValue) node[left] {\yValue};} 
    \temp
}
\foreach \xValue in {0.2,0.4,0.6,0.8} {
	\edef\temp{\noexpand\draw [-] (\xValue,{\yl+(\yu-\yl)/100}) -- (\xValue,\yl) node[below] {\xValue};}    
    \temp
}
\addplot [
color=blue,
solid,
mark options={solid},
line width=1.5pt,
forget plot
]
table[row sep=crcr]{
0.0564    0.311140 \\
0.0785    0.293246 \\
0.1007    0.275495 \\
0.1228    0.259083 \\
0.1449    0.244105 \\
0.1671    0.226729 \\
0.1892    0.209286 \\
0.2113    0.191612 \\
0.2335    0.177896 \\
0.2556    0.166826 \\
0.2777    0.160687 \\
0.2999    0.158511 \\
0.3220    0.160818 \\
0.3441    0.169800 \\
0.3663    0.189866 \\
0.3884    0.227872 \\
0.4105    0.291495 \\
0.4327    0.359309 \\
0.4548    0.403903 \\
0.4769    0.424940 \\
0.4991    0.433497 \\
0.5212    0.436719 \\
0.5433    0.438319 \\
0.5655    0.438780 \\
0.5876    0.438991 \\
0.6097    0.438938 \\
0.6319    0.439002 \\
0.6540    0.439057 \\
0.6761    0.439016 \\
0.6983    0.438981 \\
0.7204    0.438989 \\
0.7425    0.438989 \\
0.7646    0.438989 \\
0.7868    0.438989 \\
0.8089    0.438989 \\
0.8310    0.438989 \\
0.8532    0.438989 \\
0.8753    0.438989 \\
0.8974    0.438989 \\
0.9196    0.438989 \\
};
\addplot [
color=black,
dashed,
line width=1pt,
forget plot
]
table[row sep=crcr]{
0.0564    0.14679 \\
0.0785    0.15261\\
0.1007    0.13815\\
0.1228    0.12920\\
0.1449    0.12467\\
0.1671    0.12699\\
0.1892    0.12154\\
0.2113    0.11503\\
0.2335    0.11294\\
0.2556    0.11545\\
0.2777    0.12027\\
0.2999    0.12951\\
0.3220    0.14081\\
0.3441    0.15624\\
0.3663    0.17705\\
0.3884    0.21127\\
0.4105    0.27427\\
0.4327    0.34495\\
0.4548    0.39608\\
0.4769    0.41989\\
0.4991    0.42839\\
0.5212    0.43150\\
0.5433    0.43319\\
0.5655    0.43350\\
0.5876    0.43438\\
0.6097    0.43429\\
0.6319    0.43449\\
0.6540    0.43449\\
0.6761    0.43459\\
0.6983    0.43469\\
0.7204    0.43478\\
0.7425    0.43478\\
0.7646    0.43478\\
0.7868    0.43478\\
0.8089    0.43478\\
0.8310    0.43478\\
0.8532    0.43478\\
0.8753    0.43478\\
0.8974    0.43478\\
0.9196    0.43478\\
};
\addplot [
color=black,
dashed,
line width=1pt,
forget plot
]
table[row sep=crcr]{
0.0564    0.49935\\
0.0785    0.47010\\
0.1007    0.44504\\
0.1228    0.42029\\
0.1449    0.39370\\
0.1671    0.34984\\
0.1892    0.30800\\
0.2113    0.27669\\
0.2335    0.25946\\
0.2556    0.22387\\
0.2777    0.20847\\
0.2999    0.19556\\
0.3220    0.18905\\
0.3441    0.18659\\
0.3663    0.20426\\
0.3884    0.24740\\
0.4105    0.31084\\
0.4327    0.37244\\
0.4548    0.41044\\
0.4769    0.43144\\
0.4991    0.43875\\
0.5212    0.44222\\
0.5433    0.44381\\
0.5655    0.44355\\
0.5876    0.44391\\
0.6097    0.44392\\
0.6319    0.44432\\
0.6540    0.44431\\
0.6761    0.44420\\
0.6983    0.44391\\
0.7204    0.44400\\
0.7425    0.44400\\
0.7646    0.44400\\
0.7868    0.44400\\
0.8089    0.44400\\
0.8310    0.44400\\
0.8532    0.44400\\
0.8753    0.44400\\
0.8974    0.44400\\
0.9196    0.4440\\
};
\addplot [
color=red,
solid,
mark options={solid},
line width=1.5pt,
forget plot
]
table[row sep=crcr]{
0.4130300287265496    0.1 \\
0.4130300287265496    0.2 \\
0.4130300287265496    0.3 \\
0.4130300287265496    0.4 \\
0.4130300287265496    0.5 \\
};
\addplot [
color=black,
solid,
mark options={solid},
line width=1.5pt,
forget plot
]
table[row sep=crcr]{
0.4412185112523447    0.1 \\
0.4412185112523447    0.2 \\
0.4412185112523447    0.3 \\
0.4412185112523447    0.4 \\
0.4412185112523447    0.5 \\
};
\end{axis}
\end{tikzpicture}%
\caption{
Error \eqref{eq:numExp:error} for $n = 5000$. The blue line is the average $\text{err}(n,\eps_n,u_n)$ and the dashed black lines are the 10\% and 90\% quantiles. The red and black vertical lines respectively indicate $\hat{\eps}_n$ and $\eps_n^*$.
}
\end{subfigure}

\caption{
Plots of the errors between the discrete and continuum minimizers with the setting described in Section \ref{subsec:criticalBoundary}.
\label{fig:error}	
}
\end{figure}

\subsection{Uniform Bounds on Eigenfunctions} \label{subsec:linear}

Proposition~\ref{prop:Proofs:Back:DisReg:EVecAlt} implies that we can upper bound the $\Lp{\infty}$ norm of the discrete eigenvectors $\psi_{n,k}$ by 
\begin{equation} \label{eq:numExp:lambda}
\|\psi_{n,k}\|_{\Lp{\infty}}\leq C\lambda_{k}^{d+1}
\end{equation}
when $k\in\{1,\hdots,\ceil{K_n}\}$ where $K_n=\alpha \eps_n^{-d/2} + 1$. We now investigate the following: on one hand, we want to see if the power of $\eps_n$ in the definition of $K_n$, namely $-d/2$, is the lowest we can get; and on the other hand, we are interested in the optimal power of $\lambda_{k}$ in \eqref{eq:numExp:lambda}.

We choose the uniform measure on $[0,1]^2$ with periodic boundary conditions and the kernel function $\eta(t) = 1$ if $t\leq 1$ and $\eta(t) = 0$ else. We proceed as follows: given a number of points ranging from $n = 400$ to $n = 5200$, we seek the smallest $\eps_n$ such that the graph is connected. We then compute $\lb \log\l \lambda_{k} \r, \log \l \|\psi_{n,k}\|_{\Lp{\infty}} \r \rb_{k=1}^{k_n^{(i)}}$ where $k_n^{(1)}  = \alpha_1 \eps_n^{-d/4}$, $k_n^{(2)}  = \alpha_2 \eps_n^{-d/2}$, $k_n^{(3)}  = \alpha_3 \eps_n^{-d}$, $k_n^{(4)}  = n$
and $\alpha_i$ for $1 \leq i \leq 3$ are constants. While we took the arbitrary choice of $\alpha_i = 4$ for $1 \leq i \leq 3$, empirical trials have shown that the overall conclusion does not change significantly. By re-sampling the points for each $n$, we repeat the experiments one hundred times and average the results.

\begin{figure}[t!]
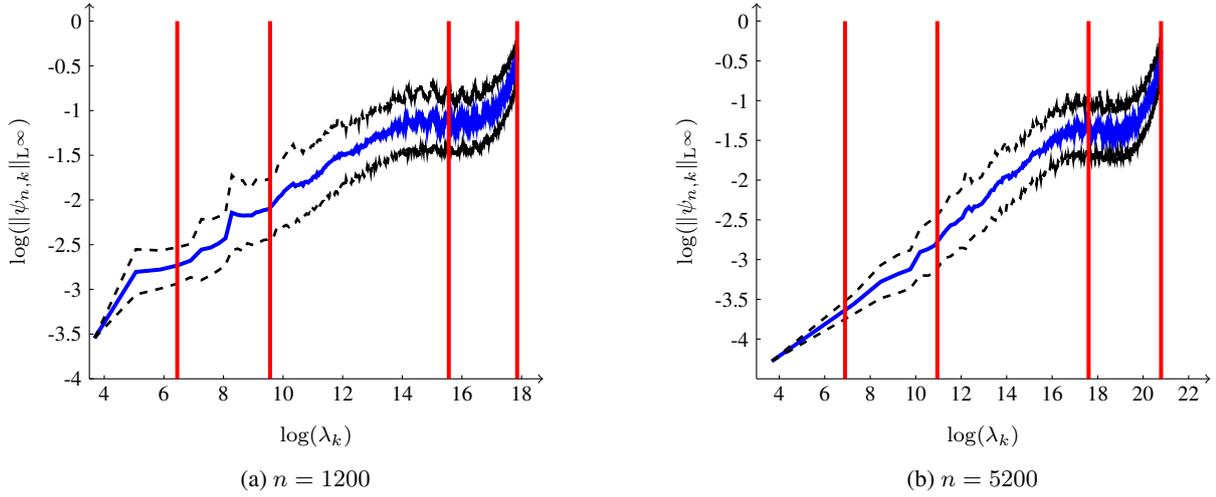

\setlength\figureheight{0.35\textwidth}
\setlength\figurewidth{0.47\textwidth}
\centering
\scriptsize
 
\begin{subfigure}[t]{0.47\textwidth}
\centering
\scriptsize
\input{eigenPairs1200.tikz}
\caption{
$n = 1200$
}
\end{subfigure}
\hspace*{0.04\textwidth}
\begin{subfigure}[t]{0.47\textwidth}
\centering
\scriptsize
\input{eigenPairs5200.tikz}
\caption{
$n = 5200$
}
\end{subfigure}

\caption{
Plots of $\log \l \|\psi_{n,k}\|_{\Lp{\infty}} \r$ versus $\log\l \lambda_{k} \r$ for $1 \leq k \leq n$ with the setting described in Section \ref{subsec:linear}. The blue line corresponds to the average of $\log \l \|\psi_{n,k}\|_{\Lp{\infty}} \r$, the dashed black lines are the mean plus/minus the standard deviation and the red lines indicate the values of $\log \l \|\lambda_{k_n^{(i),*}}\|_{\Lp{\infty}} \r$ for $1 \leq i \leq 4$ from left to right. \ref{subsec:criticalBoundary}.
\label{fig:eigenvalues}	
}
\end{figure}

\begin{figure}[t!]
\setlength\figureheight{0.25\textwidth}
\setlength\figurewidth{0.3\textwidth}
\centering
\scriptsize
 
\begin{subfigure}[t]{0.3\textwidth}
\centering
\scriptsize
\begin{tikzpicture}
\def\xl{8.3}
\def\xu{11.2}
\def\yl{-2.82}
\def\yu{-1.5}

\begin{axis}[%
width=\figurewidth,
height=\figureheight,
scale only axis,
xmin={\xl-(\xu-\xl)/5},
xmax={\xu+(\xu-\xl)/5},
ymin={\yl-(\yu-\yl)/5},
ymax={\yu+(\yu-\yl)*0.05},
hide axis,
axis background/.style={fill=white!100}
]
\draw [->] (axis cs: \xl,\yl) -- (axis cs: {\xu+(\xu-\xl)*0.05},\yl);
\draw [->] (axis cs: \xl,\yl) -- (axis cs: \xl,{\yu+(\yu-\yl)*0.05});
\node at (axis cs: {(\xu+(\xu-\xl)*0.05+\xl)/2},{\yl-0.12*(\yu-\yl)}) {$\log ( \Vert \psi_{n,k_n^{(2),*}} \Vert_{\Lp{\infty}} )$};
\node[rotate=90] at (axis cs: {\xl-0.12*(\xu-\xl)},{(\yu+(\yu-\yl)*0.05+\yl)/2}) {$\log(\lambda_{k_n^{(2),*}})$};
\foreach \yValue in {-2.6,-1.6} {
    \edef\temp{\noexpand\draw [-] ({\xl+(\xu-\xl)/100},\yValue) -- (\xl,\yValue) node[left] {\yValue};} 
    \temp
}
\foreach \xValue in {8.5,11} {
	\edef\temp{\noexpand\draw [-] (\xValue,{\yl+(\yu-\yl)/100}) -- (\xValue,\yl) node[below] {\xValue};}    
    \temp
}
\addplot [
color=blue,
solid,
mark options={solid},
line width=1.5pt,
]
table[row sep=crcr]{
8.471544678415432   -1.2540569433300188 \\
9.220931577298254   -1.7233199981534746 \\
9.564632091151571   -1.9385438742522458 \\
9.857839039535323   -2.122148883225771 \\
10.113505782555093   -2.282246366948902 \\
10.191947208861656   -2.3313660718546085 \\
10.41034579279164   -2.468126384607456 \\
10.478148896143002   -2.510584414568487 \\
10.607225938418145   -2.5914119380492266 \\
10.668769255751652   -2.6299501152430693 \\
10.786450255797519   -2.7036414830245574 \\
10.786450255797519   -2.7036414830245574 \\
10.950926452271462   -2.806635814147998 \\
};
\addplot [
color=red,
only marks,
mark size=3.0pt,
mark=x,
mark options={solid},
line width=1.5pt,
]
table[row sep=crcr]{
8.471544678415432   -1.5566685661277149 \\
9.220931577298254   -1.8463628343359824 \\
9.564632091151571   -2.088603803154077 \\
9.857839039535323   -2.129675348626022 \\
10.113505782555093   -2.1951743774725325 \\
10.191947208861656   -2.3334050732939327 \\
10.41034579279164   -2.429573656463954 \\
10.478148896143002   -2.5039722400392765 \\
10.607225938418145   -2.6023849755115283 \\
10.668769255751652   -2.655875802569578 \\
10.786450255797519   -2.7094539466160485 \\
10.786450255797519   -2.7761405886877966 \\
10.950926452271462   -2.7365904227770645 \\
};
\end{axis}
\end{tikzpicture}%
\caption{
The blue line is the fit $4.05 - 0.63 \log(\Vert \psi_{n,k_n^{(2),*}} \Vert)$ and the red line is $\lambda_{k_n^{(2),*}}$.  
}
\end{subfigure}
\hspace*{0.04\textwidth}
\begin{subfigure}[t]{0.3\textwidth}
\centering
\scriptsize
\begin{tikzpicture}
\def\xl{13}
\def\xu{18}
\def\yl{-1.3}
\def\yu{-0.9}

\begin{axis}[%
width=\figurewidth,
height=\figureheight,
scale only axis,
xmin={\xl-(\xu-\xl)/5},
xmax={\xu+(\xu-\xl)/5},
ymin={\yl-(\yu-\yl)/5},
ymax={\yu+(\yu-\yl)*0.05},
hide axis,
axis background/.style={fill=white!100}
]
\draw [->] (axis cs: \xl,\yl) -- (axis cs: {\xu+(\xu-\xl)*0.05},\yl);
\draw [->] (axis cs: \xl,\yl) -- (axis cs: \xl,{\yu+(\yu-\yl)*0.05});
\node at (axis cs: {(\xu+(\xu-\xl)*0.1+\xl)/2},{\yl-0.12*(\yu-\yl)}) {$\log ( \Vert \psi_{n,k_n^{(3),*}} \Vert_{\Lp{\infty}} )$};
\node[rotate=90] at (axis cs: {\xl-0.12*(\xu-\xl)},{(\yu+(\yu-\yl)*0.05+\yl)/2}) {$\log(\lambda_{k_n^{(3),*}})$};
\foreach \yValue in {-1.25,-0.95} {
    \edef\temp{\noexpand\draw [-] ({\xl+(\xu-\xl)/100},\yValue) -- (\xl,\yValue) node[left] {\yValue};} 
    \temp
}
\foreach \xValue in {13.5,18} {
	\edef\temp{\noexpand\draw [-] (\xValue,{\yl+(\yu-\yl)/100}) -- (\xValue,\yl) node[below] {\xValue};}    
    \temp
}
\addplot [
color=blue,
solid,
mark options={solid},
line width=1.5pt,
]
table[row sep=crcr]{
13.332381607423294   -0.921019815668136 \\
14.072748195350343   -0.9828226002805083 \\
15.561352883072093   -1.1070852512272855 \\
15.008607509043555   -1.0609443207022886 \\
15.737124653341217   -1.1217579620654312 \\
16.472944001889108   -1.1831811616715193 \\
16.213946700231215   -1.1615611227467946 \\
16.387969454210474   -1.1760878328350297 \\
17.059921617831947   -1.232179660201782 \\
17.2150380865237   -1.2451281506667957 \\
16.880929917197363   -1.2172381627029627 \\
17.032438362127355   -1.2298854700526651 \\
17.604025357655182   -1.2775992223314485 \\
};
\addplot [
color=red,
only marks,
mark size=3.0pt,
mark=x,
mark options={solid},
line width=1.5pt,
]
table[row sep=crcr]{
13.332381607423294   -0.9313156265672551 \\
14.072748195350343   -1.0300294427839842 \\
15.561352883072093   -1.0656442024342059 \\
15.008607509043555   -1.1363890575498448 \\
15.737124653341217   -1.109883714182371 \\
16.472944001889108   -1.0961466779492712 \\
16.213946700231215   -1.136096710207725 \\
16.387969454210474   -1.1932626570934897 \\
17.059921617831947   -1.2129107615556576 \\
17.2150380865237   -1.2228249089043661 \\
16.880929917197363   -1.244635961594016 \\
17.032438362127355   -1.2541350131356905 \\
17.604025357655182   -1.2758136090465209 \\
};
\end{axis}
\end{tikzpicture}%
\caption{
The blue line is the fit $0.2 - 0.08 \log(\Vert \psi_{n,k_n^{(3),*}} \Vert)$ and the red line is $\lambda_{k_n^{(3),*}}$. 
}
\end{subfigure}
\hspace*{0.04\textwidth}
\begin{subfigure}[t]{0.3\textwidth}
\centering
\scriptsize
\begin{tikzpicture}
\def\xl{15}
\def\xu{21}
\def\yl{-0.38}
\def\yu{-0.29}

\begin{axis}[%
width=\figurewidth,
height=\figureheight,
scale only axis,
xmin={\xl-(\xu-\xl)/5},
xmax={\xu+(\xu-\xl)/5},
ymin={\yl-(\yu-\yl)/5},
ymax={\yu+(\yu-\yl)*0.05},
hide axis,
axis background/.style={fill=white!100}
]
\draw [->] (axis cs: \xl,\yl) -- (axis cs: {\xu+(\xu-\xl)*0.05},\yl);
\draw [->] (axis cs: \xl,\yl) -- (axis cs: \xl,{\yu+(\yu-\yl)*0.05});
\node at (axis cs: {(\xu+(\xu-\xl)*0.1+\xl)/2},{\yl-0.12*(\yu-\yl)}) {$\log ( \Vert \psi_{n,k_n^{(4),*}} \Vert_{\Lp{\infty}} )$};
\node[rotate=90] at (axis cs: {\xl-0.12*(\xu-\xl)},{(\yu+(\yu-\yl)*0.05+\yl)/2}) {$\log(\lambda_{k_n^{(4),*}})$};
\foreach \yValue in {-0.37,-0.3} {
    \edef\temp{\noexpand\draw [-] ({\xl+(\xu-\xl)/100},\yValue) -- (\xl,\yValue) node[left] {\yValue};} 
    \temp
}
\foreach \xValue in {15.5,21} {
	\edef\temp{\noexpand\draw [-] (\xValue,{\yl+(\yu-\yl)/100}) -- (\xValue,\yl) node[below] {\xValue};}    
    \temp
}
\addplot [
color=blue,
solid,
mark options={solid},
line width=1.5pt,
]
table[row sep=crcr]{
15.653676966598418   -0.3699385869129731 \\
17.04247602435124   -0.35403829774242573 \\
17.85424044287372   -0.3447444485560981 \\
18.4300215584866   -0.3381523600994723 \\
18.876558801819492   -0.33303997810671626 \\
19.24136865849808   -0.32886328877136983 \\
19.5497891118496   -0.32533219882208986 \\
19.816941212717726   -0.3222735881234365 \\
20.05257674897675   -0.3195758089562776 \\
20.26335335051233   -0.31716263845757076 \\
20.45401917551613   -0.31497971520774737 \\
20.62807981653518   -0.31298690385756406 \\
20.788197289586122   -0.3111537271222005 \\
};
\addplot [
color=red,
only marks,
mark size=3.0pt,
mark=x,
mark options={solid},
line width=1.5pt,
]
table[row sep=crcr]{
15.653676966598418   -0.3426881730956601 \\
17.04247602435124   -0.3408879622222072 \\
17.85424044287372   -0.32965835970053553 \\
18.4300215584866   -0.3255109478911788 \\
18.876558801819492   -0.33889839903487484 \\
19.24136865849808   -0.3322643365633724 \\
19.5497891118496   -0.32482889303079693 \\
19.816941212717726   -0.3231200685385129 \\
20.05257674897675   -0.32500547257786294 \\
20.26335335051233   -0.3120709242530029 \\
20.45401917551613   -0.31619724901672647 \\
20.62807981653518   -0.30094861032534664 \\
20.788197289586122   -0.3212933628046467 \\
};
\end{axis}
\end{tikzpicture}%
\caption{
The blue line is the fit $-0.55 + 0.01 \log(\Vert \psi_{n,k_n^{(4),*}} \Vert)$ and the red line is $\lambda_{k_n^{(4),*}}$.  
}
\end{subfigure}

\caption{
Plots of $\log \l \|\psi_{n,k_n^{(i),*}}\|_{\Lp{\infty}} \r$ versus $\log\l \lambda_{k_n^{(i),*}} \r$ for $2 \leq i \leq 4$ with the setting described in Section \ref{subsec:linear}.
\label{fig:linear}	
}
\end{figure}
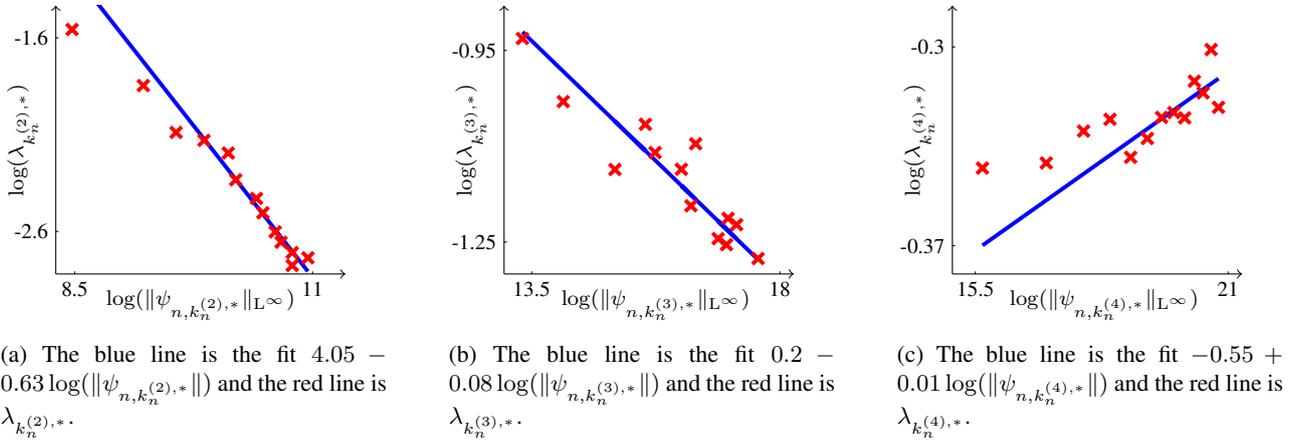

From Figure \ref{fig:eigenvalues} there appears to be different regimes of growth for the eigenpairs depending on the value of $1 \leq k \leq n$. In fact, we see that below $\log(\lambda_{k_n^{(3)}})$, $\log(\Vert \psi_{n,k} \Vert_{\Lp{\infty}})$ seems to be increasing monotonically, while it unexpectedly first decreases and then sharply increases from $\log(\lambda_{k_n^{(3)}})$ to $\log(\lambda_{k_n^{(4)}})$. It is the subject of future research to provide explanations for the latter phenomena. 

Let us define 
\[
k_n^{(i),*} = \argmax_{1 \leq k \leq k_n^{(i)}} \log \l \Vert \psi_{n,k} \Vert_{\Lp{\infty}} \r
\]
and, using the seven largest values of $n$, for $1 \leq i \leq 4$, we will be computing the best linear fits for the points 
\[
\lb( \log \l \lambda_{n,k_n^{(i),*}} \r, \log \l \Vert \psi_{k_n^{(i),*}} \Vert_{\Lp{\infty}} \r\rb_{n=400}^{5200}.
\]

For $k_n^{(i)}$ with $1 \leq i \leq 2$, the linear fits in the log-log domain appear to be very accurate as depicted in Figure \ref{fig:linear}, so we are able to confirm the theoretical guarantees of Proposition \ref{prop:Proofs:Back:DisReg:EVecAlt}.
In particular, we obtain
\[
\Vert \psi_{n,k_n^{(2),*}} \Vert_{\Lp{\infty}} \approx C \lambda_{k_n^{(2),*}}^{-0.63}
\]
and $\lambda_{k_n^{(2),*}}^{-0.63} \leq \lambda_{k_n^{(2),*}}^{d+1}$ since $\lambda_k \geq 1$ (in fact, we have $\{\lambda_k\}_{k=1}^\infty = \{ 4 \pi^2 k^2 \spaceBar k \in \bbN \}$ in this particular setting). However, it shows that the bound is not sharp for the flat torus.

The situation for the regime $k_n^{(3)}$ is different and we obtain that \[
\Vert \psi_{n,k_n^{(3),*}} \Vert_{\Lp{\infty}} \approx C \lambda_{k_n^{(3),*}}^{-0.08}.
\]
The much smaller exponent in the above compared to the $k_n^{(2),*}$ setting indicates that we seem to be switching from one regime of growth of $\Vert \psi_{n,k} \Vert_{\Lp{\infty}}$ to another one.

Finally, we note that the linear fit in Figure \ref{fig:linear} for the $k_n^{(4)}$ regime yields
\[
\Vert \psi_{n,k_n^{(4),*}} \Vert_{\Lp{\infty}} \approx C \lambda_{k_n^{(4),*}}^{0.01}
\]
and, making the crude approximation $0 \approx 0.01$, this confirms our intuition that on the flat torus, the graph Laplacian should have uniformly bounded $\Lp{\infty}$-norms of eigenfunctions: indeed, the continuum eigenfunctions are uniformly bounded in $\Lp{\infty}$, so we expect the same behaviour for their discrete counterparts. This suggests that one should be able to pick $\alpha = 0$ in Theorem \ref{thm:Main:Res:ConsFracLap} yielding the (almost) optimal Sobolev bound $s > d/2 + 2$ as discussed in Remark~\ref{rem:Main:Res:Sobolev}. This conclusion also suggests that, on the flat torus at least, one should be able to go above $\eps_n^{-d/2}$ in Proposition \ref{prop:Proofs:Back:DisReg:EVecAlt}.

\section*{Acknowledgements}

The authors would like to thank Nicol\'as Garc\'ia Trillos for his comments and insights on this paper.
MT was supported by the European Research Council under the European Union’s Horizon 2020 research and innovation programme Grant Agreement No. 777826 (NoMADS).

\bibliography{references}{}
\bibliographystyle{plain}

\end{document}